\documentclass[a4paper,regno]{amsart}
\usepackage[english]{babel}
\usepackage{amsmath}
\usepackage{amsfonts}
\usepackage{amssymb}
\usepackage{mathrsfs}
\usepackage{amsthm}
\usepackage[all]{xy}
\usepackage{graphicx}
\usepackage[usenames]{color}

\newcommand{\R}{{{\mathbb R}}}

\newcommand{\im}{{{\mathfrak{I}}}}

\newcommand{\T}{{{\mathbb T}}}
\newcommand{\OT}{{{\mathcal T}}}

\newcommand{\zz}{{{z_2}}}
\newcommand{\F}{{{\mathcal{F}(z)}}}
\newcommand{\G}{{{\mathcal{F}(h)}}}
\providecommand{\pa}[1]{\partial_{\alpha}^{#1}}
\newcommand{\di}{{{d(z,h)}}}
\providecommand{\abs}[1]{\lvert#1\rvert}
\providecommand{\brz}[1]{BR(\varpi_{1},z)_{#1}}
\providecommand{\brh}[1]{BR(\varpi_{2},h)_{#1}}
\providecommand{\brzep}[1]{BR(\varpi_{1}^{\varepsilon},z^{\varepsilon})_{#1}}
\providecommand{\brhep}[1]{BR(\varpi_{2}^{\varepsilon},h^{\varepsilon})_{#1}}
\providecommand{\norm}[1]{\lVert#1\rVert}
\providecommand{\nor}[1]{\vert\vert\vert#1\vert\vert\vert}

\providecommand{\esth}[1]{\exp C(\norm{\F}^{2}_{L^{\infty}}+\norm{d(z,h)}_{L^{\infty}}^{2}+\norm{z}^{2}_{H^{#1}})}

\providecommand{\estheps}[1]{\exp C(\norm{\mathcal{F}(z^{\epsilon})}^{2}_{L^{\infty}(S)}+\norm{d(z^{\varepsilon},h)}_{L^{\infty}}+\norm{z^{\epsilon}}^{2}_{H^{#1}(S)})}
\theoremstyle{plain}
\newtheorem{lem}{Lemma}[section]
\newtheorem{prop}{Proposition}[section]
\newtheorem{thm}{Theorem}[section]

\theoremstyle{definition}

\title{Local-existence for the Inhomogeneous Muskat problem}
\author{Tania Pernas-Casta\~{n}o}
\date{}

\begin{document}
\maketitle
%\quad \newpage
%\tableofcontents
%\newpage \quad \newpage

\begin{abstract}
In this work we study the evolution of the interface between two different fluids in a porous media with two different permeabilities. We prove local existence in Sobolev spaces, when the free boundary is given by the discontinuity among the densities and viscosities of the fluids. 
\end{abstract}
\section{Introduction}
In this paper, we study the evolution of the interface between two different incompressible fluids with different viscosities and densities, in a porous medium where the permeability is a two dimensional step function. The velocity of a fluid in a porous media is given by the Darcy's law:
\begin{displaymath}
	\frac{\mu}{\kappa}u=-\nabla p-(0,g\rho)
	\end{displaymath}
	where $(x,t)\in\R^{2}\times\R^{+}$, $u=(u_{1}(x,t),u_{2}(x,t))$ is the incompressible velocity (i.e. $\nabla\cdot u=0$), $p=p(x,t)$ is the pressure, $\mu=\mu(x)$ is the dynamic viscosity, $\kappa=\kappa(x)$ is the permeability of the isotropic medium, $\rho=\rho(x)$ is the liquid density and $g$ is the acceleration due to gravity. The free boundary is caused by the discontinuity between the densities and viscosities of the fluids; the quantities $(\mu,\rho)$ are defined by
	\begin{displaymath}
(\mu,\rho)(x_{1},x_{2}) := \left\{ \begin{array}{ll}
(\mu^{1},\rho^{1}) & \textrm{$x\in\Omega_{1}(t)$}\\
(\mu^{2},\rho^{2}) & \textrm{$x\in\Omega_{2}(t)\cup\Omega_{3}=\R^{2}-\Omega^{1}(t)$}
\end{array} \right.
\end{displaymath}
where $\mu^{1}$, $\rho^{1}$, $\mu^{2}$ and $\rho^{2}$ are constants. And, in this work we study the case where the permeability $\kappa(x)$ is a step function separating two regions with different values of the permeability:
\begin{displaymath}
\kappa(x_{1},x_{2}) := \left\{ \begin{array}{ll}
\kappa^{1} & \textrm{$x\in \Omega_{1}(t)\cup\Omega_{2}(t)=\R^{2}-\Omega_{3}$}\\
\kappa^{2} & \textrm{$x\in \Omega_{3}$}
\end{array} \right.
\end{displaymath}

\begin{figure}[htb]
\centering
\includegraphics[width=62mm]{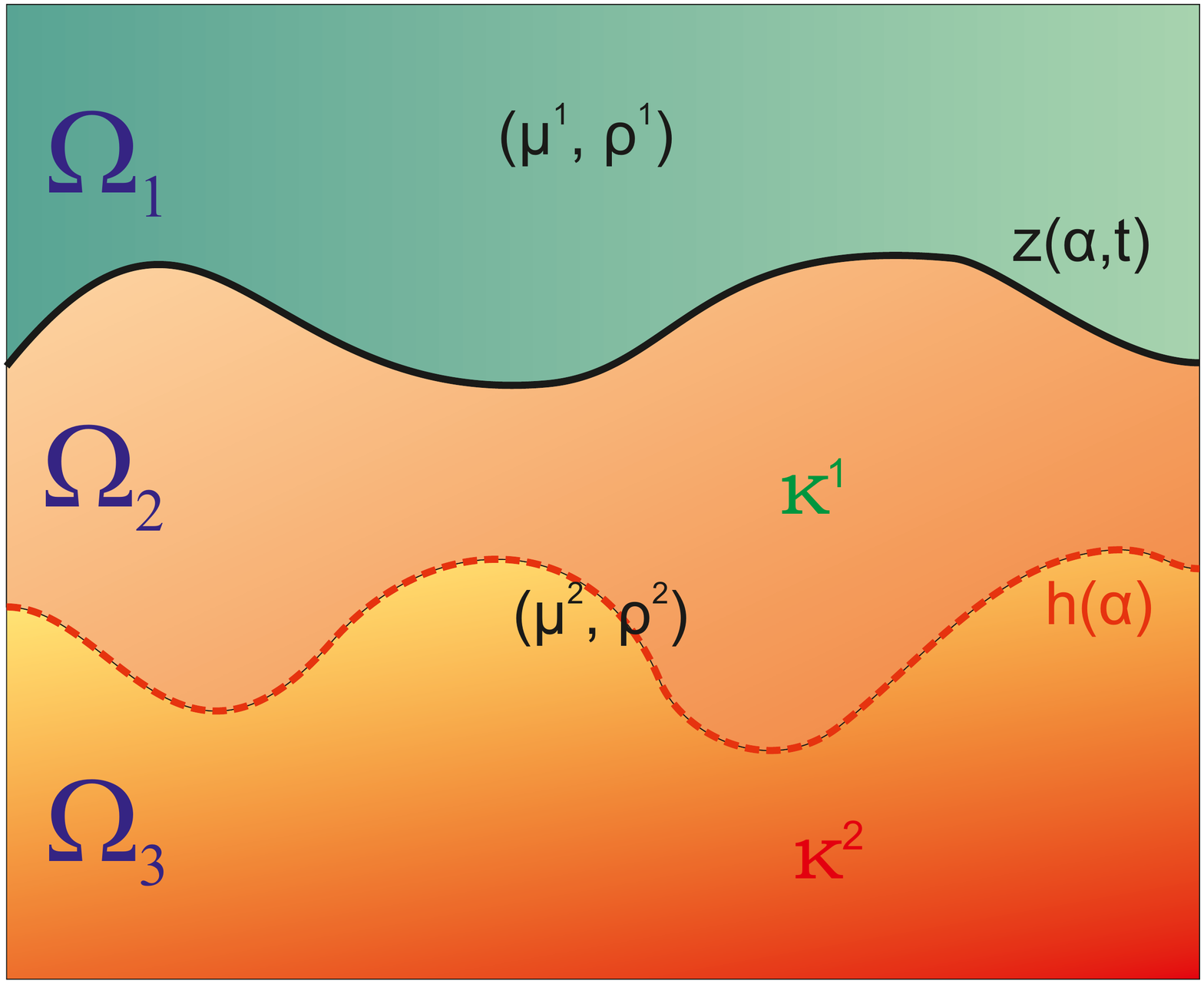}
\caption{Inhomogeneous Muskat Problem} \label{fig:inhmuskat}
\end{figure}

We parametrize the interface between two fluids by a curve
\begin{displaymath}
z(\alpha,t)=\{(z_{1}(\alpha,t),z_{2}(\alpha,t)):\alpha\in\R\}
\end{displaymath}
and 
\begin{displaymath}
h(\alpha)=\{(h_{1}(\alpha),h_{2}(\alpha)):\alpha\in\R\}
\end{displaymath}
for the curve, fixed on time, which separates two different regions with different permeability (see Figure \ref{fig:inhmuskat}). We are going to consider that these two curves don't touch each other initially.

Using Darcy's law we can see that the vorticity is nule inside the different regions, so we can consider in distributions sense:
\begin{displaymath}
\omega(x,t)=\varpi_{1}(\alpha,t)\delta(x-z(\alpha,t))+\varpi_{2}(\alpha,t)\delta(x-h(\alpha)).
\end{displaymath}
Using the Biot-Savart law we know that
\begin{align*}
u(x,t)&=\frac{1}{2\pi}PV\int_{\R}\frac{(x-z(\beta,t))^{\bot}}{\abs{x-z(\beta,t)}^{2}}\varpi_{1}(\beta,t)d\beta+\frac{1}{2\pi}PV\int_{\R}\frac{(x-h(\beta))^{\bot}}{\abs{x-h(\beta)}^{2}}\varpi_{2}(\beta,t)d\beta\\
&\equiv\brz{x}+\brh{x}.
\end{align*}
If we calculate directional limits in the normal direction of $z(\alpha,t)$ and $h(\alpha)$:
\begin{align*}
&u^{\pm}(z(\alpha,t),t)=\brz{z}(\alpha,t)+\brh{z}(\alpha,t)\mp\frac{1}{2}\frac{\varpi_{1}(\alpha,t)}{\abs{\partial_{\alpha}z(\alpha,t)}^{2}}\partial_{\alpha}z(\alpha,t)\\
&u^{\pm}(h(\alpha),t)=\brz{h}(\alpha,t)+\brh{h}(\alpha,t)\mp\frac{1}{2}\frac{\varpi_{2}(\alpha,t)}{\abs{\partial_{\alpha}h(\alpha)}^{2}}\partial_{\alpha}h(\alpha)
\end{align*}
Since $p^{+}(z(\alpha,t),t)=p^{-}(z(\alpha,t),t)$ (see \cite{hele}) we have,
\begin{displaymath}
\frac{\mu^{2}}{\kappa^{1}}u^{+}(z(\alpha,t),t)-\frac{\mu^{1}}{\kappa^{1}}u^{-}(z(\alpha,t),t)\cdot\partial_{\alpha}z(\alpha,t)=-g(\rho^{2}-\rho^{1})\partial_{\alpha}z_{2}(\alpha,t)
\end{displaymath}
Using the above limits,
\begin{align*}
&\frac{\mu^{2}}{\kappa^{1}}u^{+}(z(\alpha,t),t)-\frac{\mu^{1}}{\kappa^{1}}u^{-}(z(\alpha,t),t)\cdot\partial_{\alpha}z(\alpha,t)=\\
&=\frac{\mu^{2}-\mu^{1}}{\kappa^{1}}(\brz{z}+\brh{z})\cdot\partial_{\alpha}z(\alpha,t)+\frac{\mu^{2}-\mu^{1}}{2\kappa^{1}}\varpi_{1}(\alpha,t)
\end{align*}
Hence,
\begin{equation}
\label{vortici1}
\varpi_{1}(\alpha,t)=-2\frac{\mu^{2}-\mu^{1}}{\mu^{2}+\mu^{1}}(\brz{z}+\brh{z})\cdot\partial_{\alpha}z(\alpha,t)-2\kappa^{1}\frac{\rho^{2}-\rho^{1}}{\mu^{2}+\mu^{1}}g\partial_{\alpha}z_{2}(\alpha,t).
\end{equation}

For $\varpi_{2}(\alpha,t)$ we proceed in the same way. Since $p^{+}(h(\alpha),t)=p^{-}(h(\alpha),t)$ then,
\begin{displaymath}
\mu^{2}(\frac{u^{-}(h(\alpha),t)}{\kappa^{2}}-\frac{u^{+}(h(\alpha),t)}{\kappa^{1}})\cdot\partial_{\alpha}h(\alpha)=-\partial_{\alpha}(p^{-}(h(\alpha),t)-p^{+}(h(\alpha),t))=0
\end{displaymath}
and
\begin{align*}
&\mu^{2}(\frac{u^{-}(h(\alpha),t)}{\kappa^{2}}-\frac{u^{+}(h(\alpha),t)}{\kappa^{1}})\cdot\partial_{\alpha}h(\alpha)=\\
&=\frac{\mu^{2}}{\kappa^{2}-\kappa^{1}}(\brz{h}+\brh{h})\cdot\partial_{\alpha}h(\alpha)+\frac{\mu^{2}}{2(\kappa^{2}-\kappa^{1})}\varpi_{2}(\alpha,t).
\end{align*}
Therefore,
\begin{equation}
\label{vortici2}
\varpi_{2}(\alpha,t)=-2\frac{\kappa^{1}-\kappa^{2}}{\kappa^{2}+\kappa^{1}}(\brz{h}+\brh{h})\cdot\partial_{\alpha}h(\alpha).
\end{equation}

We consider our curve with the following periodic conditions,
\begin{displaymath}
(z_{1}(\alpha+2k\pi,t),z_{2}(\alpha+2k\pi,t))=(z_{1}(\alpha,t)+2k\pi,z_{2}(\alpha,t))
\end{displaymath}
with $z(\alpha,0)=z_{0}(\alpha)$.
We will add tangential term in order to get $\abs{\partial_{\alpha}z(\alpha,t)}^{2}\equiv A(t)$. For that we take,
\begin{align*}
c(\alpha,t)=&\frac{\alpha+\pi}{2\pi A(t)}\int_{\T}\partial_{\alpha}z(\beta,t)\cdot\partial_{\alpha}(\brz{z}+\brh{z})d\beta\\
&-\int_{-\pi}^{\alpha}\frac{\partial_{\alpha}z(\beta,t)}{A(t)}\cdot\partial_{\alpha}(\brz{z}+\brh{z})d\beta
\end{align*}
where $\T=[-\pi,\pi]$.
Finally, the system which describes our problem is:
\begin{equation*}
(P)\left\{ \begin{array}{ll}
z_{t}(\alpha,t)=\brz{z}(\alpha,t)+\brh{z}(\alpha,t)+c(\alpha,t)\partial_{\alpha}z(\alpha,t)\\
c(\alpha,t)=\frac{\alpha+\pi}{2\pi A(t)}\int_{\T}\partial_{\alpha}z(\beta,t)\cdot\partial_{\alpha}(\brz{z}+\brh{z})d\beta\\
-\int_{-\pi}^{\alpha}\frac{\partial_{\alpha}z(\beta,t)}{A(t)}\cdot\partial_{\alpha}(\brz{z}+\brh{z})d\beta\\
\varpi_{1}(\alpha,t)=-2\frac{\mu^{2}-\mu^{1}}{\mu^{2}+\mu^{1}}(\brz{z}+\brh{z})\cdot\partial_{\alpha}z(\alpha,t)\\
-2\kappa^{1}\frac{\rho^{2}-\rho^{1}}{\mu^{2}+\mu^{1}}g\partial_{\alpha}z_{2}(\alpha,t)\\
\varpi_{2}(\alpha,t)=-2\frac{\kappa^{1}-\kappa^{2}}{\kappa^{2}+\kappa^{1}}(\brz{h}+\brh{h})\cdot\partial_{\alpha}h(\alpha)
\end{array} \right.
\end{equation*}

For the stability of the problem we consider the Rayleigh-Taylor condition. Rayleigh \cite{ray} and Saffman-Taylor \cite{tay} gave this condition which must be satisfied for the linearized model in order to have a solution locally in time, namely that the normal component of the pressure gradient jump at the interface has to have a distinguished sign. This condition can be written as
\begin{displaymath}
\sigma(\alpha,t)=\frac{\mu^{2}-\mu^{1}}{\kappa^{1}}(\brz{z}+\brh{z})\cdot\pa{\bot}z(\alpha)+(\rho^{2}-\rho^{1})g\pa{}z_{1}(\alpha)>0.
\end{displaymath}

 Using Hopf's lemma, the Rayleigh-Taylor condition is satisfied for $\mu^{1}=\rho^{1}=0$ (see \cite{splash}). For the case of equal viscosities ($\mu^{1}=\mu^{2}$), this condition holds when the more dense fluid lies below the interface \cite{contour}.
 
 We will focus on the existence of classical solution locally in time in Sobolev spaces for the stable regime. There is a vast literature about these problems. In the case with $\kappa(x)\equiv\textbf{constant}$, this stability has been used to prove local existence in Sobolev spaces, when $\mu^{1}\neq\mu^{2}$ and $\rho^{1}\neq\rho^{2}$, in \cite{hele}. Taking the initial data on $H^2$ with $\mu^1=\rho^1=0$, local existence has been proved on \cite{granerh2}.  When $\mu^{1}=\mu^{2}$ local existence and instant analyticity in the stable case are available, see \cite{raytay} and \cite{contour}. For small data, the fact that $\sigma>0$ has been used to prove global existence as we can check in \cite{global}, \cite{global2}, \cite{para}, \cite{Granero} and \cite{onglobal2d3d}. For improvements for local and global existence for the case with finite slope, see \cite{globalregul}.
 
 For the case where $\kappa(x)$ is a step function, the authors in \cite{localsolv} prove local existence in Sobolev spaces by means of energy methods when the system is in the stable regime and the two fluids have the same viscosities ($\mu^{1}=\mu^{2}$). In this scenario, moreover, they consider $h(\alpha)=(\alpha,-h_{2})$, with $h_{2}>0$. Then the formula for the strength of the vorticities are simpler
\begin{align*}
&\varpi_{1}(\alpha,t)=-(\rho^{2}-\rho^{1})\partial_{\alpha}\zz(\alpha,t),\\
&\varpi_{2}(\alpha,t)=\frac{\kappa^{1}-\kappa^{2}}{\kappa^{1}+\kappa^{2}}\frac{\kappa^{1}(\rho^{2}-\rho^{1})}{\pi}PV\int_{\R}\frac{h_{2}+z_{2}(\alpha,t)}{\abs{h(\alpha)-z(\beta)}^{2}}\pa{}z_{2}(\beta,t)d\beta.
\end{align*} 
 In our case, with different viscosities, the expressions (\ref{vortici1}) and (\ref{vortici2}) involves the Birkhoff-Rott integrals, so we find a delicate issue, we need to invert an operator.
 
 Finally we introduce the functions that measures the arc-chord condition of the curves and the distance beetween both:
 \begin{displaymath}
\F(\alpha,\beta,t)=\frac{\beta^{2}}{\abs{z(\alpha)-z(\alpha-\beta)}^{2}},\quad\alpha,\beta\in\R
\end{displaymath}
with 
\begin{displaymath}
\F(\alpha,0,t)=\frac{1}{\abs{\partial_{\alpha}z(\alpha,t)}^{2}};
\end{displaymath}
and
\begin{displaymath}
d(z,h)=\frac{1}{\abs{z(\alpha,t)-h(\beta)}^{2}},\quad\alpha,\beta\in\R.
\end{displaymath}

The main theorem of this paper is the following:
\begin{thm}
\label{main}
Let $z_{0}(\alpha)\in H^{k}(\T)$ for $k\ge 3$, $h(\alpha)\in H^{k}(\T)$, $\mathcal{F}(z_{0})(\alpha,\beta)\in L^{\infty}$, $\G(\alpha,\beta)\in L^{\infty}$ and $d(z_{0},h)\in L^{\infty}$. Then, if the Rayleigh-Taylor condition is satisfied, there exists a classical solution of the Muskat problem $(P)$, $z\in\mathcal{C}^{1}([0,T],H^{k}(\T))$ where $T=T(z_{0})$.  
\end{thm}

In order to prove this theorem we have organized the paper as follows.

We devote section \ref{invertiroperador} to the study of the operator 
\begin{displaymath}
\mathcal{T}(u_{1},u_{2})(\alpha) = \;
   \begin{pmatrix}
      T_{1} & T_{2} \\
      T_{3} & T_{4} 
   \end{pmatrix}
   \begin{pmatrix}
   u_{1}\\
   u_{2}
   \end{pmatrix}
\end{displaymath}
where 
\begin{align*}
&T_{1}(u)(\alpha)=2BR(u,z)_{z}(\alpha)\cdot\partial_{\alpha}z(\alpha),\\
&T_{2}(u)(\alpha)=2BR(u,h)_{z}(\alpha)\cdot\partial_{\alpha}z(\alpha),\\
&T_{3}(u)(\alpha)=2BR(u,z)_{h}(\alpha)\cdot\partial_{\alpha}h(\alpha),\\
&T_{4}(u)(\alpha)=2BR(u,h)_{h}(\alpha)\cdot\partial_{\alpha}h(\alpha);
\end{align*}

In this section we want to estimate the $H^{\frac{1}{2}}$-norm of the $(I+M\OT)^{-1}$ for $M=\begin{pmatrix}
\mu_{1}&0\\
0&\mu_{2}
\end{pmatrix}$ with $\abs{\mu_{i}}\le 1$ for $i=1,2$. The main obstacle in this problem is to found some operators in order to estimate the $L^{2}$-norm of the inverse operator.

Once we have the above estimates, we are qualified to estimate the strength of the vorticities and the Birkhoff-Rott integrals. These estimations can be found in sections \ref{estiamplitud} and \ref{estibr}.

For the purpose of study the local existence of classical solutions in Sobolev Spaces, we will use energy methods. For that we will need to obtain several a priori estimates for the curve $z(\alpha,t)$ with regularity $H^{k}$ for $k\ge 3$. We present these technical computations in section \ref{estimacionesz}.

The other tools which we will need to show the estimates of the evolution of our energy, are the study of the evolution of the distance between $z$ and $h$ and the evolution of the minimum of the R-T condition. Sections \ref{evoldi} and \ref{evolsig} are dedicated to that.

Finally, after all these computations, in section \ref{regu} we follow the classical procedure and show the main theorem \ref{main}.
\section{Inverse Operator}\label{invertiroperador}

\subsection{The basic operator}
Let us consider the operator 
\begin{displaymath}
\OT(u_{1},u_{2})(\alpha) = \;
   \begin{pmatrix}
      T_{1} & T_{2} \\
      T_{3} & T_{4} 
   \end{pmatrix}
   \begin{pmatrix}
   u_{1}\\
   u_{2}
   \end{pmatrix}
\end{displaymath}
where 
\begin{align*}
&T_{1}(u)(\alpha)=2BR(u,z)_{z}(\alpha)\cdot\partial_{\alpha}z(\alpha)\\
&T_{2}(u)(\alpha)=2BR(u,h)_{z}(\alpha)\cdot\partial_{\alpha}z(\alpha)\\
&T_{3}(u)(\alpha)=2BR(u,z)_{h}(\alpha)\cdot\partial_{\alpha}h(\alpha)\\
&T_{4}(u)(\alpha)=2BR(u,h)_{h}(\alpha)\cdot\partial_{\alpha}h(\alpha)
\end{align*}
\begin{lem}
\label{tl2h1}
Suppose that $\norm{\F}_{L^{\infty}}<\infty$, $\norm{\G}_{L^{\infty}}<\infty$, $\norm{d(z,h)}_{L^{\infty}}<\infty$ and $z\in\mathcal{C}^{2,\delta}$, $h\in\mathcal{C}^{2,\delta}$. Then $\OT:L^{2}\times L^{2}\to H^{1}\times H^{1}$ is compact and
\begin{displaymath}
\norm{\OT}_{L^{2}\times L^{2}\to H^{1}\times H^{1}}\le C\norm{\F}^{2}_{L^{\infty}}\norm{d(z,h)}^{2}_{L^{\infty}}\norm{z}^{4}_{\mathcal{C}^{2,\delta}}
\end{displaymath}

\end{lem}
\begin{proof}
We have 
\begin{displaymath}
\OT(w)(\alpha) = \;
   \begin{pmatrix}
      T_{1} & T_{2} \\
      T_{3} & T_{4} 
   \end{pmatrix}
   \begin{pmatrix}
   u\\
   v
   \end{pmatrix}=\begin{pmatrix}
   T_{1}(u)+T_{2}(v)\\
   T_{3}(u)+T_{4}(v)
   \end{pmatrix}
\end{displaymath}
and we consider $\norm{(u,v)}_{L^{2}}=\norm{u}_{L^{2}}+\norm{v}_{L^{2}},$
then
\begin{displaymath}
\norm{\OT(w)}_{L^{2}}=\norm{T_{1}(u)+T_{2}(v)}_{L^{2}}+\norm{T_{3}(u)+T_{4}(v)}_{L^{2}}
\end{displaymath}

We want to estimate $\norm{\partial_{\alpha}\OT(w)}_{L^{2}}$. Since
\begin{align*}
&\norm{\partial_{\alpha}T_{1}(u)+\partial_{\alpha}T_{2}(v)}_{L^{2}}\le\norm{\partial_{\alpha}T_{1}(u)}_{L^{2}}+\norm{\partial_{\alpha}T_{2}(v)}_{L^{2}},\\
&\norm{\partial_{\alpha}T_{3}(u)+\partial_{\alpha}T_{4}(v)}_{L^{2}}\le\norm{\partial_{\alpha}T_{3}(u)}_{L^{2}}+\norm{\partial_{\alpha}T_{4}(v)}_{L^{2}},
\end{align*}
it is enough to estimate each $T_{i}$ for $i=1,2,3,4$ separately.

Operator $T_{1}$ and $T_{4}$ are exactly the same as the operator $T$ on \cite{hele}. Therefore, by lemma 3.1 on \cite{hele} we have:
\begin{align*}
&\norm{\partial_{\alpha}T_{1}(u)}_{L^{2}}\le C\norm{\F}^{2}_{L^{\infty}}\norm{z}^{4}_{\mathcal{C}^{2,\delta}}\norm{u}_{L^{2}},\\
&\norm{\partial_{\alpha}T_{4}(v)}_{L^{2}}\le C\norm{\G}^{2}_{L^{\infty}}\norm{h}^{4}_{\mathcal{C}^{2,\delta}}\norm{v}_{L^{2}}.
\end{align*}
Let us estimate operator $T_{2}$ and $T_{3}$.
We write first,
\begin{align*}
\partial_{\alpha}T_{2}(v)&=\frac{1}{\pi}PV\int_{\R}\pa{} (\frac{(z(\alpha)-h(\alpha-\beta))^{\bot}\cdot\pa{} z(\alpha)}{\abs{z(\alpha)-h(\alpha-\beta)}^{2}})v(\alpha-\beta)d\beta\\
&+\frac{1}{\pi}PV\int_{\R}\frac{(z(\alpha)-h(\alpha-\beta))^{\bot}\cdot\pa{} z(\alpha)}{\abs{z(\alpha)-h(\alpha-\beta)}^{2}}\pa{} v(\alpha-\beta)d\beta\equiv I_{1}+I_{2}.
\end{align*}
Then, using $\Delta zh=z(\alpha)-h(\alpha-\beta)$ in order to reduce notation,
\begin{align*}
&I_{1}=\frac{1}{\pi}PV\int_{\R}\frac{(\pa{} z(\alpha)-\pa{} h(\alpha-\beta))^{\bot}\cdot\pa{} z(\alpha)}{\abs{z(\alpha)-h(\alpha-\beta)}^{2}}v(\alpha-\beta)d\beta\\
&+\frac{1}{\pi}PV\int_{\R}\frac{(z(\alpha)-h(\alpha-\beta))^{\bot}\cdot\partial^{2}_{\alpha} z(\alpha)}{\abs{z(\alpha)-h(\alpha-\beta)}^{2}}v(\alpha-\beta)d\beta\\
&-\frac{2}{\pi}PV\int_{\R}\frac{(\Delta zh)^{\bot}\cdot\pa{} z(\alpha)(\Delta zh)\cdot\pa{}\Delta zh}{\abs{z(\alpha)-h(\alpha-\beta)}^{4}}v(\alpha-\beta)d\beta\\
&\equiv I_{1}^{1}+I_{1}^{2}+I_{1}^{3}.
\end{align*} 

Since $\pa{} z(\alpha)\cdot\pa{} z(\alpha)^{\bot}=0$ we have,
\begin{displaymath}
I_{1}^{1}=-\frac{1}{\pi}PV\int_{\R}\frac{\pa{} h(\alpha-\beta)\cdot\pa{} z(\alpha)}{\abs{z(\alpha)-h(\alpha-\beta)}^{2}}v(\alpha-\beta)d\beta\le C\norm{d(z,h)}_{L^{\infty}}\norm{\pa{} z}_{L^{\infty}}\norm{h}_{H^{1}}\norm{v}_{L^{2}}.
\end{displaymath}
Using the Cauchy inequality it is easy to get $u\cdot v\le\frac{\abs{u}^{2}}{2}+\frac{\abs{v}^{2}}{2}$, then
\begin{align*}
I_{1}^{2}&\le\frac{1}{2\pi}PV\int_{\R}v(\alpha-\beta)d\beta+\frac{1}{2\pi}PV\int_{\R}\frac{\abs{\partial^{2}_{\alpha}z(\alpha)}^{2}}{\abs{z(\alpha)-h(\alpha-\beta)}^{2}}v(\alpha-\beta)d\beta\\
&\le C\norm{v}_{L^{2}}+C\norm{d(z,h)}_{L^{\infty}}\norm{z}_{\mathcal{C}^{2}}^{2}\norm{v}_{L^{2}}
\end{align*}

\begin{align*}
&\abs{I_{1}^{3}}\le\frac{2}{\pi}PV\int_{\R}\frac{\abs{z(\alpha)-h(\alpha-\beta)}^{2}\abs{\pa{} z(\alpha)}\abs{\pa{} z(\alpha)-\pa{} h(\alpha-\beta)}}{\abs{z(\alpha)-h(\alpha-\beta)}^{4}}\abs{v(\alpha-\beta)}d\beta\\
&\le C\norm{d(z,h)}_{L^{\infty}}\norm{z}_{\mathcal{C}^{1}}(\norm{z}_{\mathcal{C}^{1}}+\norm{h}_{\mathcal{C}^{1}})\norm{v}_{L^{2}}
\end{align*}

On the other hand, using integration by parts
\begin{align*}
I_{2}&=\frac{1}{\pi}PV\int_{\R}\partial_{\beta}(\frac{(z(\alpha)-h(\alpha-\beta))^{\bot}\cdot\pa{} z(\alpha)}{\abs{z(\alpha)-h(\alpha-\beta)}^{2}})v(\alpha-\beta)d\beta\\
&=\frac{1}{\pi}PV\int_{\R}\frac{(\pa{} h(\alpha-\beta))^{\bot}\cdot\pa{} z(\alpha)}{\abs{z(\alpha)-h(\alpha-\beta)}^{2}}v(\alpha-\beta)d\beta\\
&-\frac{2}{\pi}PV\int_{\R}\frac{(z(\alpha)-h(\alpha-\beta))^{\bot}\cdot\pa{} z(\alpha)(z(\alpha)-h(\alpha-\beta))\cdot\pa{} h(\alpha-\beta)}{\abs{z(\alpha)-h(\alpha-\beta)}^{4}}v(\alpha-\beta)d\beta\\
&\le C\norm{d(z,h)}_{L^{\infty}}\norm{\pa{} z}_{L^{\infty}}\norm{h}_{\mathcal{C}^{1}}\norm{v}_{L^{2}}+C\norm{d(z,h)}_{L^{\infty}}\norm{z}_{\mathcal{C}^{1}}\norm{h}_{\mathcal{C}^{1}}\norm{v}_{L^{2}}.
\end{align*}

Then,

\begin{displaymath}
\norm{\pa{} T_{2}(v)}_{L^{2}}\le C\norm{d(z,h)}_{L^{\infty}}\norm{z}^{2}_{\mathcal{C}^{2}}\norm{h}_{\mathcal{C}^{1}}\norm{v}_{L^{2}}.
\end{displaymath} 

Finally, we have to estimate $\pa{} T_{3}$. We have,
\begin{align*}
\pa{} T_{3}(u)(\alpha)&=\frac{1}{\pi}PV\int_{\R}\pa{} (\frac{(h(\alpha)-z(\alpha-\beta))^{\bot}\cdot\pa{} h(\alpha)}{\abs{h(\alpha)-z(\alpha-\beta)}^{2}})u(\alpha-\beta)d\beta\\
&+\frac{1}{\pi}PV\int_{\R}\frac{(h(\alpha)-z(\alpha-\beta))^{\bot}\cdot\pa{} h(\alpha)}{\abs{h(\alpha)-z(\alpha-\beta)}^{2}}\pa{} u(\alpha-\beta)d\beta.
\end{align*}
Changing $z$ for $h$, we can check that we have the same estimates as in $T_{2}$. Thus,
\begin{displaymath}
\norm{T_{3}(u)}_{L^{2}}\le C\norm{d(z,h)}_{L^{\infty}}\norm{h}^{2}_{\mathcal{C}^{2}}\norm{z}_{\mathcal{C}^{1}}\norm{u}_{L^{2}}.
\end{displaymath}

Therefore, 
\begin{displaymath}
\norm{\pa{}\OT(u,v)}_{L^{2}}\le C\norm{\F}^{2}_{L^{\infty}}\norm{\G}^{2}_{L^{\infty}}\norm{d(z,h)}^{2}_{L^{\infty}}\norm{h}^{4}_{\mathcal{C}^{2,\delta}}\norm{z}^{4}_{\mathcal{C}^{2,\delta}}\norm{(u,v)}_{L^{2}}.
\end{displaymath}
Since $h$ is fixed on time, $\norm{\G}_{L^{\infty}}^{2}$ and $\norm{h}^{4}_{\mathcal{C}^{2,\delta}}$ are not dependent of time. Thus we get,
\begin{displaymath}
\norm{\pa{}\OT(w)}_{L^{2}}\le C\norm{\F}^{2}_{L^{\infty}}\norm{d(z,h)}^{2}_{L^{\infty}}\norm{z}^{4}_{\mathcal{C}^{2,\delta}}\norm{(w)}_{L^{2}}.
\end{displaymath}
\end{proof}

%%%%%%%%%%%%%%%%%%%%%%%%%%%%%%%%%%%%%%%%%%%%%%%%%%%%%%%%%

\subsection{Estimates on the inverse operator}\label{estiminversesect}

We are going to work with the adjoint operator of $\OT$ in order to estimate the inverse operator $(I+M\OT)^{-1}$.

We have,
\begin{align*}
&\left( \begin{pmatrix}
   u_{1}\\
   u_{2}
   \end{pmatrix},
   \begin{pmatrix}
      T_{1} & T_{2} \\
      T_{3} & T_{4} 
   \end{pmatrix}
   \begin{pmatrix}
   w_{1}\\
   w_{2}
   \end{pmatrix}
   \right) =\left( 
   \begin{pmatrix}
   u_{1}\\
   u_{2}
   \end{pmatrix},
   \begin{pmatrix}
      T_{1}(w_{1})+ T_{2}(w_{2}) \\
      T_{3}(w_{1})+ T_{4}(w_{2}) 
   \end{pmatrix}\right)\\
   &=(T_{1}(w_{1}),u_{1})+(T_{2}(w_{2}),u_{1})+(T_{3}(w_{1}),u_{2})+(T_{4}(w_{2}),u_{2})\\
   &=(w_{1},T^{*}_{1}(u_{1}))+(w_{2},T^{*}_{2}(u_{1}))+(w_{1},T^{*}_{3}(u_{2}))+(w_{2},T^{*}_{4}(u_{2}))=\\
   &\left( \begin{pmatrix}
   w_{1}\\
   w_{2}
   \end{pmatrix},
   \begin{pmatrix}
      T^{*}_{1}(u_{1}) + T^{*}_{3}(u_{2}) \\
      T^{*}_{2}(u_{1}) + T^{*}_{4}(u_{2}) 
   \end{pmatrix}
   \right) =\left( 
   \begin{pmatrix}
   w_{1}\\
   w_{2}
   \end{pmatrix},
   \begin{pmatrix}
      T^{*}_{1} & T^{*}_{3} \\
      T^{*}_{2} & T^{*}_{4} 
   \end{pmatrix}
   \begin{pmatrix}
      u_{1} \\
      u_{2} 
   \end{pmatrix}\right)
\end{align*}
The adjoint operator is given by
\begin{displaymath}
\OT^{*}(u_{1},u_{2})(\alpha) = \;
   \begin{pmatrix}
      T^{*}_{1} & T^{*}_{3} \\
      T^{*}_{2} & T^{*}_{4} 
   \end{pmatrix}
   \begin{pmatrix}
   u_{1}\\
   u_{2}
   \end{pmatrix}
\end{displaymath}

 where we can compute:
\begin{displaymath}
T_{1}^{*}(u)(\alpha)=-\frac{1}{\pi}PV\int_{\R}\frac{(z(\alpha)-z(\beta))^{\bot}\cdot\partial_{\alpha}z(\beta)}{\abs{z(\alpha)-z(\beta)}^{2}}u(\beta)d\beta,
\end{displaymath}
\begin{displaymath}
T_{2}^{*}(u)(\alpha)=-\frac{1}{\pi}PV\int_{\R}\frac{(h(\alpha)-z(\beta))^{\bot}\cdot\partial_{\alpha}z(\beta)}{\abs{h(\alpha)-z(\beta)}^{2}}u(\beta)d\beta,
\end{displaymath}
\begin{displaymath}
T_{3}^{*}(u)(\alpha)=-\frac{1}{\pi}PV\int_{\R}\frac{(z(\alpha)-h(\beta))^{\bot}\cdot\partial_{\alpha}h(\beta)}{\abs{z(\alpha)-h(\beta)}^{2}}u(\beta)d\beta,
\end{displaymath}
and
\begin{displaymath}
T_{4}^{*}(u)(\alpha)=-\frac{1}{\pi}PV\int_{\R}\frac{(h(\alpha)-h(\beta))^{\bot}\cdot\partial_{\alpha}h(\beta)}{\abs{h(\alpha)-h(\beta)}^{2}}u(\beta)d\beta.
\end{displaymath}
\begin{prop}\label{estimaT}
Suppose that $\norm{\F}_{L^{\infty}}<\infty$, $\norm{\G}_{L^{\infty}}<\infty$,\\ $\norm{d(z,h)}_{L^{\infty}}<\infty$ and $z,h\in\mathcal{C}^{2,\delta}$. Then $\OT^{*}:L^{2}\times L^{2}\to H^{1}\times H^{1}$ and
\begin{displaymath}
\norm{\OT^{*}}_{L^{2}\times L^{2}\to H^{1}\times H^{1}}\le\norm{\F}_{L^{\infty}}^{2}\norm{d(z,h)}_{L^{\infty}}^{2}\norm{z}^{2}_{\mathcal{C}^{2,\delta}}.
\end{displaymath}
\end{prop}
\begin{proof}
In the same way as in the study of $\OT$, we can prove this estimate studying each $T_{i}^{*}$.
\begin{displaymath}
T_{1}^{*}(u)=-\frac{1}{\pi}PV\int_{\R}\frac{(z(\alpha)-z(\beta))^{\bot}\cdot\partial_{\alpha}z(\beta)}{\abs{z(\alpha)-z(\beta)}^{2}}u(\beta)d\beta
\end{displaymath}
then
\begin{align*}
\partial_{\alpha}T_{1}^{*}(u)=&-\frac{1}{\pi}PV\int_{\R}\partial_{\alpha}(\frac{(\Delta z)^{\bot}\cdot\partial_{\alpha}z(\alpha-\beta)}{\abs{\Delta z}^{2}})u(\alpha-\beta)d\beta\\
&-\frac{1}{\pi}PV\int_{\R}\frac{(\Delta z)^{\bot}\cdot\partial_{\alpha}z(\alpha-\beta)}{\abs{\Delta z}^{2}}\partial_{\alpha}u(\alpha-\beta)d\beta\equiv I_{1}+I_{2}.
\end{align*}
$I_{1}$ is estimated in the same way that operator $T_{1}$. Using integration by parts
\begin{align*}
I_{2}&=\frac{1}{\pi}PV\int_{\R}\frac{(\Delta z)^{\bot}\cdot\partial_{\alpha}z(\alpha-\beta)}{\abs{\Delta z}^{2}}\partial_{\beta}u(\alpha-\beta)d\beta\\
&=-\frac{1}{\pi}PV\int_{\R}\partial_{\beta}(\frac{(\Delta z)^{\bot}\cdot\partial_{\alpha}z(\alpha-\beta)}{\abs{\Delta z}^{2}})u(\alpha-\beta)d\beta\\
&=-\frac{1}{\pi}PV\int_{\R}\frac{(\partial^{\bot}_{\alpha} z(\alpha-\beta)\cdot\partial_{\alpha}z(\alpha-\beta)}{\abs{\Delta z}^{2}}u(\alpha-\beta)d\beta\\
&-\frac{1}{\pi}PV\int_{\R}\frac{(\Delta z)^{\bot}\cdot\partial^{2}_{\alpha}z(\alpha-\beta)}{\abs{\Delta z}^{2}}u(\alpha-\beta)d\beta\\
&+\frac{2}{\pi}PV\int_{\R}\frac{(\Delta z)^{\bot}\cdot\partial_{\alpha}z(\alpha-\beta)\Delta z\cdot\partial_{\alpha}z(\alpha-\beta)}{\abs{\Delta z}^{2}})u(\alpha-\beta)d\beta\equiv I_{2}^{1}+I_{2}^{2}+I_{2}^{3}.
\end{align*}
Since $\partial_{\alpha}^{\bot}z\cdot\partial_{\alpha}z=0$, $I_{2}^{1}=0$.

We can write
\begin{align*}
I_{2}^{2}&=-\frac{1}{\pi}PV\int_{\R}(\frac{(\Delta z)^{\bot}}{\abs{\Delta z}^{2}}-\frac{\pa{\bot}z(\alpha)}{\beta\abs{\pa{} z(\alpha)}^{2}})\cdot\partial^{2}_{\alpha}z(\alpha-\beta)u(\alpha-\beta)d\beta\\
&-\frac{1}{\pi}PV\int_{\R}\frac{\pa{\bot}z(\alpha)}{\beta\abs{\pa{} z(\alpha)}^{2}}\cdot\partial^{2}_{\alpha}z(\alpha-\beta)u(\alpha-\beta)d\beta\equiv I_{2}^{21}+I_{2}^{22}.
\end{align*} 
Since we compute:
\begin{align*}
&\frac{(\Delta z)^{\bot}}{\abs{\Delta z}^{2}}-\frac{\pa{\bot}z(\alpha)}{\beta\abs{\pa{} z(\alpha)}^{2}}=\frac{\beta\abs{\pa{} z(\alpha)}^{2}\Delta z^{\bot}-\pa{\bot}z(\alpha)\abs{\Delta z}^{2}}{\beta\abs{\pa{} z(\alpha)}^{2}\abs{\Delta z}^{2}}\\
&=\frac{\beta^{2}\abs{\pa{} z(\alpha)}^{2}\int_{0}^{1}\pa{\bot}z(\alpha-\beta+t\beta)dt-\pa{\bot}z(\alpha)\abs{\Delta z}^{2}}{\beta\abs{\pa{} z(\alpha)}^{2}\abs{\Delta z}^{2}}\\
&=\frac{\beta^{3}\abs{\pa{} z(\alpha)}^{2}\int_{0}^{1}\int_{0}^{1}\pa{2\bot}z(\alpha-s\beta+st\beta)(t-1)dsdt+\pa{\bot}z(\alpha)(\beta^{2}\abs{\pa{}z(\alpha)}^{2}-\abs{\Delta z}^{2})}{\beta\abs{\pa{} z(\alpha)}^{2}\abs{\Delta z}^{2}}\\
&\frac{\beta^{2}\int_{0}^{1}\int_{0}^{1}\pa{2}z(\alpha-s\beta+st\beta)(t-1)dsdt}{\abs{\Delta z}^{2}}\\
&+\frac{\beta^{2}\pa{\bot}z(\alpha)\int_{0}^{1}\int_{0}^{1}\pa{2}z(\alpha-\beta+t\beta+s\beta-ts\beta)(1-t)dsdt\cdot\int_{0}^{1}(\pa{}z(\alpha)+\pa{}z(\alpha-\beta+t\beta))dt}{\abs{\pa{} z(\alpha)}^{2}\abs{\Delta z}^{2}},
\end{align*} 
therefore
\begin{align*}
I_{2}^{21}\le C\norm{\F}_{L^{\infty}}\norm{z}_{\mathcal{C}^{2}}^{2}\norm{u}_{L^{2}}.
\end{align*}
For the term $I_{2}^{22}$
\begin{align*}
I_{2}^{22}=&-\frac{1}{\pi}PV\int_{\R}\frac{\pa{\bot}z(\alpha)}{\abs{\pa{} z(\alpha)}^{2}}\cdot\frac{\partial^{2}_{\alpha}z(\alpha-\beta)-\pa{2}z(\alpha)}{\beta}u(\alpha-\beta)d\beta\\
&-\frac{1}{\pi}PV\int_{\R}\frac{\pa{\bot}z(\alpha)}{\abs{\pa{} z(\alpha)}^{2}}\cdot\partial^{2}_{\alpha}z(\alpha)\frac{u(\alpha-\beta)}{\beta}d\beta\\
&\le C\norm{\F}^{\frac{1}{2}}_{L^{\infty}}\norm{z}_{\mathcal{C}^{2,\delta}}\norm{u}_{L^{2}}-\frac{1}{\pi}\frac{\pa{\bot}z(\alpha)\cdot\pa{2}z(\alpha)}{\abs{\pa{}z(\alpha)}^{2}}H(u)\\
&\le C\norm{\F}^{\frac{1}{2}}_{L^{\infty}}\norm{z}_{\mathcal{C}^{2,\delta}}\norm{u}_{L^{2}}.
\end{align*} 
We can see easily for $\phi=\alpha-\beta+t\beta$
\begin{align*}
I_{2}^{3}&= \frac{2}{\pi}PV\int_{\R}\frac{\beta^{2}\int_{0}^{1} \pa{\bot}z(\phi)dt\cdot\partial_{\alpha}z(\alpha-\beta)\int_{0}^{1}\pa{}z(\phi)dt\cdot\partial_{\alpha}z(\alpha-\beta)}{\abs{\Delta z}^{2}}u(\alpha-\beta)d\beta\\
&\le C\norm{\F}_{L^{\infty}}\norm{z}^{2}_{\mathcal{C}^{1}}\norm{z}^{2}_{H^{1}}\norm{u}_{L^{2}}.
\end{align*}
Now, we consider
\begin{align*}
T_{2}^{*}(v)=-\frac{1}{\pi}PV\int_{\R}\frac{(h(\alpha)-z(\beta))^{\bot}\cdot\partial_{\alpha}z(\beta)}{\abs{h(\alpha)-z(\beta)}^{2}}v(\beta)d\beta,
\end{align*}
then
\begin{align*}
\pa{}T_{2}^{*}(v)&=-\frac{1}{\pi}PV\int_{\R}\pa{}(\frac{(h(\alpha)-z(\alpha-\beta))^{\bot}\cdot\partial_{\alpha}z(\alpha-\beta)}{\abs{h(\alpha)-z(\alpha-\beta)}^{2}})v(\alpha-\beta)d\beta\\
&-\frac{1}{\pi}PV\int_{\R}\frac{(h(\alpha)-z(\alpha-\beta))^{\bot}\cdot\partial_{\alpha}z(\alpha-\beta)}{\abs{h(\alpha)-z(\alpha-\beta)}^{2}}\pa{}v(\alpha-\beta)d\beta\equiv J_{1}+J_{2}.
\end{align*}
Using $\pa{\bot}z\cdot\pa{}z=0$,
\begin{align*}
J_{1}&=-\frac{1}{\pi}PV\int_{\R}\frac{\pa{\bot}h(\alpha)\cdot\partial_{\alpha}z(\alpha-\beta)}{\abs{h(\alpha)-z(\alpha-\beta)}^{2}}v(\alpha-\beta)d\beta\\
&-\frac{1}{\pi}PV\int_{\R}\frac{(h(\alpha)-z(\alpha-\beta))^{\bot}\cdot\pa{2}z(\alpha-\beta)}{\abs{h(\alpha)-z(\alpha-\beta)}^{2}}v(\alpha-\beta)d\beta\\
&+\frac{2}{\pi}PV\int_{\R}\frac{(\Delta hz)^{\bot}\cdot\partial_{\alpha}z(\alpha-\beta)\Delta hz\cdot(\pa{}h(\alpha)-\pa{}z(\alpha-\beta))}{\abs{h(\alpha)-z(\alpha-\beta)}^{4}}v(\alpha-\beta)d\beta\\
&\equiv J_{1}^{1}+J_{1}^{2}+J_{1}^{3}.
\end{align*}
Directly,
\begin{align*}
\abs{J_{1}^{1}}\le C\norm{d(z,h)}_{L^{\infty}}\norm{z}_{\mathcal{C}^{2}}\norm{h}_{\mathcal{C}^{1}}\norm{v}_{L^{2}},
\end{align*}

\begin{align*}
\abs{J_{1}^{2}}\le C\norm{d(z,h)}_{L^{\infty}}^{\frac{1}{2}}\norm{z}^{2}_{\mathcal{C}^{2}}\norm{v}_{L^{2}}
\end{align*}
and
\begin{align*}
&J_{1}^{3}\le C\norm{d(z,h)}_{L^{\infty}}\norm{z}_{\mathcal{C}^{1}}(\norm{z}_{\mathcal{C}^{1}}+\norm{h}_{\mathcal{C}^{1}})\norm{v}_{L^{2}}.
\end{align*}

Now, we study the term $J_{2}$. Since $\pa{}z\cdot\pa{\bot}z=0$,
\begin{align*}
J_{2}&=-\frac{1}{\pi}PV\int_{\R}\partial_{\beta}(\frac{(h(\alpha)-z(\alpha-\beta))^{\bot}\cdot\partial_{\alpha}z(\alpha-\beta)}{\abs{h(\alpha)-z(\alpha-\beta)}^{2}})v(\alpha-\beta)d\beta\\
&=\frac{1}{\pi}PV\int_{\R}\frac{(h(\alpha)-z(\alpha-\beta))^{\bot}\cdot\partial_{\alpha}^{2}z(\alpha-\beta)}{\abs{h(\alpha)-z(\alpha-\beta)}^{2}}v(\alpha-\beta)d\beta\\
&-\frac{2}{\pi}PV\int_{\R}\frac{(\Delta zh)^{\bot}\cdot\partial_{\alpha}z(\alpha-\beta)\Delta zh\cdot\pa{}z(\alpha-\beta)}{\abs{h(\alpha)-z(\alpha-\beta)}^{4}}v(\alpha-\beta)d\beta.\\
\end{align*}
Using the same procedure as in term $J_{1}$,
\begin{align*}
\abs{J_{2}^{1}}\le C\norm{d(z,h)}^{\frac{1}{2}}_{L^{\infty}}\norm{z}_{\mathcal{C}^{2}}\norm{v}_{L^{2}}
\end{align*}
and
\begin{align*}
\abs{J_{2}^{2}}\le C\norm{d(z,h)}_{L^{\infty}}\norm{z}^{2}_{\mathcal{C}^{1}}\norm{v}_{L^{2}}.
\end{align*}
The operator $T_{3}^{*}(v)(\alpha)$ is estimated as well as $T_{2}^{*}(u)(\alpha)$ changing $z$ with $h$ and vice verse. For $T_{4}^{*}(v)(\alpha)$ we do the same as for $T_{1}^{*}(u)(\alpha)$ changing $z$ for $h$ and instead of $\F$ the arc-chord condition for $h$, $\G$.

In conclusion,
\begin{displaymath}
\norm{\pa{}\OT^{*}w}_{L^{2}}\le C\norm{\F}^{2}_{L^{\infty}}\norm{\G}_{L^{\infty}}^{2}\norm{d(z,h)}^{2}_{L^{\infty}}\norm{z}_{\mathcal{C}^{2,\delta}}^{2}\norm{h}_{\mathcal{C}^{2,\delta}}^{2}\norm{w}_{L^{2}}.
\end{displaymath}
\end{proof}

Now it will be useful to consider the following functions:

Let $x$ be outside the curve $z(\alpha)$ and $h(\alpha)$, then we define
\begin{align*}
f_{1}(x)=&-\frac{1}{\pi}PV\int_{\R}\frac{(x-z(\beta))^{\bot}\cdot\partial_{\alpha}z(\beta)}{\abs{x-z(\beta)}^{2}}u(\beta)d\beta\\
&=\frac{1}{\pi}PV\int_{\R}\frac{(x_{2}-z_{2}(\beta))\partial_{\alpha}z_{1}(\beta)}{\abs{x-z(\beta)}^{2}}u(\beta)d\beta-\frac{1}{\pi}PV\int_{\R}\frac{(x_{1}-z_{1}(\beta))\partial_{\alpha}z_{2}(\beta)}{\abs{x-z(\beta)}^{2}}u(\beta)d\beta.
\end{align*}
In the following we identify $(u_{1},u_{2})$ with $u_{1}+iu_{2}$. Since $-u^{\bot}\cdot v=u_{2}v_{1}-u_{1}v_{2}$ and $(u_{1}+iu_{2})(v_{1}+iv_{2})=(u_{1}v_{1}+u_{2}v_{2})+i(u_{2}v_{1}-u_{1}v_{2})$
we get,
\begin{displaymath}
-u^{\bot}\cdot v=\im(u\bar{v}).
\end{displaymath}
Therefore, we can write 
\begin{displaymath}
f_{1}(x)=\frac{1}{\pi}\im\int_{\R}\frac{(x-z(\beta))\overline{\partial_{\alpha}z(\beta)}}{\abs{x-z(\beta)}^{2}}u(\beta)d\beta
\end{displaymath}
In the same way,
\begin{displaymath}
f_{2}(x)=\frac{1}{\pi}\im\int_{\R}\frac{(x-h(\beta))\overline{\partial_{\alpha}h(\beta)}}{\abs{x-h(\beta)}^{2}}v(\beta)d\beta
\end{displaymath}
Both are the real part of the following Cauchy integrals
\begin{align*}
F_{1}(x)=f_{1}(x)+ig_{1}(x)=\frac{1}{i\pi}\int_{\R}\frac{(x-z(\beta))\overline{\partial_{\alpha}z(\beta)}}{\abs{x-z(\beta)}^{2}}u(\beta)d\beta,\\
F_{2}(x)=f_{2}(x)+ig_{2}(x)=\frac{1}{i\pi}\int_{\R}\frac{(x-h(\beta))\overline{\partial_{\alpha}h(\beta)}}{\abs{x-h(\beta)}^{2}}v(\beta)d\beta
\end{align*}

Taking $x=z(\alpha)+\epsilon\partial^{\bot}_{\alpha}z(\alpha)$ and letting $\epsilon\to 0$, we obtain

\begin{equation}
\label{f1}
f_{1}(z(\alpha))=T_{1}^{*}(u)(\alpha)-sign(\epsilon)u(\alpha).
\end{equation}

and taking $x=h(\alpha)+\epsilon\partial^{\bot}_{\alpha}h(\alpha)$ and letting $\epsilon\to 0$
\begin{equation}
\label{f4}
f_{2}(h(\alpha))=T_{4}^{*}(v)(\alpha)-sign(\epsilon)v(\alpha).
\end{equation}

Since the curve $z(\alpha)$ does not touch the curve $h(\alpha)$, we have
\begin{equation}
\label{f2}
f_{1}(h(\alpha))=T_{2}^{*}(u)(\alpha)
\end{equation}
and
\begin{equation}
\label{f3}
f_{2}(z(\alpha))=T_{3}^{*}(v)(\alpha).
\end{equation}

On the other hand,
\begin{displaymath}
\lim_{\epsilon\to 0}g_{1}(z(\alpha)\pm\epsilon\partial^{\bot}_{\alpha}z(\alpha))=\lim_{\epsilon\to 0}\im(F_{1}(z(\alpha)\pm\epsilon\partial^{\bot}_{\alpha}z(\alpha))\equiv G_{1}(u)(\alpha)
\end{displaymath}
where
\begin{displaymath}
G_{1}(u)(\alpha)=-\frac{1}{\pi}PV\int_{\R}\frac{(z(\alpha)-z(\beta))\cdot\partial_{\alpha}z(\beta)}{\abs{z(\alpha)-z(\beta)}^{2}}u(\beta)d\beta.
\end{displaymath}

In the same way, taking limits
\begin{displaymath}
\lim_{\epsilon\to 0}g_{2}(h(\alpha)\pm\epsilon\partial^{\bot}_{\alpha}h(\alpha))=\lim_{\epsilon\to 0}\im(F_{2}(h(\alpha)\pm\epsilon\partial^{\bot}_{\alpha}h(\alpha))\equiv G_{2}(u)(\alpha)
\end{displaymath}
where
\begin{displaymath}
G_{2}(v)(\alpha)=-\frac{1}{\pi}PV\int_{\R}\frac{(h(\alpha)-h(\beta))\cdot\partial_{\alpha}h(\beta)}{\abs{h(\alpha)-h(\beta)}^{2}}v(\beta)d\beta
\end{displaymath}
 
 Therefore, we have the fact that $g_{i}^{+}(z(\alpha))=g_{i}^{-}(z(\alpha))$ and $g_{i}^{+}(h(\alpha))=g_{i}^{-}(h(\alpha))$ for $i=1,2$, where $(\cdot)^{+}$ denotes the limit obtained approaching from above to the boundaries in the normal direction and $(\cdot)^{-}$ from below.(This fact will be used on Subsection \ref{operhi}).
 
Now we will show that $\OT^{*}w=\lambda w\Rightarrow\abs{\lambda}<1$.

   If $w$ is a eigenvector of $\OT$, we have
   \begin{displaymath}
   \OT^{*}w = \;
   \begin{pmatrix}
      T_{1}^{*} &  T_{3}^{*} \\
      T_{2}^{*} &  T_{4}^{*} 
   \end{pmatrix}
   \begin{pmatrix}
      u \\
      v
   \end{pmatrix}=
   \begin{pmatrix}
      T_{1}^{*}u + T_{3}^{*}v \\
      T_{2}^{*}u + T_{4}^{*}v 
   \end{pmatrix}=
    \begin{pmatrix}
      \lambda u \\
      \lambda v
   \end{pmatrix}= \lambda w.
   \end{displaymath}
   Let us compute $\nabla f_{i}$ for $i=1,2$. 
   The identity
   \begin{align*}
   f_{1}(x)&=\frac{1}{\pi}\im\int_{\R}\frac{(x-z(\beta))\overline{\partial_{\alpha}z(\beta)}}{\abs{x-z(\beta)}^{2}}u(\beta)d\beta=\frac{1}{\pi}\im\int_{\R}\frac{\overline{\partial_{\alpha}z(\beta)}}{\overline{(x-z(\beta))}}u(\beta)d\beta\\
   &=-\frac{1}{\pi}\im\int_{\R}\partial_{\beta}\ln(x-z(\beta))u(\beta)d\beta=\frac{1}{\pi}\im\int_{\R}\ln(x-z(\beta))\partial_{\beta}u(\beta)d\beta
   \end{align*}
   yields
   \begin{displaymath}
   \nabla f_{1}(x)=\frac{1}{\pi}\im\int_{\R}\partial_{\beta}u(\beta)\nabla\ln(x-z(\beta))d\beta.
   \end{displaymath}
   That is
   \begin{displaymath}
   \nabla f_{1}(x)=\frac{1}{\pi}\int_{\R}\partial_{\beta}u(\beta)\nabla\arg(x-z(\beta))d\beta=\frac{1}{\pi}\int_{\R}\frac{(x-z(\beta))^{\bot}}{\abs{x-z(\beta)}^{2}}\partial_{\beta}u(\beta)d\beta.
   \end{displaymath}
   In the same way,
   \begin{displaymath}
   \nabla f_{2}(x)=\frac{1}{\pi}\int_{\R}\frac{(x-h(\beta))^{\bot}}{\abs{x-h(\beta)}^{2}}\partial_{\beta}v(\beta)d\beta.
   \end{displaymath}
   Taking $x=z(\alpha)+\epsilon z(\alpha)$ and letting $\epsilon\to 0$ in $\nabla f_{1}$ we have
   \begin{equation}
   \label{nf1}
   \nabla f_{1}(z(\alpha))=2BR(\partial_{\alpha}u,z)_{z}-sign(\epsilon)\frac{\pa{}u(\alpha)\partial_{\alpha}z(\alpha)}{2\abs{\partial_{\alpha}z(\alpha)}^{2}}.
   \end{equation}
   On the other hand, taking $x=h(\alpha)+\epsilon h(\alpha)$ on $\nabla f_{2}$ and letting $\epsilon\to 0$,
   \begin{equation}
   \label{nf4}
   \nabla f_{2}(h(\alpha))=2BR(\partial_{\alpha}v,h)_{h}-sign(\epsilon)\frac{\pa{}v(\alpha)\partial_{\alpha}h(\alpha)}{2\abs{\partial_{\alpha}h(\alpha)}^{2}}.
   \end{equation}
  Obviously,
  
   \begin{equation}
   \label{nf2}
   \nabla f_{1}(h(\alpha))=2BR(\partial_{\alpha}u,z)_{h}
   \end{equation}
 and
   \begin{equation}
   \label{nf3}
   \nabla f_{2}(z(\alpha))=2BR(\partial_{\alpha}v,h)_{z}.
   \end{equation}
  
Assuming now that $\OT^{*}w=\lambda w$, $\Omega_{1}$ is the domain placed above of the curve $z(\alpha)$, $\Omega_{2}$ is the domain between $z(\alpha)$ and $h(\alpha)$ and $\Omega_{3}$ is below of the curve $h(\alpha)$. The analyticity of $F_{i}$ for $i=1,2$ allows us to obtain:
\begin{align}
\label{lambda1}
&0<\int_{\Omega_{1}}\abs{F'_{1}(x)+F_{2}'(x)}^{2}dx=2\int_{\Omega_{1}}\abs{(\nabla f_{1}(x)+\nabla f_{2}(x))}^{2}dx\\\nonumber
&=-2\int_{\Omega_{1}}\Delta(f_{1}(x)+f_{2}(x))(f_{1}(x)+f_{2}(x))dx\\\nonumber
&-2\int_{\T}(f_{1}^{+}(z(\alpha))+f_{2}^{+}(z(\alpha)))(\nabla f_{1}^{+}(z(\alpha))+\nabla f_{2}^{+}(z(\alpha)))\cdot\frac{\partial_{\alpha}^{\bot}z(\alpha)}{\abs{\partial_{\alpha}z(\alpha)}}d\alpha\\\nonumber
&=2\int_{\T}(-T_{1}^{*}(u)(\alpha)+u(\alpha)-T_{3}^{*}(v)(\alpha))(2BR(\partial_{\alpha}u,z)_{z}+2BR(\partial_{\alpha}v,h)_{z})\cdot\frac{\partial_{\alpha}^{\bot}z(\alpha)}{\abs{\partial_{\alpha}z(\alpha)}}d\alpha\\\nonumber
&=\int_{\T}(u(\alpha)-\lambda u(\alpha))M(u,v,h,z)d\alpha=(1-\lambda)\int_{\T}u(\alpha)M(u,v,h,z)d\alpha\\\nonumber
&\equiv(1-\lambda)A,
\end{align}
\begin{align}
\label{lambda2}
&0<\int_{\Omega_{2}}\abs{F'_{1}(x)+F_{2}'(x)}^{2}dx\\\nonumber
&=2\int_{\T}(f_{1}(z(\alpha))+f_{2}(z(\alpha)))(\nabla f_{1}(z(\alpha))+\nabla f_{2}(z(\alpha)))\cdot\frac{\partial_{\alpha}^{\bot}z(\alpha)}{\abs{\partial_{\alpha}z(\alpha)}}d\alpha\\\nonumber
&-2\int_{\T}(f_{1}^{+}(h(\alpha))+f_{2}^{+}(h(\alpha)))(\nabla f_{1}^{+}(h(\alpha))+\nabla f_{2}^{+}(h(\alpha)))\cdot\frac{\partial_{\alpha}^{\bot}h(\alpha)}{\abs{\partial_{\alpha}h(\alpha)}}d\alpha\\\nonumber
&=2\int_{\T}(u(\alpha)+T_{1}^{*}(u)(\alpha)+T_{3}^{*}(v)(\alpha))M(u,v,h,z)d\alpha\\\nonumber
&+2\int_{\T}(v(\alpha)-T_{2}^{*}(u)(\alpha)-T_{4}^{*}(v)(\alpha))(2BR(\partial_{\alpha}u,z)_{h}+2BR(\partial_{\alpha}v,h)_{h})\cdot\frac{\partial_{\alpha}^{\bot}z(\alpha)}{\abs{\partial_{\alpha}z(\alpha)}}d\alpha\\\nonumber
&=\int_{\T}(u(\alpha)+\lambda u(\alpha))M(u,v,h,z)d\alpha+2\int_{\T}(v(\alpha)-\lambda v(\alpha))N(u,v,h,z)\\\nonumber
&=(1+\lambda)\int_{\T}u(\alpha)M(u,v,h,z)d\alpha+(1-\lambda)\int_{\T}v(\alpha)N(u,v,h,z)d\alpha\\\nonumber
&\equiv (1+\lambda)A+(1-\lambda)B
\end{align}
and
\begin{align}
\label{lambda3}
&0<\int_{\Omega_{3}}\abs{F'_{1}(x)+F_{2}'(x)}^{2}dx\\\nonumber
&=2\int_{\T}(f_{1}^{-}(h(\alpha))+f_{2}^{-}(h(\alpha)))(\nabla f_{1}^{-}(h(\alpha))+\nabla f_{2}^{-}(h(\alpha)))\cdot\frac{\partial_{\alpha}^{\bot}h(\alpha)}{\abs{\partial_{\alpha}h(\alpha)}}d\alpha\\\nonumber
&=2\int_{\T}(v(\alpha)+T_{2}^{*}(u)(\alpha)+T_{4}^{*}(v)(\alpha))N(u,v,h,z)d\alpha\\\nonumber
&=\int_{\T}(v(\alpha)+\lambda v(\alpha))N(u,v,h,z)d\alpha=(1+\lambda)B
\end{align}
where we have used (\ref{f1})-(\ref{nf3}). 
Suppose that $\abs{\lambda}\ge 1$ then $\lambda\in(-\infty,-1]\cup[1,\infty)$:
\begin{itemize}
\item[$\to$] If $\lambda\in(-\infty,-1]$ then
\begin{itemize}
\item[i)]For (\ref{lambda1}) we get that $A>0$.
\item[ii)]For (\ref{lambda3}) we get that $B<0$ and $\lambda\neq -1$.
\item[iii)]Therefore, (\ref{lambda2}) is a contradiction.
\end{itemize}
\item[$\to$] If $\lambda\in[1,\infty)$
\begin{itemize}
\item[i)]For (\ref{lambda1}) we get that $A<0$ and $\lambda\neq 1$.
\item[ii)]For (\ref{lambda3}) we get that $B>0$.
\item[iii)]Therefore, (\ref{lambda2}) is a contradiction.
\end{itemize}
\end{itemize}
Thus $\abs{\lambda}<1$. At this point, since $\OT^{*}$ is a compact operator, we know that there exists $(I-M\OT^{*})^{-1}$ for $M=\left (\begin{matrix} 
         \mu_{1} & 0  \\
         0 & \mu_{2} 
      \end{matrix}\right)$ with $\abs{\mu_{i}}<1$ for $i=1,2$.
      
      Our propose is to prove that $H^{\frac{1}{2}}$-norm of the inverse operator are bounded by $\exp(C\nor{z,h}^{2})$ where $\nor{z,h}^{2}=\norm{\F}^{2}_{L^{\infty}}+\norm{d(z,h)}^{2}_{L^{\infty}}+\norm{z}^{2}_{H^{3}}$.
      To prove that we will start with the following proposition:
\begin{prop}\label{norml2}
The norms $\norm{(I\pm\OT^{*})^{-1}}_{L^{2}_{0}}$ are bounded by $\exp(C\nor{z,h}^{2})$ for some universal constant $C$. Here the space $L^{2}_{0}$ is the usual $L^{2}$ with the extra condition of mean value zero.
\end{prop}
\begin{proof}
The proof follows if we demostrate the estimate
\begin{equation}
\label{esti}
e^{-C\nor{z,h}^{2}}\le\frac{\norm{\varpi-\OT^{*}\varpi}_{L^{2}_{0}}}{\norm{\varpi+\OT^{*}\varpi}_{L^{2}_{0}}}\le e^{C\nor{z,h}^{2}}
\end{equation}
valid for every nonzero $\varpi\in L^{2}_{0}\times L^2_{0}$.

This is because if we assume $\norm{\varpi-\OT^{*}\varpi}_{L^{2}_{0}}\le e^{-2C\nor{z,h}^{2}}$ for some $\norm{\varpi}_{L^{2}_{0}}=1$ then we obtain $\norm{\varpi+\OT^{*}\varpi}_{L^{2}_{0}}\ge 2\norm{\varpi}_{L^{2}_{0}}-e^{-2C\nor{z,h}^{2}}\ge 1$ wich contradicts \ref{esti}. Therefore we must have $\norm{\varpi-\OT^{*}\varpi}_{L^{2}_{0}}\ge e^{-2C\nor{z,h}^{2}}$ for all $\norm{\varpi}_{L^{2}_{0}}=1$ i.e. $\norm{(I-\OT^{*})^{-1}}_{L^{2}_{0}}\le e^{2C\nor{z,h}^{2}}$. Similarly we also have $\norm{(I+\OT^{*})^{-1}}_{L^{2}_{0}}\le e^{2C\nor{z,h}^{2}}$.

Since
\begin{displaymath}
\varpi+\OT^{*}\varpi=\left (\begin{matrix} u+T^{*}_{1}u+T^{*}_{3}v\\
																v+T_{2}^{*}u+T_{4}^{*}v
																\end{matrix}\right)=\left (\begin{matrix} f_{1}^{-}(z(\alpha))+f_{2}(z(\alpha))\\
																f_{1}(h(\alpha))+f_{2}^{-}(h(\alpha))
																\end{matrix}\right)\equiv\left(\begin{matrix}
																m^{+}\\
w																\end{matrix}\right)
\end{displaymath}
\begin{displaymath}
\varpi-\OT^{*}\varpi=\left (\begin{matrix} u-T^{*}_{1}u-T^{*}_{3}v\\
																v-T_{2}^{*}u-T_{4}^{*}v
																\end{matrix}\right)=(-1)\left (\begin{matrix} f_{1}^{+}(z(\alpha))+f_{2}(z(\alpha))\\
																f_{1}(h(\alpha))+f_{2}^{+}(h(\alpha))
																\end{matrix}\right)\equiv(-1)\left(\begin{matrix}
																f\\
																m^{-}
																\end{matrix}\right)
\end{displaymath}
Next we will see that we can write the above function as some operators, which we call $\mathcal{H}_{i}$ for $i=1,2,3$,  
where $i$ denotes the corresponding domain $\Omega_{1}$, $\Omega_{2}$, and $\Omega_{3}$(See Subsection \ref{operhi}). The relations with these operator are:
\begin{align*}
&m^{+}=\mathcal{H}_{2}^{z}(f,m^{-}),\\
&w=\mathcal{H}_{3}(f,m^{-}),\\
&f=\mathcal{H}_{1}(m^{+},w),\\
&m^{-}=\mathcal{H}_{2}^{h}(m^{+},w).
\end{align*}
And we will prove that
\begin{displaymath}
\norm{\mathcal{H}_{i}(\varpi)}_{L^{2}}\le e^{C\nor{z,h}^{2}}\norm{\varpi}_{L^{2}},
\end{displaymath}
where $C$ denotes a universal constant not necessarily the same at each occurrence.
With all these assumptions, the proof is as follows:
\begin{align*}
&\norm{\varpi+\OT^{*}\varpi}_{L^{2}_{0}}=\norm{\left(\begin{matrix}
\mathcal{H}_{2}^{z}(f,m^{-})\\
\mathcal{H}_{3}(f,m^{-})
\end{matrix}\right)}_{L^{2}_{0}}\le e^{C\nor{z,h}^{2}}\norm{\left(\begin{matrix}
f\\
m^{-}
\end{matrix}\right)}_{L^{2}_{0}}\\
&=e^{2C\nor{z,h}^{2}}\norm{\varpi-\OT^{*}\varpi}_{L^{2}_{0}}.
\end{align*}
In the same way,
\begin{align*}
&\norm{\varpi-\OT^{*}\varpi}_{L^{2}_{0}}=\norm{\left(\begin{matrix}
\mathcal{H}_{1}(m^{+},w)\\
\mathcal{H}_{2}^{h}(m^{+},w)
\end{matrix}\right)}_{L^{2}_{0}}\le e^{C\nor{z,h}^{2}}\norm{\left(\begin{matrix}
m^{+}\\
w
\end{matrix}\right)}_{L^{2}_{0}}\\
&=e^{2C\nor{z,h}^{2}}\norm{\varpi+\OT^{*}\varpi}_{L^{2}_{0}}.
\end{align*}

\end{proof}
Once we have the estimation of $(I\pm\OT^{*})^{-1}$, we introduce the term $M=\left (\begin{matrix} 
         \mu_{1} & 0  \\
         0 & \mu_{2} 
      \end{matrix}\right)$ with $\abs{\mu_{i}}<1$ for all $i=1,2$.
      \begin{lem}\label{l2conm}
      The following estimate holds:
      \begin{displaymath}
      \norm{(I+ M\OT^{*})^{-1}}_{L^{2}_{0}}\le e^{C\nor{z,h}^{2}}
      \end{displaymath}
      for a universal constant $C$ and $\abs{\mu_{i}}\le 1$ for $i=1,2$.
      \end{lem}
      \begin{proof}
      If we look at the identity $I+ M\OT^{*}=M(I+\OT^{*})+(I-MI)$, using the estimate on proposition \ref{norml2} we can conclude that 
      \begin{displaymath}
      \norm{(I+ M\OT^{*})^{-1}}_{L^{2}_{0}}\le\exp{C\nor{z,h}^{2}}
      \end{displaymath}
      for $1-e^{-C_{1}\nor{z,h}^{2}}\le\abs{\mu_{i}}\le 1$.
      
      For $\abs{\mu_{i}}\le 1-e^{-C_{1}\nor{z,h}^{2}}$:
      
      Since $\norm{M\OT^{*}}_{L^{2}}<1$ then we can write $(I+ M\OT^{*})^{-1}=\sum_{n}(M\OT^{*})^{n}$. Taking norms,
      \begin{displaymath}
      \norm{(I+ M\OT^{*})^{-1}}_{L^{2}_{0}}\le\sum_{n}\norm{M\OT^{*}}^{n}_{L^{2}_{0}}\le\sum_{n}(1-e^{-C_{1}\nor{z,h}^{2}})^{n}=e^{C_{1}\nor{z,h}^{2}}
      \end{displaymath}
      \end{proof}
     
Now we are in position to prove the $H^{\frac{1}{2}}$-norm,

\begin{prop}
\label{invertoperhunmedio}
For $\mu_{i}\le 1$ the following estimate holds
\begin{displaymath}
\norm{(I+M\OT)^{-1}}_{H^{\frac{1}{2}}_{0}}=\norm{(I+M\OT^{*})^{-1}}_{H^{\frac{1}{2}}_{0}}\le e^{C\nor{z,h}^{2}},\\
\end{displaymath}
where $C$ is a universal constant and $M=\left(\begin{matrix}
\mu_{1} & 0\\
0 & \mu_{2}
\end{matrix}\right)$.
\end{prop}
\begin{proof}
We will use the same idea as in Proposition \ref{norml2}, therefore we are going to prove:
\begin{equation*}
e^{-C\nor{z,h}^{2}}\le\frac{\norm{\varpi-\OT^{*}\varpi}_{H^{\frac{1}{2}}_{0}}}{\norm{\varpi+\OT^{*}\varpi}_{H^{\frac{1}{2}}_{0}}}\le e^{C\nor{z,h}^{2}}.
\end{equation*}
To do that, using (\ref{estimaT}) and $\abs{\mu_{i}}<1$, 
\begin{align*}
&\norm{\Lambda^{\frac{1}{2}}(\varpi+M\OT^{*}\varpi)}_{L^{2}_{0}}\le\norm{\Lambda^{\frac{1}{2}}(\varpi-M\OT^{*}\varpi)}_{L^{2}_{0}}+2\norm{M\Lambda^{\frac{1}{2}}(\OT^{*}\varpi)}_{L^{2}_{0}}\\\
&\le\norm{\Lambda^{\frac{1}{2}}(\varpi-M\OT^{*}\varpi)}_{L^{2}_{0}}+2\norm{\OT^{*}\varpi)}_{H^{1}}\\
&\le\norm{\Lambda^{\frac{1}{2}}(\varpi-M\OT^{*}\varpi)}_{L^{2}_{0}}+e^{C\nor{z,h}^{2}}\norm{\varpi}_{L^{2}_{0}}.
\end{align*}
Using the estimate of Lemma \ref{l2conm}, 
\begin{displaymath}
\norm{\varpi}_{L^{2}_{0}}=\norm{(I-M\OT^{*})^{-1}(I-M\OT^{*})\varpi}\le e^{C\nor{z,h}^{2}}\norm{\varpi-M\OT^{*}\varpi}_{L^{2}_{0}}.
\end{displaymath}
Therefore,
\begin{displaymath}
\norm{\varpi+M\OT^{*}\varpi}_{H^{\frac{1}{2}}_{0}}\le e^{C\nor{z,h}^{2}}\norm{\varpi-M\OT^{*}\varpi}_{H^{\frac{1}{2}}_{0}}
\end{displaymath}
Analogously, we get
\begin{displaymath}
\norm{\varpi-M\OT^{*}\varpi}_{H^{\frac{1}{2}}_{0}}\le e^{C\nor{z,h}^{2}}\norm{\varpi+M\OT^{*}\varpi}_{H^{\frac{1}{2}}_{0}}
\end{displaymath}
and we finish the proof.
\end{proof}

\subsection{$\mathcal{H}_{i}$ Operators}\label{operhi}

The truth of the above results depend on the existence of the $\mathcal{H}_{i}$ Operators which we have denoted on Proposition \ref{norml2}.

We will start with considering a flat domain, where the boundaries are $(x,0)$ and $(x,1)$.

Let be $F$ a harmonic function, decaying at infinity, above $(x,1)$ such that
\begin{displaymath}
\left\{ \begin{array}{ll}
\Delta F=0\\
F(x,1)=f(x)
\end{array} \right.
\end{displaymath}
Taking Fourier transform, we can get $\hat{F}(\xi,y)=e^{-\abs{\xi}(y-1)}\hat{f}(\xi)$.

Now, if we calculate the harmonic conjugate,which we will call $G$, we can get $\hat{G}(\xi,y)=-i sign(\xi)\hat{f}(\xi)e^{-\abs{\xi}(y-1)}$. And therefore, $\hat{G}(\xi,1)=-i sign(\xi)\hat{f}(\xi)$.

Now we consider between the boundaries the harmonic function $M$ such that,
\begin{displaymath}
\left\{ \begin{array}{ll}
\Delta M=0\\
M(x,1)=m^{+}(x)\\
M(x,0)=m^{-}(x)
\end{array} \right.
\end{displaymath}
Taking Fourier Transform and computing the harmonic conjugate, we get\\ $\hat{N}(\xi,y)=iA\cosh(\xi y)+iB\sinh(\xi y)$.

At the end, we want to relate these harmonic function with ours $F_{i}$ described at the Subsection \ref{estiminversesect}. We saw that $g_{i}^{+}(z(\alpha))=g_{i}^{-}(z(\alpha))$ and $g_{i}^{+}(h(\alpha))=g_{i}^{-}(h(\alpha))$ for $i=1,2$.  That is why we consider now $G(x,1)=N(x,1)$ and before $N(x,0)=R(x,0)$.

Therefore, since
\begin{align*}
&\hat{N}(\xi,1)=iA\cosh(\xi)+iB\sinh(\xi)=\hat{G}(\xi,1)=-i sign(\xi)\hat{f}(\xi),\\
&\hat{M}(\xi,0)=B=\hat{m}^{-}(\xi)
\end{align*}
then
\begin{displaymath}
A=\frac{sign(\xi)\hat{f}(\xi)-\hat{m}^{-}(\xi)\sinh(\xi)}{\cosh(\xi)}
\end{displaymath}
and
\begin{displaymath}
\hat{m}^{+}(\xi)=\hat{M}(\xi,1)=\frac{sign(\xi)\hat{f}(\xi)\sinh(\xi)+\hat{m}^{-}(\xi)}{\cosh(\xi)}.
\end{displaymath}
Moreover, 
\begin{displaymath}
\hat{N}(\xi,0)=\frac{isign(\xi)\hat{f}(\xi)-i\hat{m}^{-}(\xi)\sinh(\xi)}{\cosh(\xi)}.
\end{displaymath}

Finally, we consider an harmonic function $W$ below $(x,0)$ in such a way that
 
\begin{displaymath}
\left\{ \begin{array}{ll}
\Delta W=0\\
W(x,0)=w(x)
\end{array} \right.
\end{displaymath}

With the same procedure as before, we get the harmonic conjugate $\hat{R}(\xi,y)=isign(\xi)\hat{w}(\xi)e^{\abs{x}y}$.
Since, $\hat{R}(\xi,0)=\hat{N}(\xi,0)$ therefore 
\begin{displaymath}
\hat{w}(\xi)=\frac{\hat{f}(\xi)-sign(\xi)\hat{m}^{-}(\xi)\sinh(\xi)}{\cosh(\xi)}.
\end{displaymath}

Thus we just put $\hat{m}^{+}$ and $\hat{w}$ as a function of $\hat{f}$ and $\hat{m}^{-}$. We do this going from the top of the domain to the bottom. If we do the same going from the bottom to the top we will obtain 
\begin{align*}
&\hat{f}(\xi)=\frac{-\hat{w}(\xi)-sign(\xi)\hat{m}^{+}(\xi)\sinh(\xi)}{\cosh(\xi)},\\
&\hat{m}^{-}(\xi)=\frac{\hat{m}^{+}(\xi)-sign(\xi)\hat{w}(\xi)\sinh(\xi)}{\cosh(\xi)}.
\end{align*}

Here we define ours operators like:
\begin{align*}
&\widehat{H_{1}(m^{+},w)}=\frac{-\hat{w}(\xi)-sign(\xi)\hat{m}^{+}(\xi)\sinh(\xi)}{\cosh(\xi)},\\
&\widehat{H_{2}^{h}(m^{+},w)}=\frac{\hat{m}^{+}(\xi)-sign(\xi)\hat{w}(\xi)\sinh(\xi)}{\cosh(\xi)},\\
&\widehat{H_{2}^{z}(f,m^{-})}=\frac{\hat{m}^{-}(\xi)+sign(\xi)\hat{f}(\xi)\sinh(\xi)}{\cosh(\xi)},\\
&\widehat{H_{3}(f,m^{-})}=\frac{\hat{f}(\xi)-sign(\xi)\hat{f}(\xi)\sinh(\xi)}{\cosh(\xi)},
\end{align*}
which are bounded in $L^{2}$.
\begin{figure}[htb]
\centering
\includegraphics[width=65mm]{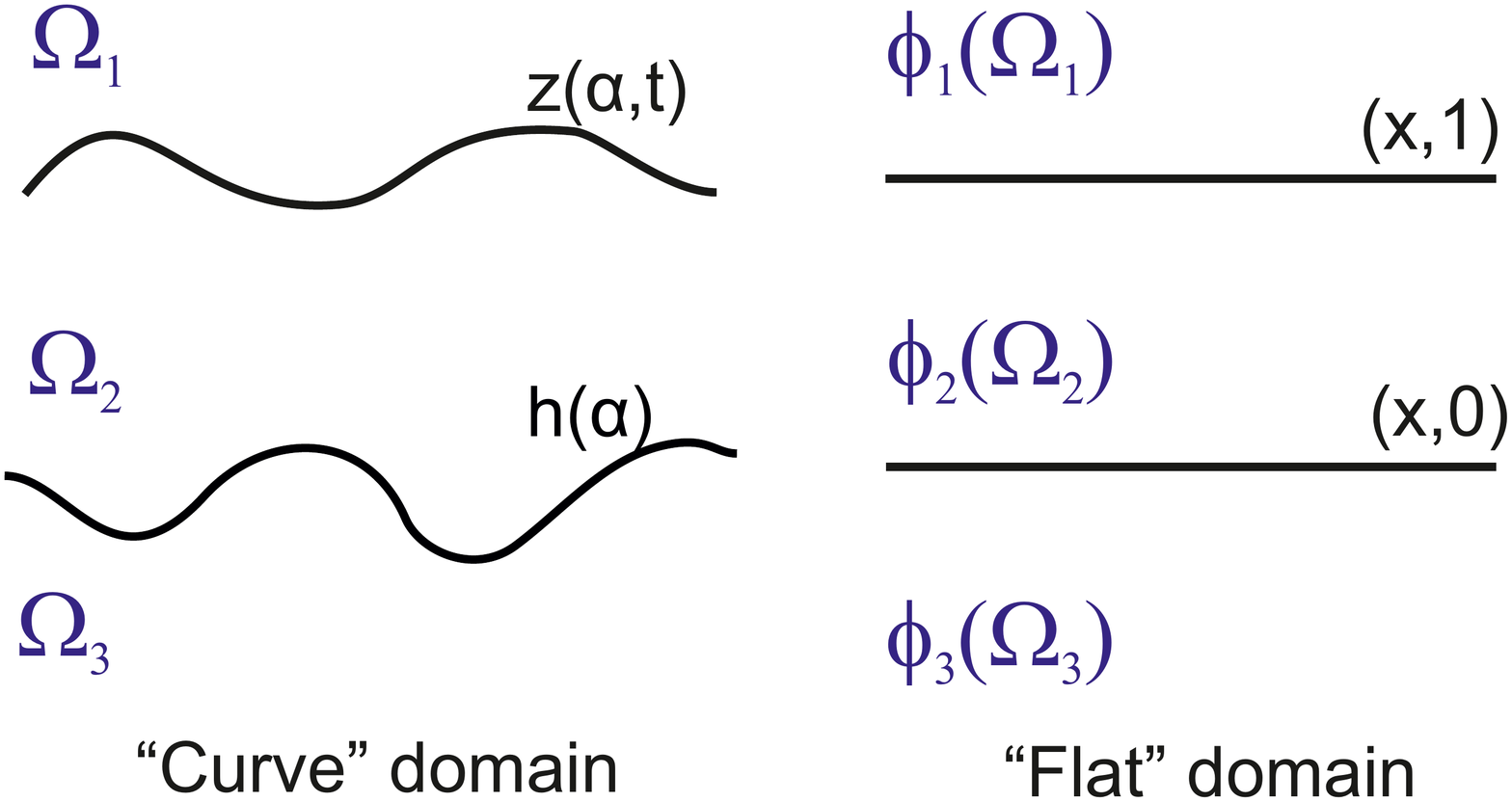}
 \caption{Conformal maps $\phi_{i}$}\label{fig:flat}
\end{figure}

Let $\phi_{i}$ the conformal mapping from the $\Omega_{i}$ domain to the "flat" domain (See the figure \ref{fig:flat}), then the corresponding operator in the "curve" domain are denoted by $\mathcal{H}_{i}$.

For the $L^{2}$-norm of the $\mathcal{H}_{i}$ Operator we can repeat the proofs in \cite{hele} for their Hilbert operator $\mathcal{H}_{1}$.

To do this we only have to look at the formulas:
\begin{align*}
&\mathcal{H}_{1}(m^{+},w)=H_{1}(m^{+}\circ\phi_{1}^{-1},w\circ\phi_{1}^{-1})\circ\phi_{1},\\
&\mathcal{H}_{2}^{h}(m^{+},w)=H_{2}^{h}(m^{+}\circ\phi_{2}^{-1},w\circ\phi_{2}^{-1})\circ\phi_{2},\\
&\mathcal{H}_{2}^{z}(f,m^{-})=H_{2}^{z}(f\circ\phi_{2}^{-1},m^{-}\circ\phi_{2}^{-1})\circ\phi_{2},\\
&\mathcal{H}_{3}(f,m^{-})=H_{3}(f\circ\phi_{3}^{-1},m^{-}\circ\phi_{3}^{-1})\circ\phi_{3}.
\end{align*}

Since our parametric curves $z(\alpha)$ and $h(\alpha)$ are $\mathcal{C}^{2,\delta}$ satisfying the arc-chord conditions $\norm{\F}_{L^{\infty}}<\infty$, $\norm{\G}_{L^{\infty}}<\infty$ and the distance $\norm{\di}_{L^{\infty}}<\infty$. Then we have tangent balls to the boundary contained inside the domains $\Omega_{i}$. Furthermore, we can estimate from below the radius of those balls by $C\nor{z,h}^{-1}$(As in Lemma 4.3 in \cite{hele}).

Following the steps of the proof of Lemma 4.4 in \cite{hele} we can conclude that 
\begin{displaymath}
\norm{\mathcal{H}_{i}}_{L^{2}}\le e^{C\nor{z,h}^{2}}
\end{displaymath}
for all $i=1,2,3$.

\section{Estimates on $\varpi$}\label{estiamplitud}
In this section we show that the norm of amplitude of the vorticity $\varpi=(\varpi_{1},\varpi_{2})$ is bounded in $H^{k}$, for $k\ge 2$.
\begin{lem}
Let $\varpi=(\varpi_{1},\varpi_{2})$ be a function given by
\begin{align}
\label{var1}
&\varpi_{1}(\alpha)=-\gamma_{1}T_{1}(\varpi_{1})(\alpha)-\gamma_{1}T_{2}(\varpi_{2})(\alpha)-N\pa{}z_{2}(\alpha),\\
\label{var2}
&\varpi_{2}(\alpha)=-\gamma_{2}T_{3}(\varpi_{1})(\alpha)-\gamma_{2}T_{4}(\varpi_{2})(\alpha)
\end{align}
where $\gamma_{1}=\frac{\mu_{2}-\mu_{1}}{\mu_{1}+\mu_{2}}$, $\gamma_{2}=\frac{\kappa_{1}-\kappa_{2}}{\kappa_{1}+\kappa_{2}}$ and $N=2\kappa_{1}g\frac{\rho_{2}-\rho_{1}}{\mu_{2}+\mu_{1}}$.

Then
\begin{displaymath}
\norm{\varpi}_{H^{k}}\le\esth{k+1}
\end{displaymath}
for $k\ge 2$.
\end{lem}
\begin{proof}
We can write,
\begin{equation}
\label{amplieq}
\varpi=M\OT\varpi-v
\end{equation}
where $M=\left(\begin{matrix}
-\gamma_{1}&0\\
0&-\gamma_{2}
\end{matrix}\right)$, $\OT=\left(\begin{matrix}
T_{1}&T_{2}\\
T_{3}&T_{4}
\end{matrix}\right)$ and $v=\left(\begin{matrix}
N\pa{}z_{2}(\alpha)\\
0
\end{matrix}\right)$.
The formula (\ref{amplieq}) is equivalent to
\begin{displaymath}
\varpi=(I+M\OT)^{-1}v.
\end{displaymath}
It yields
\begin{displaymath}
\norm{\varpi}_{H^{\frac{1}{2}}}\le\norm{(I+M\OT)^{-1}}_{H^{\frac{1}{2}}}\norm{\pa{}z_{2}}_{H^{\frac{1}{2}}}.
\end{displaymath}
Since $\abs{\gamma_{i}}<1$ for all $i$, the proposition \ref{invertoperhunmedio} gives
\begin{displaymath}
\norm{\varpi}_{H^{\frac{1}{2}}}\le e^{C\nor{z,h}^{2}}.
\end{displaymath}
Recall that $\nor{z,h}^{2}=\norm{\F}^{2}_{L^{\infty}}+\norm{d(z,h)}^{2}_{L^{\infty}}+\norm{z}^{2}_{H^{3}}$.
Now, we consider the $H^{k+1}$-norm
\begin{displaymath}
\norm{\varpi}_{H^{k+1}}=\norm{\varpi_{1}}_{H^{k+1}}+\norm{\varpi_{2}}_{H^{k+1}}.
\end{displaymath}
Then, we study each component one by one. 

Taking the $k$ derivative of (\ref{var1}) we get:
\begin{displaymath}
\pa{k}\varpi_{1}(\alpha)=-\lambda_{1}\pa{k}(2\brz{z}\cdot\pa{}z(\alpha))-\lambda_{1}\pa{k}(2\brh{z}\cdot\pa{}z(\alpha))-N\pa{k+1}z_{2}(\alpha).
\end{displaymath}
Using Leibniz's rule we have,
\begin{align*}
&2\pa{k}(\brz{z}\cdot\pa{}z(\alpha))=\sum_{j=0}^{k}\frac{C_{k}}{\pi}\int_{\R}\pa{k-j}(\frac{\Delta z^{\bot}\cdot\pa{}z(\alpha)}{\abs{\Delta z}^{2}})\pa{j}\varpi_{1}(\alpha-\beta)d\beta\\
&=\sum_{j=0}^{k-1}\frac{C_{k}}{\pi}\int_{\R}\pa{k-j}(\frac{\Delta z^{\bot}\cdot\pa{}z(\alpha)}{\abs{\Delta z}^{2}})\pa{j}\varpi_{1}(\alpha-\beta)d\beta+T_{1}(\pa{k}\varpi_{1})(\alpha)
\end{align*}
and
\begin{align*}
&2\pa{k}(\brh{z}\cdot\pa{}z(\alpha))\\
&=\sum_{j=0}^{k-1}\frac{C_{k}}{\pi}\int_{\R}\pa{k-j}(\frac{\Delta zh^{\bot}\cdot\pa{}z(\alpha)}{\abs{\Delta zh}^{2}})\pa{j}\varpi_{2}(\alpha-\beta)d\beta+T_{2}(\pa{k}\varpi_{2})(\alpha).
\end{align*}
Recall that $\Delta z=z(\alpha)-z(\alpha-\beta)$ and $\Delta zh=z(\alpha)-h(\alpha-\beta)$.
Therefore, we obtain
\begin{displaymath}
\pa{k}\varpi_{1}(\alpha)+\lambda_{1}T_{1}(\pa{k}\varpi_{1})(\alpha)+\lambda_{1}T_{2}(\pa{k}\varpi_{2})(\alpha)=R_{k}^{1}(\varpi_{1})+R_{k}^{2}(\varpi_{2})-N\pa{k+1}z_{2}(\alpha)
\end{displaymath}
where
\begin{align*}
&R^{1}_{k}(\varpi_{1})=\sum_{j=0}^{k-1}\frac{C_{k}}{\pi}\int_{\R}\pa{k-j}(\frac{\Delta z^{\bot}\cdot\pa{}z(\alpha)}{\abs{\Delta z}^{2}})\pa{j}\varpi_{1}(\alpha-\beta)d\beta,\\
&R^{2}_{k}(\varpi_{2})=\sum_{j=0}^{k-1}\frac{C_{k}}{\pi}\int_{\R}\pa{k-j}(\frac{\Delta zh^{\bot}\cdot\pa{}z(\alpha)}{\abs{\Delta zh}^{2}})\pa{j}\varpi_{2}(\alpha-\beta)d\beta.
\end{align*}
If we observe that
\begin{displaymath}
\pa{}T_{1}(\varpi_{1})(\alpha)=\frac{1}{\pi}\int_{\R}\pa{}(\frac{\Delta z^{\bot}\cdot\pa{}z(\alpha)}{\abs{\Delta z}^{2}})\varpi_{1}(\alpha-\beta)d\beta+T_{1}(\pa{}\varpi_{1})(\alpha)
\end{displaymath}
we get
\begin{align*}
&R_{k}^{1}(\varpi_{1})=\pa{k-1}(\frac{1}{\pi}\int_{\R}\pa{}(\frac{\Delta z^{\bot}\cdot\pa{}z(\alpha)}{\abs{\Delta z}^{2}})\varpi_{1}(\alpha-\beta)d\beta)\\
&=\pa{k-1}(\pa{}T_{1}(\varpi_{1})(\alpha)-T_{1}(\pa{}\varpi_{1})(\alpha))=\pa{k}T_{1}(\varpi_{1})(\alpha)-\pa{k-1}T_{1}(\pa{}\varpi_{1})(\alpha)
\end{align*}
and
\begin{displaymath}
R_{k}^{2}(\varpi_{2})=\pa{k}T_{2}(\varpi_{2})(\alpha)-\pa{k-1}T_{2}(\pa{}\varpi_{2})(\alpha).
\end{displaymath}
Taking the $k$-derivatives in (\ref{var2}), we get
\begin{displaymath}
\pa{k}\varpi_{2}(\alpha)+\lambda_{2}T_{3}(\pa{k}\varpi_{1})(\alpha)+\lambda_{2}T_{4}(\pa{k}\varpi_{2})(\alpha)=R_{k}^{3}(\varpi_{1})+R^{4}_{k}(\varpi{2})
\end{displaymath}
with
\begin{align*}
&R^{3}_{k}(\varpi_{1})=\pa{k}T_{3}(\varpi_{1})(\alpha)-\pa{k-1}T_{3}(\pa{}\varpi_{1})(\alpha),\\
&R^{4}_{k}(\varpi_{2})=\pa{k}T_{4}(\varpi_{2})(\alpha)-\pa{k-1}T_{4}(\pa{}\varpi_{2})(\alpha).
\end{align*}
Next let us consider
\begin{align}
\label{s1}
&\Lambda^{\frac{1}{2}}\pa{k}\varpi_{1}(\alpha)+\lambda_{1}T_{1}(\Lambda^{\frac{1}{2}}\pa{k}\varpi_{1})(\alpha)+\lambda_{1}T_{2}(\Lambda^{\frac{1}{2}}\pa{k}\varpi_{2})(\alpha)\\\nonumber
&=\lambda_{1}T_{1}(\Lambda^{\frac{1}{2}}\pa{k}\varpi_{1})(\alpha)+\lambda_{1}T_{2}(\Lambda^{\frac{1}{2}}\pa{k}\varpi_{2})(\alpha)-\lambda_{1}\Lambda^{\frac{1}{2}}T_{1}(\pa{k}\varpi_{1})(\alpha)-\lambda_{1}\Lambda^{\frac{1}{2}}T_{2}(\pa{k}\varpi_{2})(\alpha)\\
&+\Lambda^{\frac{1}{2}}R^{1}_{k}(\varpi_{1})(\alpha)+\Lambda^{\frac{1}{2}}R^{2}_{k}(\varpi_{2})(\alpha)-N\Lambda^{\frac{1}{2}}\pa{k+1}z_{2}(\alpha)\nonumber
\end{align}
and
\begin{align}
\label{s2}
&\Lambda^{\frac{1}{2}}\pa{k}\varpi_{2}(\alpha)+\lambda_{2}T_{3}(\Lambda^{\frac{1}{2}}\pa{k}\varpi_{1})(\alpha)+\lambda_{2}T_{4}(\Lambda^{\frac{1}{2}}\pa{k}\varpi_{2})(\alpha)\\\nonumber
&=\lambda_{2}T_{3}(\Lambda^{\frac{1}{2}}\pa{k}\varpi_{1})(\alpha)+\lambda_{2}T_{4}(\Lambda^{\frac{1}{2}}\pa{k}\varpi_{2})(\alpha)-\lambda_{2}\Lambda^{\frac{1}{2}}T_{3}(\pa{k}\varpi_{1})(\alpha)-\lambda_{2}\Lambda^{\frac{1}{2}}T_{4}(\pa{k}\varpi_{2})(\alpha)\\
&+\Lambda^{\frac{1}{2}}R^{3}_{k}(\varpi_{1})(\alpha)+\Lambda^{\frac{1}{2}}R^{4}_{k}(\varpi_{2})(\alpha).\nonumber
\end{align}
Then, we write
\begin{displaymath}
\left(\begin{matrix}
\Lambda^{\frac{1}{2}}\pa{k}\varpi_{1}\\
\Lambda^{\frac{1}{2}}\pa{k}\varpi_{2}
\end{matrix}\right)+\left(\begin{matrix}
\gamma_{1}&0\\
0&\gamma_{2}
\end{matrix}\right)\left(\begin{matrix}
T_{1}&T_{2}\\
T_{3}&T_{4}
\end{matrix}\right)\left(\begin{matrix}
\Lambda^{\frac{1}{2}}\pa{k}\varpi_{1}\\
\Lambda^{\frac{1}{2}}\pa{k}\varpi_{2}
\end{matrix}\right)=\left(\begin{matrix}
S_{1}\\
S_{2}
\end{matrix}\right)
\end{displaymath}
where $S_{1}$ is the right hand side of (\ref{s1}) and $S_{2}$ the right hand side of (\ref{s2}).
Using the estimate for the inverse $(I+M\OT)^{-1}$ in the space $H^{\frac{1}{2}}$  we get
\begin{displaymath}
\norm{\varpi}_{H^{k+1}}\le\norm{\Lambda^{\frac{1}{2}}\pa{k}\varpi}_{H^{\frac{1}{2}}}\le C\norm{(I+M\OT)^{-1}}_{H^{\frac{1}{2}}}\norm{S}_{H^{\frac{1}{2}}}\le e^{C\nor{z,h}^{2}}\norm{S}_{H^{\frac{1}{2}}}.
\end{displaymath}
We have that
\begin{displaymath}
S=\left(\begin{matrix}
S_{1}\\
S_{2}
\end{matrix}\right)=M\OT(\Lambda^{\frac{1}{2}}\pa{k}\varpi)-M\Lambda^{\frac{1}{2}}(\OT(\pa{k}\varpi))+\Lambda^{\frac{1}{2}}(\R_{k}(\varpi))-\left(\begin{matrix}
N\Lambda^{\frac{1}{2}}\pa{k+1}z_{2}(\alpha)\\
0
\end{matrix}\right)
\end{displaymath}
where $\R_{k}(\varpi)=\left(\begin{matrix}
R^{1}_{k}&R^{2}_{k}\\
R^{3}_{k}&R^{4}_{k}
\end{matrix}\right)\left(\begin{matrix}
\varpi_{1}\\
\varpi_{2}
\end{matrix}\right)$.
Thus,
\begin{displaymath}
\norm{S}_{H^{\frac{1}{2}}}\le C\norm{\OT(\Lambda^{\frac{1}{2}}\pa{k}\varpi)}_{H^{\frac{1}{2}}}+\norm{\OT(\pa{k}\varpi))}_{H^{1}}+\norm{\R_{k}(\varpi)}_{H^{1}}+\norm{z}_{H^{k+2}}.
\end{displaymath}
Using the lemma \ref{tl2h1},
\begin{displaymath}
\norm{\OT(\pa{k}\varpi)}_{H^{1}}\le C\norm{\F}_{L^{\infty}}^{2}\norm{\G}^{2}_{L^{\infty}}\norm{d(z,h)}^{2}_{L^{\infty}}\norm{z}_{\mathcal{C}^{2}}^{4}\norm{h}^{4}_{\mathcal{C}^{2}}\norm{\varpi}_{H^{k}}.
\end{displaymath}
Since $R^{i}_{k}(\varpi_{j})=\pa{k}T_{i}(\varpi_{j})(\alpha)-\pa{k-1}T_{i}(\pa{}\varpi_{j})(\alpha)$ for $i,j=1,2,3,4$ we can write
\begin{displaymath}
\R_{k}(\varpi)=\pa{k}\OT(\varpi)-\pa{k-1}\OT(\pa{}\varpi).
\end{displaymath}
Then using the lemma \ref{thkhk1} (proved below),
\begin{align*}
&\norm{\R_{k}(\varpi)}_{H^{1}}\le\norm{\pa{k}\OT(\varpi)}_{H^{1}}+\norm{\pa{k-1}\OT(\pa{}\varpi)}_{H^{1}}\le\norm{\OT(\varpi)}_{H^{k+1}}+\norm{\OT(\pa{}\varpi)}_{H^{k}}\\
&\le C\nor{z,h}^{2}(\norm{z}_{H^{k+2}}^{2}+\norm{h}_{H^{k+2}}^{2})\norm{\varpi}_{H^{k}}+C\nor{z,h}^{2}(\norm{z}_{H^{k+1}}^{2}+\norm{h}_{H^{k+1}}^{2})\norm{\pa{}\varpi}_{H^{k-1}}.
\end{align*}
Finally, using lemma \ref{tl2h1}
\begin{align*}
&\norm{\OT(\Lambda^{\frac{1}{2}}\pa{k}\varpi)}_{H^{\frac{1}{2}}}\le\norm{\OT(\Lambda^{\frac{1}{2}}\pa{k}\varpi)}_{H^{1}}\\
&\le C\norm{\F}^{2}_{L^{\infty}}\norm{\G}^{2}_{L^{\infty}}\norm{d(z,h)}^{2}_{L^{\infty}}\norm{z}^{4}_{\mathcal{C}^{2}}\norm{h}^{4}_{\mathcal{C}^{2}}\norm{\Lambda^{\frac{1}{2}}\pa{k}\varpi}_{L^{2}}.
\end{align*}
Since 
\begin{displaymath}
\pa{k}\varpi+M\OT\pa{k}\varpi=\R_{k}(\varpi)-\left(\begin{matrix}
N\pa{k+1}z_{2}\\
0
\end{matrix}\right),
\end{displaymath}
then
\begin{displaymath}
\pa{k}\varpi=(I+M\OT)^{-1}(\R_{k}(\varpi)-\left(\begin{matrix}
N\pa{k+1}z_{2}\\
0
\end{matrix}\right)).
\end{displaymath}
Therefore,
\begin{align*}
&\norm{\Lambda^{\frac{1}{2}}\pa{k}\varpi}_{L^{2}}=\norm{\pa{k}\varpi}_{H^{\frac{1}{2}}}\le\norm{(I+M\OT)^{-1}}_{H^{\frac{1}{2}}}(\norm{\R_{k}(\varpi)}_{H^{\frac{1}{2}}}+\norm{z}_{H^{k+\frac{3}{2}}})\\
&\le e^{C\nor{z,h}^{2}}(C\nor{z,h}^{2}(\norm{z}_{H^{k+2}}^{2}+\norm{h}^{2}_{H^{k+2}})\norm{\varpi}_{H^{k}}+\norm{z}_{H^{k+\frac{3}{2}}}).
\end{align*}
In conclusion,
\begin{displaymath}
\norm{\varpi}_{H^{k+1}}\le e^{C\nor{z,h}^{2}}(\norm{\F}^{2}_{L^{\infty}}\norm{d(z,h)}^{2}_{L^{\infty}}\norm{z}_{H^{k+2}}^{2}\norm{\varpi}_{H^{k}}+\norm{z}_{H^{k+2}}).
\end{displaymath}
For $k=\frac{1}{2}$, since $\norm{\varpi}_{H^{\frac{1}{2}}}\le e^{C\nor{z,h}^{2}}$ then
\begin{displaymath}
\norm{\varpi}_{H^{\frac{3}{2}}}\le \exp{C(\norm{\F}^{2}_{L^{\infty}}+\norm{d(z,h)}^{2}_{L^{\infty}}+\norm{z}^{2}_{H^{\frac{5}{2}}})}.
\end{displaymath}
Therefore using induction on $k\ge 2$ allows us to finish the proof.
\end{proof}
%%%%%%%%%%%%%%%%%%%%%%%%%%%%%%%%%%%%%%%%%%%%%%%%%%%%%%%%%%%%%%%
\begin{lem}
\label{thkhk1}
The operator $\OT$ maps Sobolev space $H^{k}\times H^{k}$, $k\ge 1$, into $H^{k+1}\times H^{k+1}$ as long as $z,h\in H^{k+2}$ and satisfies the estimate
\begin{displaymath}
\norm{\OT}_{H^{k}\times H^{k}\to H^{k+1}\times H^{k+1}}\le C\nor{z,h}^{2}\norm{z}^{2}_{H^{k+2}}
\end{displaymath}
\end{lem}
\begin{proof}
For the lemma $5.2$ in \cite{hele} we have
\begin{displaymath}
\norm{T_{1}(\varpi_{1})}_{H^{k+1}}\le C(\norm{\F}^{2}_{L^{\infty}}+\norm{z}^{2}_{H^{3}})\norm{z}^{2}_{H^{k+2}}\norm{\varpi_{1}}_{H^{k}}
\end{displaymath}
and changing $z$ for $h$ then
\begin{displaymath}
\norm{T_{4}(\varpi_{2})}_{H^{k+1}}\le C(\norm{\G}^{2}_{L^{\infty}}+\norm{h}^{2}_{H^{3}})\norm{h}^{2}_{H^{k+2}}\norm{\varpi_{2}}_{H^{k}}.
\end{displaymath}
Let us see what happens with $T_{2}(\varpi_{2})$. Taking the $k+1$-derivatives,
\begin{align*}
&\pa{k+1}T_{2}(\varpi_{2})(\alpha)=\sum_{j=0}^{k+1}\frac{C_{k}}{\pi}\int_{\R}\pa{k+1-j}(\frac{\Delta zh^{\bot}\cdot\pa{}z(\alpha)}{\abs{\Delta zh}^{2}})\pa{j}\varpi_{2}(\alpha-\beta)d\beta\\
&=T_{2}(\pa{k+1}\varpi_{2})(\alpha)+\frac{1}{\pi}\int_{\R}\pa{k+1-j}(\frac{\Delta zh^{\bot}\cdot\pa{}z(\alpha)}{\abs{\Delta zh}^{2}})\varpi_{2}(\alpha-\beta)d\beta\\
&+\sum_{j=1}^{k}\frac{C_{k}}{\pi}\int_{\R}\pa{k+1-j}(\frac{\Delta zh^{\bot}\cdot\pa{}z(\alpha)}{\abs{\Delta zh}^{2}})\pa{j}\varpi_{2}(\alpha-\beta)d\beta\\
&=T_{2}(\pa{k+1}\varpi_{2})(\alpha)+\frac{1}{\pi}\int_{\R}\frac{\Delta zh^{\bot}\cdot\pa{k+2}z(\alpha)}{\abs{\Delta zh}^{2}}\varpi_{2}(\alpha-\beta)d\beta\\
&+ \text{``other terms''}\\
&=T_{2}(\pa{k+1}\varpi_{2})(\alpha)+J_{1}+\text{``other terms''}
\end{align*}
The estimate for ``other terms'' is straighforward. For $T_{2}(\pa{k+1}\varpi_{2})(\alpha)$ we integrate by parts:
\begin{align*}
&T_{2}(\pa{k+1}\varpi_{2})(\alpha)=\frac{-1}{\pi}\int_{\R}\frac{\Delta zh^{\bot}\cdot\pa{}z(\alpha)}{\abs{\Delta zh}^{2}}\partial_{\beta}(\pa{k}\varpi_{2}(\alpha-\beta))d\beta\\
&=\frac{1}{\pi}\int_{\R}\partial_{\beta}(\frac{\Delta zh^{\bot}\cdot\pa{}z(\alpha)}{\abs{\Delta zh}^{2}})\pa{k}\varpi_{2}(\alpha-\beta)d\beta\\
&=-\frac{1}{\pi}\int_{\R}\frac{\pa{\bot}h(\alpha-\beta)\cdot\pa{}z(\alpha)}{\abs{\Delta zh}^{2}}\pa{k}\varpi_{2}(\alpha-\beta)d\beta\\
&+\frac{2}{\pi}\int_{\R}\frac{\Delta zh^{\bot}\cdot\pa{}z(\alpha)\Delta zh\cdot\pa{}h(\alpha-\beta)}{\abs{\Delta zh}^{4}}\pa{k}\varpi_{2}(\alpha-\beta)d\beta\equiv I_{1}+I_{2}.
\end{align*}
It is easy estimate $I_{1}$
\begin{displaymath}
\abs{I_{1}}\le C\norm{\pa{}z}_{L^{\infty}}\norm{d(z,h)}_{L^{\infty}}\norm{h}_{H^{1}}\norm{\varpi_{2}}_{H^{k}}.
\end{displaymath}
For $I_{2}$, using the Cauchy inequality
\begin{align*}
&\abs{I_{2}}\le \frac{1}{2\pi}\int_{\R}\frac{(\abs{\Delta zh}^{2}+\abs{\pa{}z(\alpha)}^{2})(\abs{\Delta zh}^{2}+\abs{\pa{}h(\alpha-\beta)}^{2})}{\abs{\Delta zh}^{4}}\pa{k}\varpi_{2}(\alpha-\beta)d\beta\\
&\le C\norm{\varpi_{2}}_{H^{k}}+C\norm{d(z,h)}_{L^{\infty}}\norm{z}_{\mathcal{C}^{1}}^{2}\norm{\varpi_{2}}_{H^{k}}+C\norm{d(z,h)}_{L^{\infty}}\norm{h}_{\mathcal{C}^{1}}^{2}\norm{\varpi_{2}}_{H^{k}}+C\norm{d(z,h)}^{2}_{L^{\infty}}\norm{z}_{\mathcal{C}^{1}}^{2}\norm{h}_{\mathcal{C}^{1}}^{2}\norm{\varpi_{2}}_{H^{k}}.
\end{align*}
If we use the same procedure to estimate $J_{1}$,
\begin{align*}
J_{1}\le\frac{1}{2\pi}\int_{\R}\frac{\abs{\Delta zh}^{2}+\abs{\pa{k+2}z(\alpha)}^{2}}{\abs{\Delta zh}^{2}}\varpi_{2}(\alpha-\beta)d\beta\le C\norm{\varpi_{2}}_{L^{2}}+C\norm{d(z,h)}_{L^{\infty}}\norm{z}_{H^{k+2}}^{2}\norm{\varpi_{2}}_{L^{2}}.
\end{align*}

Then,
\begin{align*}
\norm{T_{2}(\varpi_{2})}_{H^{k+1}}\le C\nor{z,h}^{2}\norm{z}^{2}_{H^{k+2}}\norm{\varpi_{2}}_{H^{k}}.
\end{align*}
Since $T_{3}(\varpi_{1})$ is $T_{2}(\varpi_{2})$ changing $z$ for $h$ and viceverse, the estimations will be
\begin{align*}
\norm{T_{3}(\varpi_{1})}_{H^{k+1}}\le C\nor{z,h}^{2}\norm{h}^{2}_{H^{k+2}}\norm{\varpi_{1}}_{H^{k}}.
\end{align*}
Therefore,
\begin{displaymath}
\norm{\OT\varpi}_{H^{k+1}}\le C\nor{z,h}^{2}(\norm{z}_{H^{k+2}}^{2}+\norm{h}_{H^{k+2}}^{2})\norm{\varpi}_{H^{k}}.
\end{displaymath}
Since $h$ is fixed on time, we get the desired estimate.
\end{proof}

%%%%%%%%%%%%%%%%%%%%%%%%%%%%%%%%%%%%%%%%%%%%%%%%%%%%%%%
\section{Estimates on $\brz{z}+\brh{z}+\brz{h}+\brh{h}$}\label{estibr}
This section is devoted to show that the Birkhoff-Rott integral is as regular as $\pa{}z$.
\begin{lem}
The following estimate holds
\begin{align}
\label{brzh}
&\norm{\brz{z}}_{H^{k}}+\norm{\brh{z}}_{H^{k}}+\norm{\brz{h}}_{H^{k}}+\norm{\brh{h}}_{H^{k}}\\\nonumber
&\le\esth{k+1}
\end{align}
for $k\ge 2$.
\end{lem}
\begin{proof}
The lemma $6.1$ on \cite{hele} gives us,
\begin{displaymath}
\norm{\brz{z}}_{H^{k}}\le\le\exp{C(\norm{\F}_{L^{\infty}}^{2}+\norm{z}^{2}_{H^{k+1}})}
\end{displaymath}
and
\begin{displaymath}
\norm{\brh{h}}_{H^{k}}\le\exp{C(\norm{\G}_{L^{\infty}}^{2}+\norm{h}^{2}_{H^{k+1}})}.
\end{displaymath}
Using that $\norm{h}^{2}_{H^{k+1}}$ and $\norm{\G}_{L^{\infty}}^{2}$ are not dependent of time,
\begin{align*}
&\norm{\brz{z}}_{H^{k}}+\norm{\brh{h}}_{H^{k}}\\
&\le\esth{k+1}.
\end{align*}
Let us see what happens with $\brz{h}$ and $\brh{z}$.
It is enough study one of then. For example let study $\brh{z}$.
For $k=2$,
\begin{align*}
&\norm{\brh{z}}_{L^{2}}\le C\norm{d(z,h)}_{L^{\infty}}(\norm{z}_{L^{2}}+\norm{h}_{L^{2}})\norm{\varpi_{2}}_{L^{2}}\\
&\le\esth{1}.
\end{align*} 
If we take two derivatives, we get $\brh{z}=B_{1}+B_{2}+B_{3}+\text{``other terms''}$ where
\begin{align*}
&B_{1}=\frac{1}{2\pi}\int_{\R}\frac{(\Delta zh)^{\bot}}{\abs{\Delta zh}^{2}}\pa{2}\varpi_{2}(\alpha-\beta)d\beta,\\
&B_{2}=\frac{1}{2\pi}\int_{\R}\frac{(\pa{2}z(\alpha)-\pa{2}h(\alpha-\beta))^{\bot}}{\abs{\Delta zh}^{2}}\varpi_{2}(\alpha-\beta)d\beta,\\
&B_{3}=-\frac{1}{\pi}\int_{\R}\frac{(\Delta zh)^{\bot}(\Delta zh\cdot(\pa{2}z(\alpha)-\pa{2}h(\alpha-\beta)))}{\abs{\Delta zh}^{4}}\varpi_{2}(\alpha-\beta)d\beta.
\end{align*}
Using the estimations in $\varpi$ and the distance of $z$ and $h$,
\begin{align*}
&\norm{B_{1}}_{L^{2}}\le C\norm{d(z,h)}_{L^{\infty}}^{\frac{1}{2}}\norm{\varpi_{2}}_{H^{2}}\\
&\le\esth{3}.
\end{align*}
For $B_{2}$,
\begin{align*}
&\norm{B_{2}}_{L^{2}}\le C\norm{d(z,h)}_{L^{\infty}}(\norm{z}_{H^{2}}+\norm{h}_{H^{2}})\norm{\varpi_{2}}_{L^{2}}\\
&\le\esth{2}.
\end{align*}
And $B_{3}$ will be the same,
\begin{align*}
&\norm{B_{3}}_{L^{2}}\le C\norm{d(z,h)}_{L^{\infty}}(\norm{z}_{H^{2}}+\norm{h}_{H^{2}})\norm{\varpi_{2}}_{L^{2}}\\
&\le\esth{2}.
\end{align*}
These estimations allow us to get the desire result.
\end{proof}

\section{A priori estimates on $z(\alpha,t)$}\label{estimacionesz}
In this section we want to give the following a priori estimates
\begin{align*}
&\frac{d}{dt}\norm{z}_{H^{k}}^{2}\le\esth{k}\\
&-\frac{\kappa^{1}}{2\pi(\mu^{1})+\mu^{2}}\int_{\T}\frac{\sigma(\alpha)}{A(t)}\pa{k}z(\alpha)\cdot\Lambda(\pa{k}z)(\alpha)d\alpha
\end{align*} 
for $k\ge 3$.

To do that, we split the computations in four subsections and we will do it for $k=3$. The case $k>3$ is left to the reader.

\subsection{Estimates for the $L^{2}$ norm of the curve}\label{estimal2}
We have
\begin{displaymath}
z_{t}(\alpha)=\brz{z}+\brh{z}+c(\alpha)\pa{}z(\alpha)
\end{displaymath}
where
\begin{align*}
c(\alpha)=&\frac{\alpha+\pi}{2\pi A(t)}\int_{\T}\pa{}z(\beta)\cdot(\pa{}\brz{z}+\pa{}\brh{z})d\beta\\
&-\int_{\pi}^{\alpha}\frac{\pa{}z(\beta)}{A(t)}\cdot(\pa{}\brz{z}+\pa{}\brh{z})d\beta.
\end{align*}
Recall that $A(t)=\abs{\pa{}z(\alpha)}^{2}$.
Then,
\begin{align*}
&\frac{1}{2}\frac{d}{dt}\int_{\T}\abs{z(\alpha)}^{2}d\alpha=\int_{\T}z(\alpha)\cdot z_{t}(\alpha)d\alpha=\int_{\T}z(\alpha)\cdot\brz{z}d\alpha\\
&+\int_{\T}z(\alpha)\cdot\brh{z}d\alpha+\int_{\T}c(\alpha)z(\alpha)\cdot\pa{}z(\alpha)d\alpha\equiv I_{1}+I_{2}+I_{3}.
\end{align*}
Taking $I_{1}+I_{2}\le\norm{z}_{L^{2}}(\norm{\brz{z}}_{L^{2}}+\norm{\brh{z}}_{L^{2}})$ and the inequality (\ref{brzh}) allow us to write,
\begin{displaymath}
I_{1}+I_{2}\le\esth{1}.
\end{displaymath}
Next we get,
\begin{align*}
I_{3}&\le A^{\frac{1}{2}}(t)\norm{c}_{L^{\infty}}\int_{\T}\abs{z(\alpha)}d\alpha\le 2\int_{\T}\abs{\brz{z}}+\abs{\brh{z}}d\alpha\int_{\T}\abs{z(\alpha)}d\alpha\\
&\le\esth{1}.
\end{align*}
Therefore,
\begin{displaymath}
\frac{d}{dt}\norm{z}_{L^{2}}^{2}(t)\le\esth{3}.
\end{displaymath}

%%%%%%%%%%%%%%%%%%%%%%%%%%%%%%%%%%%%%%%%%%%%%%%%%%%%%%%%%%%%%
\subsection{Estimates on the $H^{3}$ norm}

Taking the $3$ derivatives on the curve, we get
\begin{align*}
&\int_{\T}\pa{3}z(\alpha)\cdot\pa{3}z_{t}(\alpha)d\alpha=\int_{\T}\pa{3}z(\alpha)\cdot\pa{3}\brz{z}d\alpha\\
&+\int_{\T}\pa{3}z(\alpha)\cdot\pa{3}\brh{z}d\alpha+\int_{\T}\pa{3}z(\alpha)\cdot\pa{3}(c(\alpha)\pa{}z(\alpha))d\alpha\\
&\equiv I_{1}+I_{2}+I_{3}.
\end{align*}
Here and in the next section we will study $I_{1}+I_{2}$. We shall estimate $I_{3}$ in section \ref{estimacionesc}. Let estimate first the term $I_{2}$. We can split $I_{2}=J_{1}+J_{2}+J_{3}+J_{4}$,  where
\begin{align*}
&J_{1}=\frac{1}{2\pi}\int_{\T}\int_{\R}\pa{3}z(\alpha)\cdot\pa{3}(\frac{(z(\alpha)-h(\alpha-\beta))^{\bot}}{\abs{z(\alpha)-h(\alpha-\beta)}^{2}})\varpi_{2}(\alpha-\beta)d\beta d\alpha,\\
&J_{2}=\frac{3}{2\pi}\int_{\T}\int_{\R}\pa{3}z(\alpha)\cdot\pa{2}(\frac{(z(\alpha)-h(\alpha-\beta))^{\bot}}{\abs{z(\alpha)-h(\alpha-\beta)}^{2}})\pa{}\varpi_{2}(\alpha-\beta)d\beta d\alpha,\\
&J_{3}=\frac{3}{2\pi}\int_{\T}\int_{\R}\pa{3}z(\alpha)\cdot\pa{}(\frac{(z(\alpha)-h(\alpha-\beta))^{\bot}}{\abs{z(\alpha)-h(\alpha-\beta)}^{2}})\pa{2}\varpi_{2}(\alpha-\beta)d\beta d\alpha,\\
&J_{4}=\frac{1}{2\pi}\int_{\T}\int_{\R}\pa{3}z(\alpha)\cdot\frac{(z(\alpha)-h(\alpha-\beta))^{\bot}}{\abs{z(\alpha)-h(\alpha-\beta)}^{2}}\pa{3}\varpi_{2}(\alpha-\beta)d\beta d\alpha.\\
\end{align*} 
 The most singular terms for $J_{1}$ are:

\begin{align*}
&J_{1}^{1}=\frac{1}{4\pi}\int_{\T}\int_{\R}\pa{3}z(\alpha)\cdot(\frac{(\pa{3}z(\alpha)-\pa{3}h(\alpha-\beta))^{\bot}}{\abs{z(\alpha)-h(\alpha-\beta)}^{2}})\varpi_{2}(\alpha-\beta)d\beta d\alpha,\\
&J_{1}^{2}=-\frac{1}{2\pi}\int_{\T}\int_{\R}\pa{3}z(\alpha)\cdot(\frac{(\Delta zh)^{\bot}\Delta zh\cdot(\pa{3}z(\alpha)-\pa{3}h(\alpha-\beta))}{\abs{z(\alpha)-h(\alpha-\beta)}^{4}})\varpi_{2}(\alpha-\beta)d\beta d\alpha.\\
\end{align*}
Using $\pa{3}z\cdot\pa{3}z^{\bot}=0$,
\begin{displaymath}
\abs{J_{1}^{1}}\le C\norm{d(z,h)}_{L^{\infty}}\norm{z}_{H^{3}}\norm{h}_{H^{3}}\norm{\varpi_{2}}_{L^{\infty}}.
\end{displaymath}
Using the same technique,
\begin{align*}
&\abs{J_{1}^{2}}\le C\norm{d(z,h)}_{L^{\infty}}\norm{z}_{H^{3}}(\norm{z}_{H^{3}}+\norm{h}_{H^{3}})\norm{\varpi_{2}}_{L^{\infty}}.
\end{align*}

Then,
\begin{displaymath}
J_{1}\le\esth{3}.
\end{displaymath}
The most singular term in $J_{2}$ is:
\begin{align*}
&J_{2}^{1}=C\int_{\T}\int_{\R}\pa{3}z(\alpha)\cdot\frac{(\pa{2}z(\alpha)-\pa{2}h(\alpha-\beta))^{\bot}}{\abs{\Delta zh}^{2}}\pa{}\varpi_{2}(\alpha-\beta)d\beta d\beta\\
&\le C\norm{d(z,h)}_{L^{\infty}}\norm{z}_{H^{3}}(\norm{z}_{\mathcal{C}^{2}}+\norm{h}_{\mathcal{C}^{2}})\norm{\varpi_{2}}_{H^{1}}.
\end{align*}
Then,
\begin{displaymath}
J_{2}\le\esth{3}.
\end{displaymath}
For $J_{3}$,
\begin{align*}
&J_{3}^{1}=C\int_{\T}\int_{\R}\pa{3}z(\alpha)\cdot\frac{(\pa{}z(\alpha)-\pa{}h(\alpha-\beta))^{\bot}}{\abs{\Delta zh}^{2}}\pa{2}\varpi_{2}(\alpha-\beta)d\beta d\alpha\\
&\le C\norm{d(z,h)}_{L^{\infty}}(\norm{z}_{\mathcal{C}^{1}}+\norm{h}_{\mathcal{C}^{1}})\norm{z}_{H^{3}}\norm{\varpi_{2}}_{H^{2}}\\
&\le\esth{3}.
\end{align*}
Using integration by parts we will estimate $J_{4}$,
\begin{align*}
&J_{4}=-\frac{1}{2\pi}\int_{\T}\int_{\R}\pa{3}z(\alpha)\cdot\frac{(\Delta zh)^{\bot}}{\abs{\Delta zh}^{2}}\partial_{\beta}\pa{2}\varpi_{2}(\alpha-\beta)d\beta d\alpha\\
&=-\frac{1}{2\pi}\int_{\T}\int_{\R}\pa{3}z(\alpha)\cdot\frac{\pa{\bot}h(\alpha-\beta)}{\abs{\Delta zh}^{2}}\pa{2}\varpi_{2}(\alpha-\beta)d\beta d\alpha\\
&+\frac{1}{\pi}\int_{\T}\int_{\R}\pa{3}z(\alpha)\cdot\frac{(\Delta zh)^{\bot}\Delta zh\cdot\pa{}h(\alpha-\beta)}{\abs{\Delta zh}^{4}}\partial_{\beta}\pa{2}\varpi_{2}(\alpha-\beta)d\beta d\alpha\equiv J_{4}^{1}+J_{4}^{2}.\\
\end{align*}
It is clear,
\begin{align*}
&J_{4}^{1}\le C\norm{d(z,h)}_{L^{\infty}}\norm{h}_{\mathcal{C}^{1}}\norm{z}_{H^{3}}\norm{\varpi_{2}}_{H^{2}},\\
&J_{4}^{2}\le C\norm{d(z,h)}_{L^{\infty}}\norm{h}_{\mathcal{C}^{1}}\norm{z}_{H^{3}}\norm{\varpi_{2}}_{H^{2}}.
\end{align*}
Therefore,
\begin{displaymath}
I_{2}\le\esth{3}.
\end{displaymath}

%%%%%%%%%%%%%%%%%%%%%%%%%%%%%%%%%%%%%%%%%%%%%%%%%%%%%%%%
\subsection{Estimations on $I_{1}$}\label{estimini7}
We can split $I_{1}$ in the following terms:
\begin{align*}
&I_{1}^{1}=\frac{1}{2\pi}\int_{\T}\int_{\R}\pa{3}z(\alpha)\cdot\pa{3}(\frac{(z(\alpha)-z(\alpha-\beta))^{\bot}}{\abs{z(\alpha)-z(\alpha-\beta)}^{2}})\varpi_{1}(\alpha-\beta)d\beta d\alpha,\\
&I_{1}^{2}=\frac{3}{2\pi}\int_{\T}\int_{\R}\pa{3}z(\alpha)\cdot\pa{2}(\frac{(z(\alpha)-z(\alpha-\beta))^{\bot}}{\abs{z(\alpha)-z(\alpha-\beta)}^{2}})\pa{}\varpi_{1}(\alpha-\beta)d\beta d\alpha,\\
&I_{1}^{3}=\frac{3}{2\pi}\int_{\T}\int_{\R}\pa{3}z(\alpha)\cdot\pa{}(\frac{(z(\alpha)-z(\alpha-\beta))^{\bot}}{\abs{z(\alpha)-z(\alpha-\beta)}^{2}})\pa{2}\varpi_{1}(\alpha-\beta)d\beta d\alpha,\\
&I_{1}^{4}=\frac{1}{2\pi}\int_{\T}\int_{\R}\pa{3}z(\alpha)\cdot\frac{(z(\alpha)-z(\alpha-\beta))^{\bot}}{\abs{z(\alpha)-z(\alpha-\beta)}^{2}}\pa{3}\varpi_{1}(\alpha-\beta)d\beta d\alpha.\\
\end{align*}
The terms $I_{1}^{1}$, $I_{1}^{2}$ and $I_{1}^{3}$ can be estimated like in the section $7.2$ in \cite{hele}. Then we have to estimate $I_{1}^{4}$.
\begin{align*}
&I_{1}^{4}=\frac{1}{2\pi}\int_{\T}\int_{\R}\pa{3}z(\alpha)\cdot(\frac{(\Delta z)^{\bot}}{\abs{\Delta z}^{2}}-\frac{\pa{\bot}z(\alpha)}{\beta\abs{\pa{}z(\alpha)}^{2}})\pa{3}\varpi_{1}(\alpha-\beta)d\beta d\alpha\\
&+\frac{1}{2\pi}\int_{\T}\int_{\R}\pa{3}z(\alpha)\cdot(\frac{\pa{\bot}z(\alpha)}{\beta\abs{\pa{}z(\alpha)}^{2}})\pa{3}\varpi_{1}(\alpha-\beta)d\beta d\alpha\equiv I_{1}^{41}+I_{1}^{42}.\\
\end{align*}
Using integration by parts,
\begin{align*}
&I_{1}^{41}=-\frac{1}{2\pi}\int_{\T}\int_{\R}\pa{3}z(\alpha)\cdot(\frac{(\Delta z)^{\bot}}{\abs{\Delta z}^{2}}-\frac{\pa{\bot}z(\alpha)}{\beta\abs{\pa{}z(\alpha)}^{2}})\partial_{\beta}\pa{2}\varpi_{1}(\alpha-\beta)d\beta d\alpha\\
&=\frac{1}{2\pi}\int_{\T}\int_{\R}\pa{3}z(\alpha)\cdot\partial_{\beta}(\frac{(\Delta z)^{\bot}}{\abs{\Delta z}^{2}}-\frac{\pa{\bot}z(\alpha)}{\beta\abs{\pa{}z(\alpha)}^{2}})\pa{2}\varpi_{1}(\alpha-\beta)d\beta d\alpha.
\end{align*}
If we decompose,
\begin{align}
&\partial_{\beta}(\frac{(\Delta z)^{\bot}}{\abs{\Delta z}^{2}}-\frac{\partial^{\bot}_{\alpha}z(\alpha)}{\abs{\partial_{\alpha}z(\alpha)}^{2}\beta})\nonumber\\
&=\frac{(\Delta\partial_{\alpha}z)^{\bot}}{\abs{\Delta z}^{2}}+\partial^{\bot}_{\alpha}z(\alpha)(\frac{1}{\abs{\Delta z}^{2}}-\frac{1}{\abs{\partial_{\alpha}z(\alpha)}^{2}\beta^{2}})-2\frac{(\Delta z)^{\bot}\Delta z\cdot\Delta\partial_{\alpha}z}{\abs{\Delta z}^{4}}\nonumber\\
&-2\frac{(\Delta z)^{\bot}(\Delta z-\beta\partial_{\alpha}z(\alpha))\cdot\partial_{\alpha}z(\alpha)}{\abs{\Delta z}^{4}}-2\frac{(\Delta z-\beta\partial_{\alpha}z(\alpha))^{\bot}\beta\abs{\partial_{\alpha}z(\alpha)}^{2}}{\abs{\Delta z}^{4}}\nonumber\\
&+(\frac{2\partial^{\bot}_{\alpha}z(\alpha)}{\abs{\partial_{\alpha}z(\alpha)}^{2}\beta^{2}}-\frac{2\beta^{2}\partial^{\bot}_{\alpha}z(\alpha)\abs{\partial_{\alpha}z(\alpha)}^{2}}{\abs{\Delta z}^{4}})\nonumber\\
&\equiv F_{1}(\alpha,\beta)+F_{2}(\alpha,\beta)+F_{3}(\alpha,\beta)+F_{4}(\alpha,\beta)+F_{5}(\alpha,\beta)+F_{6}(\alpha,\beta).\label{descomposicion}
\end{align}
We have
\begin{align*}
I_{1}^{41}&=\frac{1}{2\pi}\int_{\T}\int_{\R}\pa{3}z(\alpha)\cdot F_{1}(\alpha,\beta)\pa{2}\varpi_{1}(\alpha-\beta)d\beta d\alpha\\
&+\frac{1}{2\pi}\int_{\T}\int_{\R}\pa{3}z(\alpha)\cdot F_{2}(\alpha,\beta)\pa{2}\varpi_{1}(\alpha-\beta)d\beta d\alpha\\
&+\frac{1}{2\pi}\int_{\T}\int_{\R}\pa{3}z(\alpha)\cdot F_{3}(\alpha,\beta)\pa{2}\varpi_{1}(\alpha-\beta)d\beta d\alpha\\
&+\frac{1}{2\pi}\int_{\T}\int_{\R}\pa{3}z(\alpha)\cdot F_{4}(\alpha,\beta)\pa{2}\varpi_{1}(\alpha-\beta)d\beta d\alpha\\
&+\frac{1}{2\pi}\int_{\T}\int_{\R}\pa{3}z(\alpha)\cdot F_{5}(\alpha,\beta)\pa{2}\varpi_{1}(\alpha-\beta)d\beta d\alpha\\
&+\frac{1}{2\pi}\int_{\T}\int_{\R}\pa{3}z(\alpha)\cdot F_{6}(\alpha,\beta)\pa{2}\varpi_{1}(\alpha-\beta)d\beta d\alpha\\
&\equiv I_{1}^{411}+I_{1}^{412}+I_{1}^{413}+I_{1}^{414}+I_{1}^{415}+I_{1}^{416}.
\end{align*}
Computing
\begin{align*}
&F_{1}(\alpha,\beta)-\frac{\pa{2}z(\alpha)^{\bot}}{\beta\abs{\pa{}z(\alpha)}^{2}}=\frac{\beta^{2}\int_{0}^{1}\pa{2}z(\alpha-\beta ts)^{\bot}-\pa{2}z(\alpha)^{\bot}dsdt}{\abs{\Delta z}^{2}}\\
&+\frac{\beta^{2}\pa{2}z(\alpha)^{\bot}\int_{0}^{1}\int_{0}^{1}(1-t)\pa{2}z(\alpha)dsdt\cdot\int_{0}^{1}\pa{}z(\alpha)+\pa{}z(\alpha-\beta+\beta t)dt}{\abs{\Delta z}^{2}\abs{\pa{}z(\alpha)}^{2}}
\end{align*}
where $\alpha=\alpha-\beta+\beta t+s\beta+\beta ts$, we get
\begin{align*}
&I_{1}^{411}\le C\norm{\F}_{L^{\infty}}\norm{z}_{\mathcal{C}^{2,\delta}}\norm{z}_{H^{3}}\norm{\varpi_{1}}_{H^{2}}+C\norm{\F}^{\frac{3}{2}}_{L^{\infty}}\norm{z}_{\mathcal{C}^{2}}\norm{z}_{H^{3}}\norm{\varpi_{1}}_{H^{2}}\\
&+\frac{1}{2\pi}\int_{\T}\pa{3}z(\alpha)\cdot\frac{\pa{2}z(\alpha)^{\bot}}{\abs{\pa{}z(\alpha)}^{2}}H(\pa{2}\varpi_{1})(\alpha)d\alpha.\\
\end{align*}
Then,
\begin{displaymath}
I_{1}^{411}\le\esth{3}.
\end{displaymath}
Analogously,
\begin{align*}
I_{1}^{412}\le C\norm{\F}^{\frac{3}{2}}_{L^{\infty}}\norm{z}_{\mathcal{C}^{2}}^{2}\norm{z}_{H^{3}}\norm{\varpi_{1}}_{H^{2}}+\frac{1}{\pi}\int_{\T}\pa{3}z(\alpha)\cdot\pa{\bot}z(\alpha)\frac{\pa{2}z(\alpha)\cdot\pa{}z(\alpha)}{\abs{\pa{}z(\alpha)}^{2}}H(\pa{2}\varpi_{1})(\alpha)d\alpha.
\end{align*}
Using the fact that we can split $F_{3}$, with $\phi=\alpha-\beta+\beta t$ and $\psi=\alpha-\beta+\beta t+s\beta-\beta ts$.
\begin{align*}
&F_{3}(\alpha,\beta)=-2\beta^{3}\int_{0}^{1}\pa{\bot}z(\phi)dt\int_{0}^{1}\pa{}z(\phi)dt\cdot\int_{0}^{1}\pa{2}z(\phi)dt(\frac{1}{\abs{\Delta z}^{4}}-\frac{1}{\beta^{4}\abs{\pa{}z(\alpha)}^{4}})\\
&-\frac{2\int_{0}^{1}\pa{\bot}z(\phi)dt\int_{0}^{1}\pa{}z(\phi)dt\cdot\int_{0}^{1}\pa{2}z(\phi)dt}{\beta\abs{\pa{}z(\alpha)}^{4}}
\end{align*}
and
\begin{align*}
&\frac{1}{\abs{\Delta z}^{4}}-\frac{1}{\beta^{4}\abs{\pa{}z(\alpha)}^{4}}\\
&=\frac{\beta\int_{0}^{1}\int_{0}^{1}\pa{2}z(\psi)(t-1)dtds\cdot\int_{0}^{1}\pa{}z(\alpha)+\pa{}z(\phi)dt\int_{0}^{1}\abs{\pa{}z(\alpha)}^{2}+\abs{\pa{}z(\phi)}^{2}}{\abs{\Delta z}^{4}\abs{\pa{}z(\alpha)}^{4}},
\end{align*}
we get,
\begin{align*}
&I_{1}^{413}\le\esth{3}\\
&-\frac{1}{\pi}\int_{\T}\pa{3}z(\alpha)\cdot\pa{\bot}z(\alpha)\frac{\pa{}z(\alpha)\cdot\pa{2}z(\alpha)}{\abs{\pa{}z(\alpha)}^{4}}H(\pa{2}\varpi_{1})(\alpha)d\alpha\\
&\le\esth{3}+C\norm{\F}_{L^{\infty}}\norm{z}_{\mathcal{C}^{2}}\norm{z}_{H^{3}}\norm{\varpi_{1}}_{H^{2}}.
\end{align*}
For $I_{1}^{414}$ and $I_{1}^{415}$ it is easy to see that it is bounded in the same way as the above terms. Let us study the term $I_{1}^{416}$.
 Since,
 \begin{align*}
 &-\frac{1}{2}F_{6}(\alpha,\beta)=\partial^{\bot}_{\alpha}z(\alpha)\frac{\beta^{4}\abs{\partial_{\alpha}z(\alpha)}^{4}-\abs{\Delta z}^{4}}{\abs{\Delta z}^{4}\abs{\partial_{\alpha}z(\alpha)}^{2}\beta^{2}}\\
 &=\partial^{\bot}_{\alpha}z(\alpha)\frac{\beta^{3}\int_{0}^{1}\int_{0}^{1}(\partial^{2}_{\alpha}z(\psi)-\partial^{2}_{\alpha}z(\alpha))(1-s)dtds\int^{1}_{0}[\partial_{\alpha}z(\alpha)+\partial_{\alpha}z(\phi)]ds\int^{1}_{0}[\abs{\partial_{\alpha}z(\alpha)}^{2}+\abs{\partial_{\alpha}z(\phi)}^{2}]ds}{\abs{\Delta z}^{4}\abs{\partial_{\alpha}z(\alpha)}^{2}}\\
 &+\frac{\partial^{\bot}_{\alpha}z(\alpha)}{2}\frac{\beta^{4}\partial^{2}_{\alpha}z(\alpha)\int_{0}^{1}\int_{0}^{1}\partial^{2}_{\alpha}z(\eta)(s-1)dtds\int^{1}_{0}[\abs{\partial_{\alpha}z(\alpha)}^{2}+\abs{\partial_{\alpha}z(\phi)}^{2}]ds}{\abs{\Delta z}^{4}\abs{\partial_{\alpha}z(\alpha)}^{2}}\\
 &+\partial^{\bot}_{\alpha}z(\alpha)\frac{\beta^{4}\partial^{2}_{\alpha}z(\alpha)\partial_{\alpha}z(\alpha)\int_{0}^{1}\int_{0}^{1}\partial_{\alpha}z(\eta)\cdot\partial^{2}_{\alpha}z(\eta)(s-1)dtds}{\abs{\Delta z}^{4}\abs{\partial_{\alpha}z(\alpha)}^{2}}+\partial^{\bot}_{\alpha}z(\alpha)\frac{\beta^{3}\partial^{2}_{\alpha}z(\alpha)\partial_{\alpha}z(\alpha)}{\abs{\Delta z}^{4}}\\
 &\equiv U_{1}(\alpha,\beta)+U_{2}(\alpha,\beta)+U_{3}(\alpha,\beta)+U_{4}(\alpha,\beta)
 \end{align*}
 we get,
 \begin{align*}
 I_{1}^{416}=&-\frac{1}{\pi}\int_{\T}\int_{\R}\pa{3}z(\alpha)\cdot U_{1}(\alpha,\beta)\partial^{2}_{\alpha}\varpi(\alpha-\beta)d\alpha d\beta-\frac{1}{\pi}\int_{\T}\int_{\R}\pa{3}z(\alpha)\cdot U_{2}(\alpha,\beta)\partial^{2}_{\alpha}\varpi(\alpha-\beta)d\alpha d\beta\\
 &-\frac{1}{\pi}\int_{\T}\int_{\R}\pa{3}z(\alpha)\cdot U_{3}(\alpha,\beta)\partial^{2}_{\alpha}\varpi(\alpha-\beta)d\alpha d\beta-\frac{1}{\pi}\int_{\T}\int_{\R}\pa{3}z(\alpha)\cdot U_{4}(\alpha,\beta)\partial^{2}_{\alpha}\varpi(\alpha-\beta)d\alpha d\beta\\
 &\equiv Q_{1}+Q_{2}+Q_{3}+Q_{4}.
 \end{align*}
 It is clear,
 \begin{align*}
 &Q_{1}\le C\norm{\F}_{L^{\infty}}\norm{z}_{\mathcal{C}^{2,\delta}}\norm{z}_{H^{3}}\norm{\varpi_{1}}_{H^{2}},\\
 &Q_{2}\le C\norm{\F}^{2}_{L^{\infty}}\norm{z}_{\mathcal{C}^{2}}^{2}\norm{z}_{\mathcal{C}^{1}}\norm{z}_{H^{3}}\norm{\varpi_{1}}_{H^{2}},\\
 &Q_{3}\le C\norm{\F}^{2}_{L^{\infty}}\norm{z}_{\mathcal{C}^{2}}^{2}\norm{z}_{\mathcal{C}^{1}}\norm{z}_{H^{3}}\norm{\varpi_{1}}_{H^{2}},\\
 &Q_{4}\le C\norm{\F}^{2}_{L^{\infty}}\norm{z}_{\mathcal{C}^{2}}^{2}\norm{z}_{\mathcal{C}^{1}}\norm{z}_{H^{3}}\norm{\varpi_{1}}_{H^{2}}\\
 &-\frac{1}{\pi}\int_{\T}\pa{3}z(\alpha)\cdot\pa{\bot}z(\alpha)\frac{\pa{2}z(\alpha)\cdot\pa{}z(\alpha)}{\abs{\pa{}z(\alpha)}^{4}}H(\pa{2}\varpi_{1})(\alpha)d\alpha.\\
 \end{align*}
 Therefore,
 \begin{displaymath}
 I_{1}^{41}\le\esth{3}.
 \end{displaymath}
 Now, we have to study $I_{1}^{42}$. Using $\pa{}H=\Lambda$,
\begin{align*}
&I_{1}^{42}=\frac{1}{2\pi}\int_{\T}\pa{3}z(\alpha)\cdot\frac{\pa{\bot}z(\alpha)}{\abs{\pa{}z(\alpha)}^{2}}H(\pa{3}\varpi_{1})(\alpha)d\alpha=\frac{1}{2\pi}\int_{\T}\pa{3}z(\alpha)\cdot\frac{\pa{\bot}z(\alpha)}{\abs{\pa{}z(\alpha)}^{2}}\Lambda(\pa{2}\varpi_{1})(\alpha)d\alpha\\
&=\frac{1}{2\pi A(t)}\int_{\T}\Lambda(\pa{3}z(\alpha)\cdot\pa{\bot}z(\alpha))\pa{2}\varpi_{1}(\alpha)d\alpha.\\
\end{align*}
Since $\varpi_{1}=-\gamma_{1}T_{1}(\varpi_{1})-\gamma_{1}T_{2}(\varpi_{2})-N\pa{}z_{2}(\alpha)$ we split $I_{1}^{42}$,
\begin{align*}
&I_{1}^{421}=\frac{-N}{2\pi A(t)}\int_{\T}\Lambda(\pa{3}z(\alpha)\cdot\pa{\bot}z(\alpha))\pa{3}z_{2}(\alpha)d\alpha,\\
&I_{1}^{422}=\frac{-\gamma_{1}}{2\pi A(t)}\int_{\T}\Lambda(\pa{3}z\cdot\pa{\bot}z)(\alpha)(\pa{2}T_{1}(\varpi_{1})+\pa{2}T_{2}(\varpi_{2}))d\alpha.\\
\end{align*}
We can write $I_{1}^{421}=L_{1}+L_{2}$ where
\begin{align*}
&L_{1}=\frac{N}{2\pi A(t)}\int_{\T}\Lambda(\pa{3}z_{1}\pa{}z_{2})(\alpha)\pa{3}z_{2}(\alpha)d\alpha,\\
&L_{2}=\frac{-N}{2\pi A(t)}\int_{\T}\Lambda(\pa{3}z_{2}\pa{}z_{1})(\alpha)\pa{3}z_{2}(\alpha)d\alpha.\\
\end{align*}
Use the commutator estimation allow us,
\begin{align*}
&L_{1}\le C\norm{z}_{\mathcal{C}^{2,\delta}}\norm{z}_{H^{3}}^{2}+\frac{N}{2\pi A(t)}\int_{\T}\pa{}z_{2}(\alpha)\pa{3}z_{2}(\alpha)\Lambda(\pa{3}z_{1})(\alpha)d\alpha\\
&\le\esth{3}+L_{1}^{1}.
\end{align*}
Since $A(t)=\abs{\pa{}z(\alpha)}^{2}$ if we derivate twice with $\pa{}$ we get
\begin{displaymath}
\pa{}z_{2}(\alpha)\pa{3}z_{2}(\alpha)=-\pa{}z_{1}(\alpha)\pa{3}z_{1}(\alpha)-\abs{\pa{2}z(\alpha)}^{2}.
\end{displaymath}
Then,
\begin{align*}
&L_{1}^{1}=-\frac{N}{2\pi A(t)}\int_{\T}\abs{\pa{2}z(\alpha)}^{2}\Lambda(\pa{3}z_{1})(\alpha)d\alpha-\frac{N}{2\pi A(t)}\int_{\T}\pa{}z_{1}(\alpha)\pa{3}z_{1}(\alpha)\Lambda(\pa{3}z_{1})(\alpha)d\alpha\\
&=\frac{N}{2\pi A(t)}\int_{\T}\pa{}\abs{\pa{2}z(\alpha)}^{2}H(\pa{3}z_{1})(\alpha)d\alpha-\frac{N}{2\pi A(t)}\int_{\T}\pa{}z_{1}(\alpha)\pa{3}z_{1}(\alpha)\Lambda(\pa{3}z_{1})(\alpha)d\alpha\\
&\le C\norm{\F}_{L^{\infty}}\norm{z}_{\mathcal{C}^{2}}\norm{z}_{H^{3}}^{2}-\frac{N}{2\pi A(t)}\int_{\T}\pa{}z_{1}(\alpha)\pa{3}z_{1}(\alpha)\Lambda(\pa{3}z_{1})(\alpha)d\alpha.\\
\end{align*}
In the same way, using the commutator estimation we have,
\begin{displaymath}
L_{2}\le C\norm{\F}_{L^{\infty}}\norm{z}_{\mathcal{C}^{2,\delta}}\norm{z}_{H^{3}}^{2}-\frac{N}{2\pi A(t)}\int_{\T}\pa{}z_{1}(\alpha)\pa{3}z_{2}(\alpha)\Lambda(\pa{3}z_{2})(\alpha)d\alpha.
\end{displaymath}
Therefore,
\begin{align*}
&I_{1}^{421}\le\esth{3}\\
&-\frac{N}{2\pi A(t)}\int_{\T}\pa{}z_{1}(\alpha)\pa{3}z(\alpha)\cdot\Lambda(\pa{3}z)(\alpha)d\alpha.
\end{align*}
Here we can observe that a part of the Rayleigh-Taylor condition appears. 
Let us estimate the term $I_{1}^{422}$. We can split this term in 
\begin{align*}
&L_{3}=\frac{-\gamma_{1}}{\pi A(t)}\int_{\T}\Lambda(\pa{3}z\cdot\pa{\bot}z)(\alpha)(\pa{2}\brz{z}+\pa{2}\brh{z})\cdot\pa{}z(\alpha)d\alpha,\\
&L_{4}=\frac{-\gamma_{1}}{\pi A(t)}\int_{\T}\Lambda(\pa{3}z\cdot\pa{\bot}z)(\alpha)(\pa{}\brz{z}+\pa{}\brh{z})\cdot\pa{2}z(\alpha)d\alpha,\\
&L_{5}=\frac{-\gamma_{1}}{\pi A(t)}\int_{\T}\Lambda(\pa{3}z\cdot\pa{\bot}z)(\alpha)(\brz{z}+\brh{z})\cdot\pa{3}z(\alpha)d\alpha.\\
\end{align*}
We will estimate $L_{3}+L_{4}$ and then we will find the rest of the R-T condition in the estimations of the term $L_{5}$.
For $L_{3}$, using integration by parts,
\begin{align*}
&L_{3}=\frac{\gamma_{1}}{\pi A(t)}\int_{\T}H(\pa{3}z\cdot\pa{\bot}z)(\alpha)(\pa{2}\brz{z}+\pa{2}\brh{z})\cdot\pa{2}z(\alpha)d\alpha\\
&+\frac{\gamma_{1}}{\pi A(t)}\int_{\T}H(\pa{3}z\cdot\pa{\bot}z)(\alpha)(\pa{3}\brz{z}+\pa{3}\brh{z})\cdot\pa{}z(\alpha)d\alpha\\
&\equiv L_{3}^{1}+L_{3}^{2}.
\end{align*}
Directly, using (\ref{brzh})
\begin{align*}
&L_{3}^{1}\le C\norm{\F}_{L^{\infty}}\norm{\pa{3}z\cdot\pa{\bot}z}_{L^{2}}(\norm{\pa{2}\brz{z}}_{L^{2}}+\norm{\pa{2}\brh{z}}_{L^{2}})\norm{z}_{\mathcal{C}^{2}}\\
&\le\esth{3}.
\end{align*}
For $L_{3}^{2}$, we write
\begin{align*}
&L_{3}^{2}=\frac{\gamma_{1}}{\pi A(t)}\int_{\T}H(\pa{3}z\cdot\pa{\bot}z)(\alpha)\pa{3}\brz{z}\cdot\pa{}z(\alpha)d\alpha\\
&+\frac{\gamma_{1}}{\pi A(t)}\int_{\T}H(\pa{3}z\cdot\pa{\bot}z)(\alpha)\pa{3}\brh{z}\cdot\pa{}z(\alpha)d\alpha\equiv L_{3}^{21}+L_{3}^{22}.
\end{align*}
The application of the Leibniz's rule to $\pa{3}\brz{z}$ produces terms which can be estimated with the same tools used before. The most singular terms for $L_{3}^{21}$ are
\begin{align*}
&L_{3}^{211}=\frac{\gamma_{1}}{\pi A(t)}\int_{\T}H(\pa{3}z\cdot\pa{\bot}z)(\alpha)\pa{}(BR(\pa{2}\varpi_{1},z)_{z})\cdot\pa{}z(\alpha)d\alpha,\\
&L_{3}^{212}=\frac{\gamma_{1}}{\pi A(t)}\int_{\T}\int_{\R}H(\pa{3}z\cdot\pa{\bot}z)(\alpha)\frac{(\Delta\pa{3}z)^{\bot}\cdot\pa{}z(\alpha)}{\abs{\Delta z}^{2}}\varpi_{1}(\alpha-\beta)d\beta d\alpha,\\
&L_{3}^{213}=\frac{\gamma_{1}}{\pi A(t)}\int_{\T}\int_{\R}H(\pa{3}z\cdot\pa{\bot}z)(\alpha)\frac{(\Delta z)^{\bot}\cdot\pa{}z(\alpha)\Delta z\cdot\Delta\pa{3}z}{\abs{\Delta z}^{4}}\varpi_{1}(\alpha-\beta)d\beta d\alpha.\\
\end{align*}
Since,
\begin{displaymath}
\pa{}(BR(\pa{2}\varpi_{1},z)_{z})\cdot\pa{}z(\alpha)=\pa{}(T_{1}(\pa{2}\varpi_{1}))-BR(\pa{2}\varpi_{1},z)_{z}\cdot\pa{2}z(\alpha).
\end{displaymath}
And using the estimations on $\norm{\OT}_{L^{2}\times L^{2}\to H^{1}\times H^{1}}$ and the estimations on $\brz{z}$ we get
\begin{align*}
L_{3}^{211}&\le C\norm{\F}_{L^{\infty}}\norm{\pa{3}z\cdot\pa{\bot}z}_{L^{2}}(\norm{T_{1}(\pa{2}\varpi_{1})}_{H^{1}}+\norm{BR(\pa{2}\varpi_{1},z)_{z}}_{L^{2}}\norm{z}_{\mathcal{C}^{2}}).\\
&\le\esth{3}
\end{align*}
 For $L_{3}^{212}$ we get,
 \begin{align*}
 &M_{1}=\frac{\gamma_{1}}{\pi A(t)}\int_{\T}\int_{\R}H(\pa{3}z\cdot\pa{\bot}z)(\alpha)\Delta\pa{3}z^{\bot}\cdot\pa{}z(\alpha)\varpi_{1}(\alpha-\beta)(\frac{1}{\abs{\Delta z}^{2}}-\frac{1}{\beta^{2}\abs{\pa{}z(\alpha)}^{2}})d\beta d\alpha,\\
 &M_{2}=\frac{\gamma_{1}}{\pi A(t)}\int_{\T}\int_{\R}H(\pa{3}z\cdot\pa{\bot}z)(\alpha)\Delta\pa{3}z^{\bot}\cdot\pa{}z(\alpha)\frac{\varpi_{1}(\alpha-\beta)}{\beta^{2}\abs{\pa{}z(\alpha)}^{2}}d\beta d\alpha.\\
 \end{align*}
 If we compute $\frac{1}{\abs{\Delta z}^{2}}-\frac{1}{\beta^{2}\abs{\pa{}z(\alpha)}^{2}}=B_{1}+B_{2}+B_{3}$ where
\begin{align*}
&B_1(\alpha,\beta)=\frac{\beta\int_{0}^{1}\int_{0}^{1}\frac{\partial^{2}_{\alpha}z(\psi)-\partial^{2}_{\alpha}z(\alpha)}{\abs{\psi-\alpha}^{\delta}}\beta^{\delta}(1+s+t-st)^{\delta}(1-s)dtds\int^{1}_{0}[\partial_{\alpha}z(\alpha)+\partial_{\alpha}z(\phi)]ds}{\abs{z(\alpha)-z(\alpha-\beta)}^{2}\abs{\partial_{\alpha}z(\alpha)}^{2}},\\
&B_3(\alpha,\beta)=\frac{\beta^{2}\partial^{2}_{\alpha}z(\alpha)\int_{0}^{1}\int_{0}^{1}\partial^{2}_{\alpha}z(\eta)(s-1)dtds}{\abs{z(\alpha)-z(\alpha-\beta)}^{2}\abs{\partial_{\alpha}z(\alpha)}^{2}},\\
&B_4(\alpha,\beta)=\frac{\beta\partial^{2}_{\alpha}z(\alpha)2\partial_{\alpha}z(\alpha)}{\abs{z(\alpha)-z(\alpha-\beta)}^{2}\abs{\partial_{\alpha}z(\alpha)}^{2}},
\end{align*}
we can split $M_{1}=M_{1}^{1}+M_{1}^{2}+M_{1}^{3}$ for
\begin{align*}
&M_{1}^{1}=\frac{\gamma_{1}}{\pi A(t)}\int_{\T}\int_{\R}H(\pa{3}z\cdot\pa{\bot}z)(\alpha)\Delta\pa{3}z^{\bot}\cdot\pa{}z(\alpha)\varpi_{1}(\alpha-\beta)B_{1}(\alpha,\beta)d\beta d\alpha,\\
&M_{1}^{2}=\frac{\gamma_{1}}{\pi A(t)}\int_{\T}\int_{\R}H(\pa{3}z\cdot\pa{\bot}z)(\alpha)\Delta\pa{3}z^{\bot}\cdot\pa{}z(\alpha)\varpi_{1}(\alpha-\beta)B_{2}(\alpha,\beta)d\beta d\alpha,\\
&M_{1}^{3}=\frac{\gamma_{1}}{\pi A(t)}\int_{\T}\int_{\R}H(\pa{3}z\cdot\pa{\bot}z)(\alpha)\Delta\pa{3}z^{\bot}\cdot\pa{}z(\alpha)\varpi_{1}(\alpha-\beta)B_{3}(\alpha,\beta)d\beta d\alpha.\\
\end{align*}
It is easy see that 
\begin{align*}
&M_{1}^{1}\le C\norm{\F}^{2}_{L^{\infty}}\norm{z}_{\mathcal{C}^{2,\delta}}\norm{z}_{\mathcal{C}^{1}}\norm{\varpi_{1}}_{L^{\infty}}\norm{z}^{2}_{H^{3}},\\
&M_{1}^{2}\le C\norm{\F}^{2}_{L^{\infty}}\norm{z}^{2}_{\mathcal{C}^{2}}\norm{\varpi_{1}}_{L^{\infty}}\norm{z}^{2}_{H^{3}}.\\
\end{align*}
We need to study more precisely the term $M_{1}^{3}$. Again, we decompose $M_{1}^{3}$ in the following terms:
\begin{align*}
&M_{1}^{31}=\frac{2\gamma_{1}}{\pi A^{2}(t)}\int_{\T}\int_{\R}H(\pa{3}z\cdot\pa{\bot}z)(\alpha)\Delta\pa{3}z^{\bot}\cdot\pa{}z(\alpha)\varpi_{1}(\alpha-\beta)\beta\pa{2}z(\alpha)\cdot\pa{}z(\alpha)B(\alpha,\beta)d\beta d\alpha,\\
&M_{1}^{32}=\frac{2\gamma_{1}}{\pi A^{3}(t)}\int_{\T}\int_{\R}H(\pa{3}z\cdot\pa{\bot}z)(\alpha)\Delta\pa{3}z^{\bot}\cdot\pa{}z(\alpha)\frac{\varpi_{1}(\alpha-\beta)}{\beta}\pa{2}z(\alpha)\cdot\pa{}z(\alpha)d\beta d\alpha,\\
\end{align*}
where
\begin{displaymath}
B(\alpha,\beta)=\frac{1}{\abs{\Delta z}^{2}}-\frac{1}{\beta^{2}\abs{\pa{}z(\alpha)}^{2}}=\frac{\beta\int_{0}^{1}\int_{0}^{1}\partial^{2}_{\alpha}z(\psi)(1-s)dtds\cdot\int^{1}_{0}[\partial_{\alpha}z(\alpha)+\partial_{\alpha}z(\phi)]ds}{\abs{z(\alpha)-z(\alpha-\beta)}^{2}\abs{\partial_{\alpha}z(\alpha)}^{2}}.
\end{displaymath}
Directly,
\begin{displaymath}
M_{1}^{31}\le C\norm{\F}_{L^{\infty}}\norm{z}_{\mathcal{C}^{2}}^{2}\norm{\varpi_{1}}_{L^{\infty}}\norm{z}_{H^{3}}.
\end{displaymath}
For $M_{1}^{32}$,
\begin{align*}
&M_{1}^{321}=\frac{2\gamma_{1}}{\pi A^{3}(t)}\int_{\T}\int_{\R}H(\pa{3}z\cdot\pa{\bot}z)(\alpha)\pa{3}z(\alpha)^{\bot}\cdot\pa{}z(\alpha)\frac{\varpi_{1}(\alpha-\beta)}{\beta}\pa{2}z(\alpha)\cdot\pa{}z(\alpha)d\beta d\alpha\\
&=\frac{2\gamma_{1}}{\pi A^{3}(t)}\int_{\T}H(\pa{3}z\cdot\pa{\bot}z)(\alpha)\pa{3}z(\alpha)^{\bot}\cdot\pa{}z(\alpha)H\varpi_{1}(\alpha)\pa{2}z(\alpha)\cdot\pa{}z(\alpha)d\beta d\alpha\\
&\le C\norm{\F}^{\frac{3}{2}}_{L^{\infty}}\norm{z}_{\mathcal{C}^{2}}\norm{H\varpi_{1}}_{L^{\infty}}\norm{z}^{2}_{H^{3}}
\end{align*}
and
\begin{align*}
&M_{1}^{322}=\frac{2\gamma_{1}}{\pi A^{3}(t)}\int_{\T}\int_{\R}H(\pa{3}z\cdot\pa{\bot}z)(\alpha)\pa{3}z(\alpha-\beta)^{\bot}\cdot\pa{}z(\alpha)\frac{\varpi_{1}(\alpha-\beta)}{\beta}\pa{2}z(\alpha)\cdot\pa{}z(\alpha)d\beta d\alpha\\
&\le C\norm{\F}^{\frac{3}{2}}_{L^{\infty}}\norm{z}_{\mathcal{C}^{2}}\norm{\varpi}_{\mathcal{C}^{1}}\norm{z}^{2}_{H^{3}}\\
&+\frac{2\gamma_{1}}{\pi A^{3}(t)}\int_{\T}H(\pa{3}z\cdot\pa{\bot}z)(\alpha)H(\pa{3}z^{\bot})(\alpha)\cdot\pa{}z(\alpha)\varpi_{1}(\alpha)\pa{2}z(\alpha)\cdot\pa{}z(\alpha)d\alpha\\
&\le C\norm{\F}^{\frac{3}{2}}_{L^{\infty}}\norm{z}_{\mathcal{C}^{2}}\norm{\varpi}_{\mathcal{C}^{1}}\norm{z}^{2}_{H^{3}}.\\
\end{align*}
Therefore, $M_{1}\le\esth{3}$.

For $M_{2}$ we proceed as follows:
\begin{align*}
&M_{2}^{1}=\frac{\gamma_{1}}{\pi A^{2}(t)}\int_{\T}\int_{\R}H(\pa{3}z\cdot\pa{\bot}z)(\alpha)\Delta\pa{3}z^{\bot}\cdot\pa{}z(\alpha)\frac{\varpi_{1}(\alpha-\beta)-\varpi_{1}(\alpha)}{\beta^{2}}d\beta d\alpha,\\
&M_{2}^{2}=\frac{\gamma_{1}}{\pi A^{2}(t)}\int_{\T}H(\pa{3}z\cdot\pa{\bot}z)(\alpha)\Lambda(\pa{3}z^{\bot})(\alpha)\cdot\pa{}z(\alpha)\varpi_{1}(\alpha)d\alpha.\\
\end{align*}
In the same way as before,
\begin{align*}
&M_{2}^{1}=\frac{\gamma_{1}}{\pi A^{2}(t)}\int_{\T}H(\pa{3}z\cdot\pa{\bot}z)(\alpha)\pa{3}z(\alpha)^{\bot}\cdot\pa{}z(\alpha)\Lambda\varpi_{1}(\alpha)d\alpha\\
&+\frac{\gamma_{1}}{\pi A^{2}(t)}\int_{\T}\int_{\R}H(\pa{3}z\cdot\pa{\bot}z)(\alpha)\pa{3}z(\alpha-\beta)^{\bot}\cdot\pa{}z(\alpha)\frac{\int_{0}^{1}\pa{}\varpi_{1}(\alpha-\beta+\beta t)-\pa{}\varpi_{1}(\alpha)dt}{\beta}d\beta d\alpha\\
&+\frac{\gamma_{1}}{\pi A^{2}(t)}\int_{\T}H(\pa{3}z\cdot\pa{\bot}z)(\alpha)H(\pa{3}z^{\bot})(\alpha)\cdot\pa{}z(\alpha)\pa{}\varpi_{1}(\alpha)d\alpha\\
&\le C\norm{\F}_{L^{\infty}}\norm{\varpi_{1}}_{\mathcal{C}^{1,\delta}}\norm{z}^{2}_{H^{3}}.
\end{align*}
Using the commutator estimation and $\Lambda H=-\pa{}$
\begin{align*}
&M_{2}^{2}=\frac{\gamma_{1}}{\pi A^{2}(t)}\int_{\T}H(\pa{3}z\cdot\pa{\bot}z)(\alpha)(\Lambda(\pa{3}z^{\bot})(\alpha)\cdot\pa{}z(\alpha)\varpi_{1}(\alpha)-\Lambda(\pa{3}z^{\bot}\cdot\pa{}z\varpi_{1})(\alpha))d\alpha\\
&+\frac{\gamma_{1}}{\pi A^{2}(t)}\int_{\T}H(\pa{3}z\cdot\pa{\bot}z)(\alpha)\Lambda(\pa{3}z^{\bot}\cdot\pa{}z\varpi_{1})(\alpha)d\alpha\\
&\le C\norm{\F}_{L^{\infty}}^{2}\norm{\pa{3}z\pa{}z}_{L^{2}}\norm{\varpi_{1}\pa{}z}_{\mathcal{C}^{1,\delta}}\norm{\pa{3}z}_{L^{2}}\\
&-\frac{\gamma_{1}}{\pi A^{2}(t)}\int_{\T}\pa{}(\pa{3}z\cdot\pa{\bot}z)(\alpha)\pa{3}z^{\bot}(\alpha)\cdot\pa{}z(\alpha)\varpi_{1}(\alpha)d\alpha\\
&\le\esth{3}\\
&+\frac{\gamma_{1}}{\pi A^{2}(t)}\int_{\T}\pa{}(\pa{3}z\cdot\pa{\bot}z)(\alpha)\pa{3}z(\alpha)\cdot\pa{\bot}z(\alpha)\varpi_{1}(\alpha)d\alpha\\
&\le\esth{3}\\
&+C\norm{\F}_{L^{\infty}}\norm{\varpi_{1}}_{\mathcal{C}^{1}}\norm{z}_{H^{3}}
\end{align*}
We can estimate $L_{3}^{213}$ as before. Then, we get the estimation for $L_{3}^{21}$.

Let us estimate $L_{3}^{22}$. The most singular terms are
\begin{align*}
&L_{3}^{221}=\frac{\gamma_{1}}{\pi A(t)}\int_{\T}H(\pa{3}z\cdot\pa{\bot}z)(\alpha)\pa{}(BR(\pa{2}\varpi_{2},h)_{z})\cdot\pa{}z(\alpha)d\alpha,\\
&L_{3}^{222}=\frac{\gamma_{1}}{\pi A(t)}\int_{\T}\int_{\R}H(\pa{3}z\cdot\pa{\bot}z)(\alpha)\frac{(\Delta\pa{3}zh)^{\bot}\cdot\pa{}z(\alpha)}{\abs{\Delta zh}^{2}}\varpi_{2}(\alpha-\beta)d\beta d\alpha\\
&L_{3}^{223}=\frac{\gamma_{1}}{\pi A(t)}\int_{\T}\int_{\R}H(\pa{3}z\cdot\pa{\bot}z)(\alpha)\frac{(\Delta zh)^{\bot}\cdot\pa{}z(\alpha)\Delta zh\cdot\Delta\pa{3}zh}{\abs{\Delta zh}^{4}}\varpi_{2}(\alpha-\beta)d\beta d\alpha\\
\end{align*}
Since
\begin{align*}
\pa{}(BR(\pa{2}\varpi_{2},h)_{z})\cdot\pa{}z=\pa{}(T_{2}(\pa{2}\varpi_{2}))-BR(\pa{2}\varpi_{2},h)_{z}\cdot\pa{2}z
\end{align*}
then,
\begin{align*}
L_{3}^{221}\le C\norm{\F}_{L^{\infty}}\norm{\pa{3}z\pa{\bot}z}_{L^{2}}(\norm{T_{2}(\pa{2}\varpi_{2})}_{H^{1}}+\norm{BR(\pa{2}\varpi_{2},h)_{z}}_{L^{2}}\norm{\pa{2}z}_{L^{\infty}})
\end{align*}
Using the estimation on $\norm{\OT\varpi}_{H^{1}}$ and $\brh{z}$, $L_{3}^{221}$ is controlled.
We can get,
\begin{align*}
L_{3}^{222}\le C\norm{\F}_{L^{\infty}}\norm{d(z,h)}_{L^{\infty}}\norm{\pa{3}z\pa{}z}_{L^{2}}(\norm{\pa{3}z}_{L^{2}}+\norm{\pa{3}h}_{L^{2}})\norm{z}_{\mathcal{C}^{1}}\norm{\varpi_{2}}_{L^{\infty}}.
\end{align*}
For $L_{3}^{223}$ we get the same
\begin{align*}
L_{3}^{223}\le C\norm{\F}_{L^{\infty}}\norm{d(z,h)}_{L^{\infty}}\norm{\pa{3}z}_{L^{2}}(\norm{\pa{3}z}_{L^{2}}+\norm{\pa{3}h}_{L^{2}})\norm{z}^{2}_{\mathcal{C}^{1}}\norm{\varpi_{2}}_{L^{\infty}}.
\end{align*}
For $L_{4}$ integrating by parts we obtain:
\begin{align*}
&L_{4}\le C\norm{\F}_{L^{\infty}}\norm{\pa{3}z\pa{}z}_{L^{2}}((\norm{\pa{2}\brz{z}}_{L^{2}}+\norm{\pa{2}\brh{z}})\norm{\pa{2}z}_{L^{\infty}}\\
&+(\norm{\pa{}\brz{z}}_{L^{\infty}}+\norm{\pa{}\brh{z}}_{L^{\infty}})\norm{\pa{3}z}_{L^{2}})\\
&\le\esth{3}.
\end{align*}
Finally we have to find $\sigma(\alpha)$ in $L_{5}$ to finish the estimations. To do that let us split $L_{5}=L_{5}^{1}+L_{5}^{2}+L_{5}^{3}+L_{5}^{4}$ where
\begin{align*}
&L_{5}^{1}=\frac{\gamma_{1}}{\pi A(t)}\int_{\T}\Lambda(\pa{3}z_{1}\pa{}z_{2})(\alpha)(BR_{1}(\varpi_{1},z)_{z}+BR_{1}(\varpi_{2},h)_{z})\pa{3}z_{1}(\alpha)d\alpha,\\
&L_{5}^{2}=\frac{\gamma_{1}}{\pi A(t)}\int_{\T}\Lambda(\pa{3}z_{1}\pa{}z_{2})(\alpha)(BR_{2}(\varpi_{1},z)_{z}+BR_{2}(\varpi_{2},h)_{z})\pa{3}z_{2}(\alpha)d\alpha,\\
&L_{5}^{3}=-\frac{\gamma_{1}}{\pi A(t)}\int_{\T}\Lambda(\pa{3}z_{2}\pa{}z_{1})(\alpha)(BR_{1}(\varpi_{1},z)_{z}+BR_{1}(\varpi_{2},h)_{z})\pa{3}z_{1}(\alpha)d\alpha,\\
&L_{5}^{4}=-\frac{\gamma_{1}}{\pi A(t)}\int_{\T}\Lambda(\pa{3}z_{2}\pa{}z_{1})(\alpha)(BR_{2}(\varpi_{1},z)_{z}+BR_{2}(\varpi_{2},h)_{z})\pa{3}z_{2}(\alpha)d\alpha.\\
\end{align*}
In order to reduce the notation, we denote $BR_{i}=BR_{i}(\varpi_{1},z)_{z}+BR_{i}(\varpi_{2},h)_{z}$ for $i=1,2$.
Then we can write,
\begin{align*}
&L_{5}^{1}=\frac{\gamma_{1}}{\pi A(t)}\int_{\T}(\Lambda(\pa{3}z_{1}\pa{}z_{2})(\alpha)-\pa{}z_{2}\Lambda(\pa{3}z)(\alpha))BR_{1}\pa{3}z_{1}(\alpha)d\alpha\\
&+\frac{\gamma_{1}}{\pi A(t)}\int_{\T}BR_{1}\pa{}z_{2}(\alpha)\pa{3}z_{1}(\alpha)\Lambda(\pa{3}z_{1})(\alpha)d\alpha\\
&\le C\norm{\pa{}z}_{\mathcal{C}^{1,\delta}}\norm{BR_{1}}_{L^{\infty}}\norm{z}_{H^{3}}^{2}+\frac{\gamma_{1}}{\pi A(t)}\int_{\T}BR_{1}\pa{}z_{2}(\alpha)\pa{3}z_{1}(\alpha)\Lambda(\pa{3}z_{1})(\alpha)d\alpha.
\end{align*}
In the same way,
\begin{align*}
&L_{5}^{2}\le\esth{3}+ L^{21}_{5}
\end{align*}
where
\begin{align*}
&L_{5}^{21}=\frac{\gamma_{1}}{\pi A(t)}\int_{\T}BR_{2}\pa{}z_{2}(\alpha)\pa{3}z_{2}(\alpha)\Lambda(\pa{3}z_{1})(\alpha)d\alpha.
\end{align*}
Using $\pa{}z_{2}\pa{3}z_{2}=-\pa{}z_{1}\pa{3}z_{1}-\abs{\pa{2}z}^{2}$ we separate $L_{5}^{21}$ in
\begin{align*}
&L_{5}^{211}=-\frac{\gamma_{1}}{\pi A(t)}\int_{\T}BR_{2}\abs{\pa{2}z(\alpha)}^{2}\Lambda(\pa{3}z_{1})(\alpha)d\alpha,\\
&L_{5}^{212}=-\frac{\gamma_{1}}{\pi A(t)}\int_{\T}BR_{2}\pa{}z_{1}(\alpha)\pa{3}z_{1}(\alpha)\Lambda(\pa{3}z_{1})(\alpha)d\alpha.
\end{align*}
The fact that $\Lambda=\pa{}H$ allows us to
\begin{align*}
&L_{5}^{211}=\frac{\gamma_{1}}{\pi A(t)}\int_{\T}\pa{}(BR_{2}\abs{\pa{2}z(\alpha)}^{2})H(\pa{3}z_{1})(\alpha)d\alpha\\
&\le C(\norm{\pa{}BR_{2}}_{L^{2}}\norm{z}^{2}_{\mathcal{C}^{2}}+\norm{BR_{2}}_{L^{2}}\norm{z}_{\mathcal{C}^{2}}^{2})\norm{z}_{H^{3}}\norm{\F}_{L^{\infty}}\\
&\le\esth{3}.
\end{align*}
Then we get,
\begin{align*}
&L_{5}^{2}\le\esth{3}\\
&-\frac{\gamma_{1}}{\pi A(t)}\int_{\T}BR_{2}\pa{}z_{1}(\alpha)\pa{3}z_{1}(\alpha)\Lambda(\pa{3}z_{1})(\alpha)d\alpha.
\end{align*}
Now, we add $L_{5}^{1}+L_{5}^{2}$:
\begin{align*}
&L_{5}^{1}+L_{5}^{2}\le\esth{3}\\
&-\frac{\gamma_{1}}{\pi A(t)}\int_{\T}(\brz{z}+\brh{z})\cdot\pa{\bot}z(\alpha)\pa{3}z_{1}(\alpha)\Lambda(\pa{3}z_{1})(\alpha)d\alpha.
\end{align*}
Analogously, using $\pa{}z_{1}\pa{3}z_{1}=-\pa{}z_{2}\pa{3}z_{2}-\abs{\pa{2}z}^{2}$ we get:
\begin{align*}
&L_{5}^{3}+L_{5}^{4}\le\esth{3}\\
&-\frac{\gamma_{1}}{\pi A(t)}\int_{\T}(\brz{z}+\brh{z})\cdot\pa{\bot}z(\alpha)\pa{3}z_{2}(\alpha)\Lambda(\pa{3}z_{2})(\alpha)d\alpha.
\end{align*}
Therefore,
\begin{align*}
&L_{5}\le\esth{3}\\
&-\frac{\gamma_{1}}{\pi A(t)}\int_{\T}(\brz{z}+\brh{z})\cdot\pa{\bot}z(\alpha)\pa{3}z(\alpha)\cdot\Lambda(\pa{3}z)(\alpha)d\alpha.
\end{align*}
In conclusion,
\begin{align*}
&I_{1}^{4}\le\esth{3}\\
&-\frac{1}{\pi A(t)}\int_{\T}\left[\gamma_{1}(\brz{z}+\brh{z})\cdot\pa{\bot}z(\alpha)+N\pa{}z_{1}(\alpha)\right]\pa{3}z(\alpha)\cdot\Lambda(\pa{3}z)(\alpha)d\alpha.
\end{align*}
Since,
\begin{displaymath}
\sigma(\alpha,t)=\frac{\mu_{2}-\mu_{1}}{\kappa_{1}}(\brz{z}+\brh{z})\cdot\pa{\bot}z(\alpha)+(\rho_{2}-\rho_{1})g\pa{}z_{1}(\alpha)
\end{displaymath}
then,
\begin{align*}
&I_{1}\le\esth{3}\\
&-\frac{\kappa_{1}}{2\pi(\mu_{2}+\mu_{1})A(t)}\int_{\T}\sigma(\alpha,t)\pa{3}z(\alpha)\cdot\Lambda(\pa{3}z)(\alpha)d\alpha.
\end{align*}

%%%%%%%%%%%%%%%%%%%%%%%%%%%%%%%%%%%%%%%%%%%%%%%%%%%%%%
\subsection{Estimates on $I_{3}$}
\label{estimacionesc}
To finish all estimation on $z$, we consider:
\begin{align*}
&I_{3}=\int_{\T}\pa{3}z(\alpha)\cdot\pa{4}z(\alpha)c(\alpha)d\alpha+3\int_{\T}\abs{\pa{3}z(\alpha)}^{2}\pa{}c(\alpha)d\alpha\\
&+3\int_{\T}\pa{3}z(\alpha)\cdot\pa{2}z(\alpha)\pa{2}c(\alpha)d\alpha+\int_{\T}\pa{3}z(\alpha)\cdot\pa{}z(\alpha)\pa{3}c(\alpha)d\alpha\\
&I_{3}^{1}+I_{3}^{2}+I_{3}^{3}+I_{3}^{4}.
\end{align*}
Integrating by parts and using the definition of $c(\alpha)$,
\begin{align*}
&I_{3}^{1}=-\frac{1}{2}\int_{\T}\abs{\pa{3}z(\alpha)}^{2}\pa{}c(\alpha)d\alpha\le C\norm{\pa{}c}_{L^{\infty}}\norm{\pa{3}z}^{2}_{L^{2}}\\
&\le 2C\norm{\F}^{\frac{1}{2}}_{L^{\infty}}(\norm{\pa{}\brz{z}}_{L^{\infty}}+\norm{\pa{}\brh{z}}_{L^{\infty}})\norm{\pa{3}z}^{2}_{L^{2}}.
\end{align*}
Since $I_{3}^{2}=-6I_{3}^{1}$ we have $I_{3}^{2}$ controlled.
Computing $\pa{2}c$,
\begin{align*}
&\pa{2}c(\alpha)=-\frac{\pa{2}z(\alpha)}{A(t)}\cdot(\pa{}\brz{z}+\pa{}\brh{z})\\
&-\frac{\pa{}z(\alpha)}{A(t)}\cdot(\pa{2}\brz{z}+\pa{2}\brh{z}).
\end{align*} 
Thus,
\begin{align*}
&I_{3}^{3}=-3\int_{\T}\pa{3}z(\alpha)\cdot\pa{2}z(\alpha)\frac{\pa{2}z(\alpha)}{A(t)}\cdot(\pa{}\brz{z}+\pa{}\brh{z})d\alpha\\
&-3\int_{\T}\pa{3}z(\alpha)\cdot\pa{2}z(\alpha)\frac{\pa{}z(\alpha)}{A(t)}\cdot(\pa{2}\brz{z}+\pa{2}\brh{z})d\alpha\\
&\equiv I_{3}^{31}+I_{3}^{32}
\end{align*}
where
\begin{align*}
&I_{3}^{31}\le C\norm{\F}_{L^{\infty}}\norm{z}_{\mathcal{C}^{2}}^{2}(\norm{\pa{}\brz{z}}_{L^{\infty}}+\norm{\pa{}\brh{z}}_{L^{\infty}})\norm{\pa{3}z}_{L^{2}},\\
&I_{3}^{32}\le C\norm{\F}^{\frac{1}{2}}_{L^{\infty}}\norm{z}_{\mathcal{C}^{2}}(\norm{\pa{2}\brz{z}}_{L^{2}}+\norm{\pa{2}\brh{z}}_{L^{2}})\norm{\pa{3}z}_{L^{2}}.\\
\end{align*}
Using the estimation on $\norm{\brz{z}}_{H^{k}}+\norm{\brh{z}}_{H^{k}}$ we obtain,
\begin{displaymath}
I_{3}^{3}\le\esth{3}.
\end{displaymath}
Since $\abs{\pa{}z(\alpha)}^{2}=A(t)$, if we differentiate respect to $\alpha$:
\begin{align*}
0=2\abs{\pa{2}z(\alpha)}^{2}+2\pa{}z(\alpha)\cdot\pa{3}z(\alpha)\Rightarrow \pa{}z(\alpha)\cdot\pa{3}z(\alpha)=-\abs{\pa{2}z(\alpha)}^{2}.
\end{align*}
Then, integrating by parts
\begin{align*}
I_{3}^{4}=-\int_{\T}\abs{\pa{2}z(\alpha)}^{2}\pa{3}c(\alpha)d\alpha=2\int_{\T}\pa{2}z(\alpha)\cdot\pa{3}z(\alpha)\pa{2}c(\alpha)d\alpha=\frac{2}{3}I_{3}^{3}.
\end{align*}
Therefore,
\begin{displaymath}
I_{3}^{4}\le\esth{3}.
\end{displaymath}

Putting together all above estimates, since the case for $k>3$ is straighforward we have
\begin{align*}
&\frac{d}{dt}\norm{z}_{H^{k}}^{2}\le\esth{k}\\
&-\frac{\kappa^{1}}{2\pi(\mu^{1})+\mu^{2}}\int_{\T}\frac{\sigma(\alpha)}{A(t)}\pa{k}z(\alpha)\cdot\Lambda(\pa{k}z)(\alpha)d\alpha
\end{align*} 
for $k\ge 3$.

\section{Evolution of the distance between $z$ and $h$}
\label{evoldi}
Remind that we relate the distance of the curve $z$ with $h$ through the function
\begin{displaymath}
d(z,h)=\frac{1}{\abs{z(\alpha)-h(\alpha-\beta)}^{2}}
\end{displaymath}
\begin{lem}
The following estimate holds
\begin{align*}
\frac{d}{dt}\norm{d(z,h)}_{L^{\infty}}^{2}\le\esth{3}.
\end{align*}
\end{lem}
\begin{proof}
If we take $p>1$, we get
\begin{align*}
&\frac{d}{dt}\norm{d(z,h)}_{L^{p}}^{p}(t)=\frac{d}{dt}\int_{\T}\int_{\T}\frac{1}{\abs{z(\alpha)-h(\alpha-\beta)}^{2p}}d\alpha d\beta\\
&=-2p\int_{\T}\int_{\T}\frac{(z(\alpha)-h(\alpha-\beta))\cdot z_{t}(\alpha)}{\abs{z(\alpha)-h(\alpha-\beta)}^{2p+2}}d\alpha d\beta\\
&=-2p\int_{\T}\int_{\T}\frac{z(\alpha)\cdot z_{t}(\alpha)}{\abs{z(\alpha)-h(\alpha-\beta)}^{2p+2}}d\alpha d\beta+2p\int_{\T}\int_{\T}\frac{h(\alpha-\beta)\cdot z_{t}(\alpha)}{\abs{z(\alpha)-h(\alpha-\beta)}^{2p+2}}d\alpha d\beta\\
&=-2p\int_{\T}\int_{\T}\frac{z(\alpha)\cdot(\brz{z}+\brh{z})}{\abs{z(\alpha)-h(\alpha-\beta)}^{2p+2}}d\alpha d\beta\\
&-2p\int_{\T}\int_{\T}\frac{z(\alpha\cdot\pa{}z(\alpha)c(\alpha)}{\abs{z(\alpha)-h(\alpha-\beta)}^{2p+2}}d\alpha d\beta+2p\int_{\T}\int_{\T}\frac{h(\alpha-\beta)\cdot(\brz{z}+\brh{z})}{\abs{z(\alpha)-h(\alpha-\beta)}^{2p+2}}d\alpha d\beta\\
&+2p\int_{\T}\int_{\T}\frac{h(\alpha-\beta)\cdot \pa{}z(\alpha)c(\alpha)}{\abs{z(\alpha)-h(\alpha-\beta)}^{2p+2}}d\alpha d\beta\equiv J_{1}+J_{2}+J_{3}+J_{4}.
\end{align*}
It is easy to see that
\begin{align*}
&J_{1}\le C\norm{d(z,h)}_{L^{\infty}}\norm{z}_{L^{2}}\norm{\brz{z}+\brh{z}}_{L^{2}}\int_{\T}\int_{\T}\frac{1}{\abs{z(\alpha)-h(\alpha-\beta)}^{2p}}d\alpha d\beta\\
&\le\esth{3}\norm{d(z,h)}_{L^{p}}^{p},\\
&J_{3}\le C\norm{d(z,h)}_{L^{\infty}}\norm{h}_{L^{2}}\norm{\brz{z}+\brh{z}}_{L^{2}}\int_{\T}\int_{\T}\frac{1}{\abs{z(\alpha)-h(\alpha-\beta)}^{2p}}d\alpha d\beta\\
&\le\esth{3}\norm{d(z,h)}_{L^{p}}^{p},\\
&J_{2}\le C\norm{d(z,h)}_{L^{\infty}}\norm{c}_{L^{\infty}}\norm{z}_{L^{\infty}}\norm{z}_{\mathcal{C}^{1}}\int_{\T}\int_{\T}\frac{1}{\abs{z(\alpha)-h(\alpha-\beta)}^{2p}}d\alpha d\beta\\
&\le\esth{3}\norm{d(z,h)}_{L^{p}}^{p},
\end{align*}
and
\begin{align*}
&J_{4}\le C\norm{d(z,h)}_{L^{\infty}}\norm{c}_{L^{\infty}}\norm{h}_{L^{\infty}}\norm{z}_{\mathcal{C}^{1}}\int_{\T}\int_{\T}\frac{1}{\abs{z(\alpha)-h(\alpha-\beta)}^{2p}}d\alpha d\beta\\
&\le\esth{3}\norm{d(z,h)}_{L^{p}}^{p}.
\end{align*}
Therefore,
\begin{align*}
\frac{d}{dt}\norm{d(z,h)}_{L^{p}}^{p}\le\exp C\nor{z,h}^{2}\norm{d(z,h)}_{L^{p}}^{p}.
\end{align*}
Let integrate on $t$,
\begin{align*}
&\norm{d(z,h)}_{L^{p}}(t+h)\le\norm{d(z,h)}_{L^{p}}(t)\exp(\int_{t}^{t+h}\exp C\nor{z,h}^{2}(s)ds).
\end{align*}
If we take $p\to\infty$ we get
\begin{align*}
\norm{d(z,h)}_{L^{\infty}}(t+h)\le\norm{d(z,h)}_{L^{\infty}}(t)\exp(\int_{t}^{t+h}\exp C\nor{z,h}^{2}(s)ds),
\end{align*}
then
\begin{align*}
&\frac{d}{dt}\norm{d(z,h)}_{L^{\infty}}(t)=\lim_{h\to 0}(\frac{\norm{d(z,h)}_{L^{\infty}}(t+h)-\norm{d(z,h)}_{L^{\infty}}(t)}{h})\\
&\le\norm{d(z,h)}_{L^{\infty}}(t)\lim_{h\to 0}(\frac{\exp\int_{t}^{t+h}\exp\nor{z,h}^{2}(s)ds-1}{h})\\
&\le\norm{d(z,h)}_{L^{\infty}}(t)\exp\nor{z,h}^{2}(t).
\end{align*}
\end{proof}

\
%%%%%%%%%%%%%%%%%%%%%%%%%%%%%%%%%%%%%%%%%%%%%%%%%%%%%%%
\section{Evolution of the minimum of $\sigma(\alpha,t)$}
\label{evolsig}
We know that 
\begin{displaymath}
\sigma(\alpha,t)=\frac{\mu^{2}-\mu^{1}}{\kappa^{1}}(\brz{z}+\brh{z})\cdot\pa{\bot}z(\alpha)+(\rho^{2}-\rho^{1})g\pa{}z_{1}(\alpha).
\end{displaymath}
\begin{lem}
Let $z(\alpha,t)$ be a solution of the system with $z(\alpha,t)\in\mathcal{C}^{1}(\left[0,T\right];H^{3})$ and $m(t)=\min_{\alpha\in\T}\sigma(\alpha,t)$.
Then
\begin{displaymath}
m(t)\ge m(0)-\int_{0}^{t}\exp C\nor{z,h}^{2}(s)ds.
\end{displaymath}
\end{lem}
Recall that
\begin{align*}
&\exp C\nor{z,h}^{2}=\esth{3}.
\end{align*}
\begin{proof}
We consider $\alpha_{t}\in\T$ such that
\begin{displaymath}
m(t)=\min_{\alpha\in\T}\sigma(\alpha,t)=\sigma(\alpha_{t},t).
\end{displaymath}
We may calculate the derivate of $m(t)$, to obtain
\begin{displaymath}
m'(t)=\sigma_{t}(\alpha_{t},t).
\end{displaymath}
Using the definition,
\begin{align*}
&\sigma_{t}(\alpha,t)=\frac{\mu^{2}-\mu^{1}}{\kappa^{1}}(\partial_{t}\brz{z}+\partial_{t}\brh{z})\cdot\pa{\bot}z(\alpha)\\
&+(\frac{\mu^{2}-\mu^{1}}{\kappa^{1}}(\brz{z}+\brh{z})\cdot\pa{\bot}z_{t}(\alpha)+(\rho^{2}-\rho^{1})g\pa{}\partial_{t}z_{1}(\alpha))\\
&\equiv I_{1}+I_{2}.
\end{align*}
We have,
\begin{align*}
&\abs{I_{2}}\le C(\norm{\brz{z}+\brh{z}}_{L^{\infty}}+1)\norm{\pa{}z_{t}}_{L^{\infty}}\\
&\le\exp C\nor{z,h}^{2}\norm{\pa{}z_{t}}_{L^{\infty}}.
\end{align*}
Using the equation of $z_{t}$ we can calculate the estimations of $\norm{\pa{}z_{t}}_{L^{\infty}}$. We have,
\begin{align}
\label{paz}
&\norm{\pa{}z_{t}}_{L^{\infty}}\le\norm{\brz{z}}_{L^{\infty}}+\norm{\brh{z}}_{L^{\infty}}+\norm{\pa{}c}_{L^{\infty}}\norm{\pa{}z}_{L^{\infty}}\\
&+\norm{c}_{L^{\infty}}\norm{\pa{2}z}_{L^{\infty}}\le\exp C\nor{z,h}^{2}\nonumber
\end{align}
then we obtain
\begin{displaymath}
\abs{I_{2}}\le\exp C\nor{z,h}^{2}.
\end{displaymath}
Let us write $\partial_{t}\brz{z}=B_{1}+B_{2}+B_{3}$ where
\begin{align*}
&B_{1}=\frac{1}{2\pi}PV\int_{\R}\frac{(\Delta z)^{\bot}}{\abs{\Delta z}^{2}}\partial_{t}\varpi_{1}(\alpha-\beta)d\alpha d\beta,\\
&B_{2}=\frac{1}{2\pi}PV\int_{\R}\frac{(\Delta z_{t})^{\bot}}{\abs{\Delta z}^{2}}\varpi_{1}(\alpha-\beta)d\alpha d\beta,\\
&B_{3}=-\frac{1}{\pi}PV\int_{\R}\frac{(\Delta z)^{\bot}\Delta z\cdot\Delta z_{t}}{\abs{\Delta z}^{4}}\varpi_{1}(\alpha-\beta)d\alpha d\beta.\\
\end{align*}
We split $B_{1}$ in the following way,
\begin{align*}
&B_{1}=\frac{1}{2\pi}PV\int_{\R}(\frac{(\Delta z)^{\bot}}{\abs{\Delta z}^{2}}-\frac{\pa{\bot}z(\alpha)}{\beta\abs{\pa{}z(\alpha)}^{2}})\partial_{t}\varpi_{1}(\alpha-\beta)d\alpha d\beta+\frac{\pa{\bot}z(\alpha)}{\abs{\pa{}z(\alpha)}^{2}}H(\partial_{t}\varpi_{1})(\alpha).
\end{align*}
Then,
\begin{align*}
&\abs{B_{1}}\le C\norm{\F}_{L^{\infty}}\norm{z}_{\mathcal{C}^{2}}\norm{\partial_{t}\varpi_{1}}_{L^{2}}+\norm{\F}^{\frac{1}{2}}_{L^{\infty}}\norm{\partial_{t}\varpi_{1}}_{\mathcal{C}^{\delta}}.
\end{align*}
For estimate $B_{2}$ we separate
\begin{align*}
&B_{2}=\frac{1}{2\pi}PV\int_{\R}(\Delta z_{t})^{\bot}(\frac{1}{\abs{\Delta z}^{2}}-\frac{1}{\beta^{2}\abs{\pa{}z(\alpha)}^{2}})\varpi_{1}(\alpha-\beta)d\alpha d\beta\\
&+\frac{1}{2\pi}PV\int_{\R}\frac{(\Delta z_{t})^{\bot}}{\beta^{2}\abs{\pa{}z(\alpha)}^{2}}\varpi_{1}(\alpha-\beta)d\alpha d\beta\equiv B_{2}^{1}+B_{2}^{2}.
\end{align*}
Since,
\begin{align*}
&z_{t}(\alpha)-z_{t}(\alpha-\beta)=((\brz{z}(\alpha)-\brz{z}(\alpha-\beta))\\
&+(\brh{z}(\alpha)-\brh{z}(\alpha-\beta)))+ (c(\alpha)-c(\alpha-\beta))\pa{}z(\alpha-\beta)\\
&+c(\alpha-\beta)(\pa{}z(\alpha)-\pa{}z(\alpha-\beta))\equiv J_{1}+J_{2}+J_{3}
\end{align*}
we have,
\begin{align}
\label{j1}
&J_{1}=\beta\int_{0}^{1}\pa{}\brz{z}(\alpha-\beta+t\beta)+\pa{}\brh{z}(\alpha-\beta+t\beta)dt\\\nonumber
&\le\abs{\beta}(\norm{\pa{}\brz{z}}_{L^{\infty}}+\norm{\pa{}\brh{z}}_{L^{\infty}}),
\end{align}
\begin{align}
\label{j2}
&J_{2}\le\frac{\abs{\beta}}{A^{\frac{1}{2}}(t)}(\norm{\pa{}\brz{z}}_{L^{\infty}}+\norm{\pa{}\brh{z}}_{L^{\infty}})
\end{align}
and
\begin{align}
\label{j3}
&J_{3}=c(\alpha-\beta)\beta\int_{0}^{1}\pa{2}z(\alpha-\beta+t\beta)dt\le\norm{c}_{L^{\infty}}\abs{\beta}\norm{z}_{\mathcal{C}^{2}}.
\end{align}
Computing $\frac{1}{\abs{\Delta z}^{2}}-\frac{1}{\beta^{2}\abs{\pa{}z(\alpha)}^{2}}$ and using (\ref{j1}), (\ref{j2}) and (\ref{j3}), we get
\begin{displaymath}
\abs{B_{2}^{1}}\le C\exp C\nor{z,h}^{2}\norm{\F}^{\frac{3}{2}}_{L^{\infty}}\norm{z}_{H^{2}}\norm{\varpi_{1}}_{L^{2}}.
\end{displaymath}
Since,
\begin{align*}
\Delta z_{t}=\beta^{2}\int_{0}^{1}\int_{0}^{1}\pa{2}z(\alpha-\beta s+\beta ts)(1-t)dsdt+\beta\pa{}z_{t}(\alpha)
\end{align*}
then
\begin{align*}
&B_{2}^{2}\le C\norm{\F}_{L^{\infty}}\norm{\pa{2}z_{t}}_{L^{2}}\norm{\varpi_{1}}_{L^{2}}+C\norm{\F}_{L^{\infty}}\norm{\pa{}z_{t}}_{L^{\infty}}\norm{H(\varpi_{1})}_{L^{2}}.
\end{align*}
Using (\ref{paz}) and
\begin{align}
\label{pa2z}
&\norm{\pa{2}z_{t}}_{L^{2}}\le\norm{\pa{2}\brz{z}+\pa{2}\brh{z}}_{L^{2}}+\norm{\pa{2}c}_{L^{2}}\norm{\pa{}z}_{L^{\infty}}\\\nonumber
&+\norm{\pa{}c}_{L^{\infty}}\norm{\pa{2}z}_{L^{2}}+\norm{c}_{L^{\infty}}\norm{\pa{3}z}_{L^{2}},
\end{align}
since
\begin{align*}
\pa{2}c(\alpha)=&-\frac{\pa{2}z(\alpha)}{A(t)}(\pa{}\brz{z}+\pa{}\brh{z})\\
&-\frac{\pa{}z(\alpha)}{A(t)}(\pa{2}\brz{z}+\pa{2}\brh{z}),
\end{align*} 
we get
\begin{align*}
B_{2}^{2}\le\exp C\nor{z,h}^{2}.
\end{align*}
Using the same proceeding, we have $B_{3}\le\exp C\nor{z,h}^{2}$.

On the other hand, we split $\partial_{t}\brh{z}=C_{1}+C_{2}+C_{3}$ where,
\begin{align*}
&C_{1}=\frac{1}{2\pi}PV\int_{\R}\frac{(\Delta zh)^{\bot}}{\abs{\Delta zh}^{2}}\partial_{t}\varpi_{2}(\alpha-\beta)d\alpha d\beta,\\
&C_{2}=\frac{1}{2\pi}PV\int_{\R}\frac{\partial_{t}^{\bot}z(\alpha)}{\abs{\Delta zh}^{2}}\varpi_{2}(\alpha-\beta)d\alpha d\beta,\\
&C_{3}=-\frac{1}{\pi}PV\int_{\R}\frac{(\Delta zh)^{\bot}\Delta zh\cdot\partial_{t}z(\alpha)}{\abs{\Delta z}^{4}}\varpi_{2}(\alpha-\beta)d\alpha d\beta.\\
\end{align*}
Thus we have
\begin{align*}
&C_{1}\le C\norm{d(z,h)}_{L^{\infty}}^{\frac{1}{2}}\norm{\partial_{t}\varpi_{2}}_{L^{2}},\\
&C_{2}\le C\norm{d(z,h)}_{L^{\infty}}\norm{\partial_{t}z}_{L^{\infty}}\norm{\varpi_{2}}_{L^{2}}\le\exp C\nor{z,h}^{2},\\
&C_{3}\le C\norm{d(z,h)}_{L^{\infty}}\norm{\pa{}z_{t}}_{L^{\infty}}\norm{\varpi_{2}}_{L^{2}}\le\exp C\nor{z,h}^{2}.
\end{align*}
We only need to know what happen with $\norm{\partial_{t}\varpi_{1}}_{L^{2}}$, $\norm{\partial_{t}\varpi_{2}}_{L^{2}}$ and $\norm{\varpi_{1}}_{\mathcal{C}^{\delta}}$.
Using the definitions of $\partial_{t}\varpi_{1}$ and $\partial_{t}\varpi_{2}$ we can see that
\begin{align*}
\varpi_{t}+M\OT(\varpi_{t})=-M\R\varpi-\left(\begin{matrix}
N\partial_{t}\pa{}z_{2}(\alpha)\\
0
\end{matrix}\right)
\end{align*}
where
\begin{displaymath}
\R=\left(\begin{matrix}
R_{1}&R_{2}\\
R_{3}&0
\end{matrix}\right)
\end{displaymath}
with
\begin{align*}
&R_{1}(\varpi_{1})=\frac{1}{\pi}PV\int_{\R}\partial_{t}(\frac{(z(\alpha)-z(\alpha-\beta))^{\bot}\cdot\pa{}z(\alpha)}{\abs{z(\alpha)-z(\alpha-\beta)}^{2}})\varpi_{1}(\alpha-\beta)d\beta,\\
&R_{2}(\varpi_{2})=\frac{1}{\pi}PV\int_{\R}\partial_{t}(\frac{(z(\alpha)-h(\alpha-\beta))^{\bot}\cdot\pa{}z(\alpha)}{\abs{z(\alpha)-h(\alpha-\beta)}^{2}})\varpi_{2}(\alpha-\beta)d\beta,\\
&R_{3}(\varpi_{1})=\frac{1}{\pi}PV\int_{\R}\partial_{t}(\frac{(h(\alpha)-z(\alpha-\beta))^{\bot}\cdot\pa{}z(\alpha)}{\abs{h(\alpha)-z(\alpha-\beta)}^{2}})\varpi_{1}(\alpha-\beta)d\beta.\\
\end{align*}
Then,
\begin{displaymath}
\norm{\varpi_{t}}_{H^{\frac{1}{2}}}\le\norm{(I+M\OT)^{-1}}_{H^{\frac{1}{2}}}(\norm{\R\varpi}_{H^{\frac{1}{2}}}+\norm{\pa{}z_{t}}_{H^{\frac{1}{2}}}).
\end{displaymath}
Therefore, it is clear that in order to control $\norm{\varpi_{t}}_{L^{2}}$ we only need to estimate $\norm{\R\varpi}_{H^{\frac{1}{2}}}$. To do that, let us estimate $\norm{\R\varpi}_{H^{1}}$:

Let separates $R_{1}(\varpi_{1})=S_{1}+S_{2}+S_{3}$ where
\begin{align*}
&S_{1}=\frac{1}{\pi}PV\int_{\R}\frac{(z_{t}(\alpha)-z_{t}(\alpha-\beta))^{\bot}\cdot\pa{}z(\alpha)}{\abs{z(\alpha)-z(\alpha-\beta)}^{2}}\varpi_{1}(\alpha-\beta)d\beta,\\
&S_{2}=\frac{1}{\pi}PV\int_{\R}\frac{(z(\alpha)-z(\alpha-\beta))^{\bot}\cdot\pa{}z_{t}(\alpha)}{\abs{z(\alpha)-z(\alpha-\beta)}^{2}}\varpi_{1}(\alpha-\beta)d\beta,\\
&S_{3}=-\frac{2}{\pi}PV\int_{\R}\frac{(\Delta z)^{\bot}\cdot\pa{}z(\alpha)\Delta z\cdot\Delta z_{t}}{\abs{z(\alpha)-z(\alpha-\beta)}^{4}}\varpi_{1}(\alpha-\beta)d\beta.\\
\end{align*}
We will estimate $\pa{}S_{1}$, the other terms $\pa{}S_{2}$ and $\pa{}S_{3}$ are estimated with the same procedure.
\begin{align*}
&\pa{}S_{1}=\frac{1}{\pi}PV\int_{\R}\frac{(\pa{}z_{t}(\alpha)-\pa{}z_{t}(\alpha-\beta))^{\bot}\cdot\pa{}z(\alpha)}{\abs{z(\alpha)-z(\alpha-\beta)}^{2}}\varpi_{1}(\alpha-\beta)d\beta\\
&+\frac{1}{\pi}PV\int_{\R}\frac{(z_{t}(\alpha)-z_{t}(\alpha-\beta))^{\bot}\cdot\pa{}z(\alpha)}{\abs{z(\alpha)-z(\alpha-\beta)}^{2}}\pa{}\varpi_{1}(\alpha-\beta)d\beta\\
&+\frac{1}{\pi}PV\int_{\R}\frac{(z_{t}(\alpha)-z_{t}(\alpha-\beta))^{\bot}\cdot\pa{2}z(\alpha)}{\abs{z(\alpha)-z(\alpha-\beta)}^{2}}\varpi_{1}(\alpha-\beta)d\beta\\
&-\frac{2}{\pi}PV\int_{\R}\frac{(\Delta z_{t})^{\bot}\cdot\pa{}z(\alpha)\Delta z\cdot\Delta\pa{}z}{\abs{z(\alpha)-z(\alpha-\beta)}^{4}}\varpi_{1}(\alpha-\beta)d\beta\\
&\equiv S_{1}^{1}+S_{1}^{2}+S_{1}^{3}+S_{1}^{4}.
\end{align*}
As we could see in the evolution of the arc-chord condition, using the definitions (\ref{j1}), (\ref{j2}) and (\ref{j3}), we can write $\Delta z_{t}=J_{1}+J_{2}+J_{3}$.

For $S_{1}^{1}$, we split
\begin{align*}
&S_{1}^{1}=\frac{1}{\pi}PV\int_{\R}\frac{(\pa{}z_{t}(\alpha)-\pa{}z_{t}(\alpha-\beta))^{\bot}\cdot\pa{}z(\alpha)}{\abs{z(\alpha)-z(\alpha-\beta)}^{2}}(\varpi_{1}(\alpha-\beta)-\varpi_{1}(\alpha))d\beta\\
&+\frac{1}{\pi}PV\int_{\R}\frac{(\pa{}z_{t}(\alpha)-\pa{}z_{t}(\alpha-\beta))^{\bot}\cdot\pa{}z(\alpha)}{\abs{z(\alpha)-z(\alpha-\beta)}^{2}}\varpi_{1}(\alpha)d\beta\\
&\equiv S_{1}^{11}+S_{1}^{12}.
\end{align*}
Since $\varpi_{1}(\alpha-\beta)-\varpi_{1}(\alpha)=\beta\int_{0}^{1}\pa{}\varpi_{1}(\alpha-\beta t)dt$ and $\pa{}z_{t}(\alpha)-\pa{}z_{t}(\alpha-\beta)=\beta\int_{0}^{1}\pa{2}z_{t}(\alpha+\beta-\beta t)dt$ and we have seen (\ref{pa2z}) then we have controlled $S_{1}^{11}$.

For $S_{1}^{12}$, computing 
\begin{displaymath}
B(\alpha,\beta)=\frac{1}{\abs{\Delta z}^{2}}-\frac{1}{\beta^{2}\abs{\pa{}z(\alpha)}^{2}}=\frac{\beta\int_{0}^{1}\int_{0}^{1}\pa{2}z(\psi)(t-1)dtds\cdot\int_{0}^{1}\pa{}z(\alpha)+\pa{}z(\phi)dt}{\abs{\pa{}z(\alpha)}^{2}\abs{\Delta z}^{2}}
\end{displaymath} we split
\begin{align*}
&S_{1}^{12}=\frac{1}{\pi}PV\int_{\R}(\beta\int_{0}^{1}\pa{2}z_{t}(\alpha-\beta+\beta t)^{\bot})dt\cdot\pa{}z(\alpha)B(\alpha,\beta)\varpi_{1}(\alpha)d\beta\\
&+\frac{1}{\pi}PV\int_{\R}\frac{(\pa{}z_{t}(\alpha)-\pa{}z_{t}(\alpha-\beta))^{\bot}\cdot\pa{}z(\alpha)}{\beta^{2}\abs{\pa{}z(\alpha)}^{2}}\varpi_{1}(\alpha)d\beta\\
&\le\exp C\nor{z,h}^{2}+S_{1}^{121}.
\end{align*}
Here, we have
\begin{align*}
&S_{1}^{121}=\Lambda(\pa{\bot}z_{t})\cdot\pa{}z(\alpha)\frac{\varpi_{1}(\alpha)}{\abs{\pa{}z(\alpha)}^{2}},
\end{align*}
then
\begin{align*}
\abs{S_{1}^{121}}\le\norm{\F}_{L^{\infty}}\norm{z}_{\mathcal{C}^{1}}\norm{\varpi_{1}}_{L^{\infty}}\abs{\Lambda(\pa{}z_{t})}.
\end{align*}
Thus,
\begin{align*}
\norm{S_{1}^{1}}_{L^{2}}\le\exp C\nor{z,h}^{2}\norm{\Lambda(\pa{}z_{t})}_{L^{2}}\le\exp C\nor{z,h}^{2}\norm{\pa{2}z_{t}}_{L^{2}}\le\exp C\nor{z,h}^{2}.
\end{align*}
For $S_{1}^{2}$ we do the same thing,
\begin{align*}
&S_{1}^{2}=\frac{1}{\pi}PV\int_{\R}\beta\int_{0}^{1}\pa{\bot}z_{t}(\alpha-\beta+\beta t)dt\cdot\pa{}z(\alpha)B(\alpha,\beta)\pa{}\varpi_{1}(\alpha-\beta)d\beta\\
&+\frac{1}{\pi}PV\int_{\R}\frac{\int_{0}^{1}\pa{\bot}z_{t}(\alpha-\beta+\beta t)dt\cdot\pa{}z(\alpha)}{\beta\abs{\pa{}z(\alpha)}^{2}}\pa{}\varpi_{1}(\alpha-\beta)d\beta\\
&\le\exp c\nor{z,h}^{2}\\
&+\frac{1}{\pi}PV\int_{\R}\frac{\int_{0}^{1}\int_{0}^{1}\pa{2}z_{t}(\alpha-\beta s+\beta ts)^{\bot}(1-t)dsdt\cdot\pa{}z(\alpha)}{\abs{\pa{}z(\alpha)}^{2}}\pa{}\varpi_{1}(\alpha-\beta)d\beta\\
&+\frac{\pa{\bot}z_{t}(\alpha)\cdot\pa{}z(\alpha)}{\abs{\pa{}z(\alpha)}^{2}}H(\pa{}\varpi_{1}).
\end{align*}
Therefore, using (\ref{paz}) and (\ref{pa2z})
\begin{align*}
\norm{S_{1}^{2}}_{L^{2}}\le\exp C\nor{z,h}^{2}.
\end{align*}
For $S_{1}^{3}$ exactly the same as in $S_{1}^{2}$. And for $S_{1}^{4}$, computing:
\begin{align*}
&C(\alpha,\beta)=\frac{1}{\abs{\Delta z}^{4}}-\frac{1}{\beta^{4}\abs{\pa{}z(\alpha)}^{4}}\\
&=\frac{\beta\int_{0}^{1}\int_{0}^{1}\pa{2}z(\psi)(t-1)dsdt\cdot\int_{0}^{1}\pa{}z(\alpha)+\pa{}z(\phi)dt\int_{0}^{1}\abs{\pa{}z(\alpha)}^{2}+\abs{\pa{}z(\phi)}^{2}dt}{\abs{\Delta z}^{4}\abs{\pa{}z(\alpha)}^{4}}
\end{align*}
where $\psi=\alpha-\beta+\beta t+\beta s-\beta ts$ and $\phi=\alpha-\beta+\beta t$. We have,
\begin{align*}
&S_{1}^{4}=-\frac{2}{\pi}PV\int_{\R}\beta^{3}\int_{0}^{1}\pa{}z_{t}(\phi)^{\bot}dt\cdot\pa{}z(\alpha)\int_{0}^{1}\pa{}z(\phi)dt\cdot\int_{0}^{1}\pa{2}z(\phi)dtC(\alpha,\beta)\varpi_{1}(\alpha-\beta)d\beta\\
&-\frac{2}{\pi}PV\int_{\R}\frac{\int_{0}^{1}\pa{}z_{t}(\phi)^{\bot}dt\cdot\pa{}z(\alpha)\int_{0}^{1}\pa{}z(\phi)dt\cdot\int_{0}^{1}\pa{2}z(\phi)dt}{\beta\abs{\pa{}z(\alpha)}^{4}}\varpi_{1}(\alpha-\beta)d\beta\\
&\le\exp C\nor{z,h}^{2}+S_{1}^{41}.
\end{align*}
It is easy to get
\begin{align*}
S_{1}^{41}\le\exp C\nor{z,h}^{2}-2\frac{\pa{\bot}z_{t}(\alpha)\cdot\pa{}z(\alpha)\pa{}z(\alpha)\cdot\pa{2}z(\alpha)}{\abs{\pa{}z(\alpha)}^{4}}H(\varpi_{1}).
\end{align*}
Then,
\begin{displaymath}
\norm{S_{1}^{4}}_{L^{2}}\le\exp C\nor{z,h}^{2}.
\end{displaymath}
Therefore, $\norm{\pa{}S_{1}}_{L^{2}}\le\exp C\nor{z,h}^{2}$.
 We have controlled $\norm{\pa{}S_{2}}_{L^{2}}+\norm{\pa{}S_{3}}_{L^{2}}$ in the same way.
Now let us estimate $\norm{\pa{}R_{2}}_{L^{2}}$. We have,
\begin{align*}
&R_{2}(\varpi_{2})=\frac{1}{\pi}PV\int_{\R}\frac{z^{\bot}_{t}(\alpha)\cdot\pa{}z(\alpha)}{\abs{z(\alpha)-h(\alpha-\beta)}^{2}}\varpi_{2}(\alpha-\beta)d\beta\\
&+\frac{1}{\pi}PV\int_{\R}\frac{(z(\alpha)-h(\alpha-\beta))^{\bot}\cdot\pa{}z_{t}(\alpha)}{\abs{z(\alpha)-h(\alpha-\beta)}^{2}}\varpi_{2}(\alpha-\beta)d\beta\\
&-\frac{2}{\pi}PV\int_{\R}\frac{(\Delta zh)^{\bot}\cdot\pa{}z(\alpha)\Delta zh\cdot z_{t}(\alpha)}{\abs{z(\alpha)-h(\alpha-\beta)}^{4}}\varpi_{2}(\alpha-\beta)d\beta\\
&\equiv S_{4}+S_{5}+S_{6}.
\end{align*}
Then,
\begin{align*}
&S_{4}\le C\norm{d(z,h)}_{L^{\infty}}\norm{z}_{\mathcal{C}^{1}}\norm{z_{t}}_{L^{2}}\norm{\varpi_{1}}_{L^{2}},\\
&S_{5}\le C\norm{d(z,h)}_{L^{\infty}}^{\frac{1}{2}}\norm{z}_{\mathcal{C}^{1}}\norm{\varpi_{1}}_{L^{2}},\\
&S_{6}\le C\norm{d(z,h)}^{\frac{1}{2}}_{L^{\infty}}\norm{z}_{\mathcal{C}^{1}}\norm{z_{t}}_{L^{2}}\norm{\varpi_{1}}_{L^{2}}.
\end{align*}
In conclusion,
\begin{displaymath}
\norm{\pa{}R_{2}(\varpi_{2})}_{L^{2}}\le\exp C\nor{z,h}^{2}.
\end{displaymath}
Moreover, $\pa{}R_{3}$ is like $\pa{}R_{2}$ changing $z$ with $h$, then $\norm{\pa{}R_{3}(\varpi_{1})}_{L^{2}}\le\exp C\nor{z,h}^{2}$. Thus, we have controlled $\norm{\partial_{t}\varpi_{1}}_{L^{2}}$ and $\norm{\partial_{t}\varpi_{2}}_{L^{2}}$.
Finally, in order to control $\norm{\partial_{t}\varpi_{1}}_{\mathcal{C}^{\delta}}$ we will use 
\begin{align*}
\norm{\partial_{t}\varpi_{1}}_{\mathcal{C}^{\delta}}\le C(\norm{T_{1}(\partial_{t}\varpi_{1})}_{\mathcal{C}^{\delta}}+\norm{T_{2}(\partial_{t}\varpi_{2})}_{\mathcal{C}^{\delta}}+\norm{R_{1}(\varpi_{1})}_{\mathcal{C}^{\delta}}+\norm{R_{2}(\varpi_{2})}_{\mathcal{C}^{\delta}}+\norm{\pa{}z_{t}}_{\mathcal{C}^{\delta}}).
\end{align*}
Using the Lemma \ref{tl2h1},
\begin{align*}
&\norm{T_{1}(\varpi_{1})}_{\mathcal{C}^{\delta}}\le\norm{T_{1}(\varpi_{1})}_{H^{1}}\le C\norm{\F}^{2}_{L^{\infty}}\norm{z}^{4}_{\mathcal{C}^{2,\delta}}\norm{\partial_{t}\varpi_{1}}_{L^{2}},\\
&\norm{T_{2}(\varpi_{2})}_{\mathcal{C}^{\delta}}\le\norm{T_{2}(\varpi_{2})}_{H^{1}}\le C\norm{d(z,h)}^{2}_{L^{\infty}}\norm{h}^{4}_{\mathcal{C}^{2,\delta}}\norm{\partial_{t}\varpi_{2}}_{L^{2}},\\
\end{align*}
for $\delta\le\frac{1}{2}$.
We have already seen $\norm{R_{1}(\varpi_{1})}_{H^{1}}+\norm{R_{2}(\varpi_{2})}_{H^{1}}\le\exp C\nor{z,h}^{2}$ then
\begin{displaymath}
\norm{R_{1}(\varpi_{1})}_{\mathcal{C}^{\delta}}+\norm{R_{2}(\varpi_{2})}_{\mathcal{C}^{\delta}}\le\exp C\nor{z,h}^{2}.
\end{displaymath}
Finally let us observe that $\norm{\pa{}z_{t}}_{\mathcal{C}^{\delta}}\le\norm{z_{t}}_{H^{2}}$ which is controlled by $\norm{\pa{2}z_{t}}_{L^{2}}$.
The upper bound 
\begin{displaymath}
\abs{\sigma_{t}(\alpha,t)}\le\exp C\nor{z,h}^{2}
\end{displaymath}
gives us
\begin{displaymath}
m'(t)\ge -\exp C\nor{z,h}^{2}
\end{displaymath}
for almost every $t$. And a futher integration yields
\begin{displaymath}
m(t)\ge m(0)-\int_{0}^{t}\exp C\nor{z,h}^{2}(s)ds.
\end{displaymath}
\end{proof}

\section{Conclusion of the Local-existence}
\label{regu}
This step is classical, then we only sketch this procedure.
We regularize the problem as follows:
\begin{align*}
&z^{\varepsilon}_{t}(\alpha,t)=\brzep{z^{\varepsilon}}(\alpha,t)+\brhep{z^{\varepsilon}}(\alpha,t)+c^{\varepsilon}(\alpha,t)\partial_{\alpha}z^{\varepsilon}(\alpha,t)\\
&z^{\varepsilon}(\alpha,0)=\phi_{\varepsilon}*z_{0}(\alpha)\\
\end{align*}
where
\begin{align*}
&c^{\varepsilon}(\alpha,t)=\frac{\alpha+\pi}{2\pi A^{\varepsilon}(t)}\int_{\T}\partial_{\alpha}z^{\varepsilon}(\beta,t)\cdot\partial_{\alpha}(\brzep{z^{\varepsilon}}+\brhep{z^{\varepsilon}})d\beta\\
&-\int_{-\pi}^{\alpha}\frac{\partial_{\alpha}z^{\varepsilon}(\beta,t)}{A^{\varepsilon}(t)}\cdot\partial_{\alpha}(\brzep{z^{\varepsilon}}+\brhep{z^{\varepsilon}})d\beta,\\
&\varpi_{1}^{\varepsilon}(\alpha,t)=-2\frac{\mu^{2}-\mu^{1}}{\mu^{2}+\mu^{1}}\phi_{\varepsilon}*\phi_{\varepsilon}*(\brzep{z^{\varepsilon}}+\brhep{z^{\varepsilon}})\cdot\partial_{\alpha}z^{\varepsilon}(\alpha,t)\\
&-2\kappa^{1}\frac{\rho^{2}-\rho^{1}}{\mu^{2}+\mu^{1}}g\phi_{\varepsilon}*\phi_{\varepsilon}*\partial_{\alpha}z^{\varepsilon}_{2}(\alpha,t)\\
&\varpi_{2}(\alpha,t)=-2\frac{\kappa^{2}-\kappa^{1}}{\kappa^{2}+\kappa^{1}}\phi_{\varepsilon}*\phi_{\varepsilon}*(\brzep{h^{\varepsilon}}+\brhep{h^{\varepsilon}})\cdot\partial_{\alpha}h^{\varepsilon}(\alpha)
\end{align*}
for $\phi\in\mathcal{C}^{\infty}_{c}$, $\phi(\alpha)\ge 0$, $\phi(-\alpha)=\phi(\alpha)$, $\int_{\R}\phi(\alpha)d\alpha=1$ and $\phi_{\varepsilon}(\alpha)=\phi(\frac{\alpha}{\varepsilon})/\varepsilon$.
Using the same techniques that in the above secction, we can prove that:
\begin{align*}
&\frac{d}{dt}\norm{z^{\varepsilon}}^{2}_{H^{k}}(t)\le\estheps{k}\\
&-\frac{\kappa^{1}}{2\pi(\mu^{2}+\mu^{1})}\int_{\T}\frac{\sigma^{\varepsilon}(\alpha,t)}{A^{\varepsilon}(t)}\phi_{\varepsilon}*\pa{k}z^{\varepsilon}\cdot\Lambda(\phi_{\varepsilon}*\pa{k}z^{\varepsilon})d\alpha
\end{align*}
where
\begin{align*}
\sigma^{\varepsilon}(\alpha,t)=\frac{\mu^{2}-\mu^{1}}{\kappa^{1}}(\brzep{z^{\varepsilon}}+\brhep{z^{\varepsilon}})\cdot\pa{\bot}z^{\varepsilon}(\alpha,t)+g(\rho^{2}-\rho^{1})\pa{}z^{\varepsilon}(\alpha,t)
\end{align*}

The next step is to integrate during a time $T$ independent of $\varepsilon$. Let us observe that if $\phi_{\varepsilon}*z_{0}(\alpha)\in H^{k}$, then we have the solution $z^{\varepsilon}\in\mathcal{C}^{1}([0,T^{\varepsilon}],H^{k})$. If $\sigma(\alpha,0)>0$, there exists $T^{\varepsilon}$ dependent of $\varepsilon$ where $\sigma^{\varepsilon}(\alpha,t)>0$. Then for $t\le T^{\varepsilon}$ with our a priori estimates and the fact that
\begin{displaymath}
f(\alpha)\Lambda f(\alpha)-\frac{1}{2}\Lambda(f^{2})(\alpha)\ge 0
\end{displaymath}
we get
\begin{align*}
&\frac{d}{dt}\norm{z^{\varepsilon}}^{2}_{H^{k}}(t)\le\estheps{k}\\
&-\frac{\kappa^{1}}{4\pi(\mu^{2}+\mu^{1})A^{\varepsilon}(t)}\int_{\T}\sigma^{\varepsilon}(\alpha,t)\Lambda(\abs{\phi_{\varepsilon}*\pa{k}z^{\varepsilon}}^{2})(\alpha)d\alpha.
\end{align*}
Since 
\begin{align*}
&\norm{\Lambda\sigma^{\varepsilon}}_{L^{\infty}}\le C\norm{\sigma^{\varepsilon}}_{H^{2}}\le C(\norm{\brzep{z^{\varepsilon}}+\brhep{z^{\varepsilon}}}_{L^{2}}\\
&+\norm{\pa{2}\brzep{z^{\varepsilon}}+\pa{2}\brhep{z^{\varepsilon}}}_{L^{2}}+1)\norm{z^{\varepsilon}}_{H^{3}},
\end{align*}
then
\begin{displaymath}
\frac{d}{dt}\norm{z^{\varepsilon}}^{2}_{H^{k}}(t)\le\estheps{k}.
\end{displaymath}
In the same way as in \ref{evoldi} we can check that
\begin{align*}
\frac{d}{dt}\norm{\mathcal{F}(z^{\varepsilon})}_{L^{\infty}}^{2}\le\estheps{k}.
\end{align*}
And we have
\begin{align*}
\frac{d}{dt}\norm{d(z^{\varepsilon},h^{\varepsilon})}_{L^{\infty}}^{2}\le\estheps{k}.
\end{align*}
Therefore,
\begin{align*}
\frac{d}{dt}(\norm{\mathcal{F}(z^{\varepsilon})}_{L^{\infty}}^{2}+\norm{d(z^{\varepsilon},h^{\varepsilon})}_{L^{\infty}}^{2}+\norm{z^{\varepsilon}}^{2}_{H^{k}})\le\estheps{k}.
\end{align*}
Integrating,
\begin{align*}
&\norm{\mathcal{F}(z^{\varepsilon})}_{L^{\infty}}^{2}+\norm{d(z^{\varepsilon},h^{\varepsilon})}_{L^{\infty}}^{2}+\norm{z^{\varepsilon}}^{2}_{H^{k}}\\
&\le -\frac{1}{C}\ln(-t+\exp(-C(\norm{\mathcal{F}(z_{0})}_{L^{\infty}}^{2}+\norm{d(z_{0},h)}_{L^{\infty}}^{2}+\norm{z_{0}}^{2}_{H^{k}}))).
\end{align*}
Since $m^{\varepsilon}(t)\ge m(0)-\int_{0}^{t}\estheps{k}(s)ds$, where $m^{\varepsilon}(t)=\min_{\alpha\in\T}\sigma^{\varepsilon}(\alpha,t)$ for $t\le T^{\varepsilon}$, using the above estimations
\begin{align*}
&m^{\epsilon}(t)\ge m(0)+C(\norm{\mathcal{F}(z_{0})}_{L^{\infty}}^{2}+\norm{d(z_{0},h)}_{L^{\infty}}^{2}+\norm{z_{0}}^{2}_{H^{k}})\\
&+\ln(-t+\exp(-C(\norm{\mathcal{F}(z_{0})}_{L^{\infty}}^{2}+\norm{d(z_{0},h)}_{L^{\infty}}^{2}+\norm{z_{0}}^{2}_{H^{k}})))
\end{align*}
for $t\le T^{\varepsilon}$.
Now if we $\varepsilon\to 0$ we have $T^{\varepsilon}\nrightarrow 0$. This is because if we take $T=\min(T^{1},T^{2})$ where $T^{1}$ satisfies,
\begin{align*}
&m(0)+C(\norm{\mathcal{F}(z_{0})}_{L^{\infty}}^{2}+\norm{d(z_{0},h)}_{L^{\infty}}^{2}+\norm{z_{0}}^{2}_{H^{k}})\\
&+\ln(-T^{1}+\exp(-C(\norm{\mathcal{F}(z_{0})}_{L^{\infty}}^{2}+\norm{d(z_{0},h)}_{L^{\infty}}^{2}+\norm{z_{0}}^{2}_{H^{k}})))>0
\end{align*}
and $T_{2}$ satisfies,
\begin{align*}
-\frac{1}{C}\ln(-T^{2}+\exp(-C(\norm{\mathcal{F}(z_{0})}_{L^{\infty}}^{2}+\norm{d(z_{0},h)}_{L^{\infty}}^{2}+\norm{z_{0}}^{2}_{H^{k}})))<\infty.
\end{align*}
For $t\le T$ we have $m^{\varepsilon}(t)>0$ and
\begin{align*}
&\norm{\mathcal{F}(z^{\varepsilon})}_{L^{\infty}}^{2}+\norm{d(z^{\varepsilon},h)}_{L^{\infty}}^{2}+\norm{z^{\varepsilon}}^{2}_{H^{k}}\\
&\le-\frac{1}{C}\ln(-T^{2}+\exp(-C(\norm{\mathcal{F}(z_{0})}_{L^{\infty}}^{2}+\norm{d(z_{0},h)}_{L^{\infty}}^{2}+\norm{z_{0}}^{2}_{H^{k}})))<\infty 
\end{align*}
and $T$ only depend on $z_{0}$. Then, we have local existence when $\varepsilon\to 0$.

\bibliography{mibiblioteca}
\bibliographystyle{plain}
\end{document}